\documentclass[10pt]{article}
\usepackage{amsfonts}
\usepackage{amssymb}
\usepackage{amsmath}
\usepackage{amsthm}
\usepackage{amscd}
\usepackage[pdftex]{graphicx}
\usepackage{color}
\usepackage{enumerate}
\usepackage{subcaption}
\usepackage[titletoc,title]{appendix}
\usepackage{mathtools}
\usepackage[normalem]{ulem}
\usepackage{xparse}
\usepackage{environ}
\usepackage{stackengine,wasysym}
\usepackage{todonotes}
\usepackage{xcolor,colortbl}
\usepackage{makecell}
\usepackage{multirow}
\usepackage{hyperref}
\hypersetup{
  colorlinks   = true, %Colours links instead of using boxes
  urlcolor     = blue, %Colour for external hyperlinks
  linkcolor    = black, %Colour of internal links
  citecolor   = black %Colour of citations
}

\makeatletter
\renewcommand{\todo}[2][]{\tikzexternaldisable\@todo[#1]{#2}\tikzexternalenable}
\makeatother

\usepackage{marginnote}

\usepackage{xcolor,colortbl}

\usepackage{tikz}
\usetikzlibrary{calc}
\usetikzlibrary {shapes.geometric}
\usetikzlibrary{external}
\usetikzlibrary{cd}
\usepackage[framemethod=TikZ]{mdframed}
\usetikzlibrary{decorations.pathreplacing,calligraphy}
\mdfdefinestyle{MyFrame}{%
	linecolor=black,
	outerlinewidth=0.5pt,
	roundcorner=0pt,
	innertopmargin=\baselineskip,
	innerbottommargin=\baselineskip,
	innerrightmargin=20pt,
	innerleftmargin=20pt,
	backgroundcolor=gray!14!white}

\graphicspath{{./}{figures/}{Figures/}{../Actual/Figures/}{../Actual/}}

\newtheorem{definition}{Definition}[section]
\newtheorem{remark}{Remark}[section]
\newtheorem{theorem}{Theorem}[section]
\newtheorem{lemma}[theorem]{Lemma}
\newtheorem{proposition}[theorem]{Proposition}
\newtheorem{corollary}[theorem]{Corollary}

\newtheorem{assumption}{Assumption}
\newtheorem*{assumption*}{Assumption}

\newcommand{\supp}{\mathrm{supp}}
\newcommand{\csupp}{\overline{\mathrm{supp}}}

\definecolor{forestgreen}{cmyk}{0.76,0,0.76,0.45}

\definecolor{lightgray}{cmyk}{0.2,0.2,0.2,0.2}

\definecolor{burntorange}{rgb}{0.74902,0.341176,0}
\def\Og{\color{burntorange}}
\definecolor{myBlue}{rgb} {0,    0.4470,    0.7410}
\let\hat\widehat

\newcommand{\RR}{\mathbb{R}}

\newcommand{\mbf}[1]{\mathbf{#1}}
\newcommand{\vspan}[1]{\left\langle #1 \right\rangle}
\newcommand{\Bezier}{B\'ezier}

\NewDocumentCommand{\inpr}{ O{\domain} m m }{\left(#2, #3\right)_{#1}}

\newcommand{\ndim}{n}
\newcommand{\nref}{L}
\newcommand{\ndofSym}{m}
\newcommand{\pdegSym}{p}

\newcommand{\ndof}[2]{\ndofSym_{(#1,#2)}}
\newcommand{\pdeg}[2]{\pdegSym_{(#1,#2)}}
\newcommand{\xBoxDim}{4cm}
\newcommand{\yBoxDim}{4cm}
\newcommand*{\thistextwidth}{0.24}

\newcommand{\kntSym}{\xi}
\newcommand{\knt}[1]{\kntSym_{#1}}

\newcommand{\kntsSym}{\Xi}
\newcommand{\knts}[1]{\kntsSym_{#1}}

\newcommand{\lknts}[3]{\kntsSym[#1]^{#2}_{#3}}
\newcommand{\xknts}[3]{\hat{\kntsSym}[#1]^{#2}_{#3}}

\newcommand{\tpbSym}{\mathcal{B}}
\newcommand{\tpbsSym}{\mathcal{X}}

\newcommand{\tpb}[2]{\tpbSym_{#1}^{#2}}
\newcommand{\tpbs}[2]{\tpbsSym_{#1}^{#2}}

\newcommand{\tpcmplx}[1]{\tpbsSym_{#1}}

\newcommand{\xtpbSym}{\mathbb{X}}

\newcommand{\xtpb}[2]{\xtpbSym_{#1}^{#2}}
\newcommand{\xgb}{\chi}
\newcommand{\itpb}[2]{\mathbb{I}_{#1}^{#2}}

\newcommand{\zerob}{\alpha}
\newcommand{\zerof}{a}

\newcommand{\threeb}{\zeta}
\newcommand{\threef}{z}
\newcommand{\gb}{\phi}
\newcommand{\gf}{f}

\newcommand{\unibmeshSym}{M}

\newcommand{\grSym}{\mu}

\newcommand{\greville}[2]{\grSym(\lknts{#1}{0}{#2})}
\newcommand{\grevilleE}[3]{\grSym(\lknts{#1}{#3}{#2})}
\newcommand{\unigmeshSym}{G}
\newcommand{\gmap}[2]{g_{#1}^{#2}}
\newcommand{\gmesh}[1]{\mathcal{\unigmeshSym}_{#1}}

\newcommand{\hbSym}{\mathcal{H}}
\newcommand{\hbsSym}{\mathcal{W}}
\newcommand{\hb}[2]{\hbSym_{#1}^{#2}}
\newcommand{\hbs}[2]{\hbsSym_{#1}^{#2}}

\newcommand{\xder}[2]{d_{#1}^{#2}}

\newcommand{\splinelevelinclusion}[1]{\iota_{#1}}

\newcommand{\cocomplex}[2]{\left(#1,#2\right)}

\newcommand{\image}{\mathrm{Im}}
\newcommand{\kernel}{\mathrm{Ker}}

\newcommand{\derivSym}{D}
\newcommand{\deriv}[1]{\derivSym_{#1}}

\newcommand{\textknotvector}{knot vector}
\newcommand{\textextendeduniknotdomain}{extended univariate knot domain}
\newcommand{\textextendedknotdomain}{extended knot domain}
\newcommand{\textextendedknotdomains}{extended knot domains}

\makeatletter
\ifcase \@ptsize \relax% 10pt
  \newcommand{\miniscule}{\@setfontsize\miniscule{3}{4}}% \tiny: 5/6
\or% 11pt
  \newcommand{\miniscule}{\@setfontsize\miniscule{4}{5}}% \tiny: 6/7
\or% 12pt
  \newcommand{\miniscule}{\@setfontsize\miniscule{4}{5}}% \tiny: 6/7
\fi
\makeatother

\makeatletter
    \def\head#1{\expandafter\@head#1,\@eol}
    \def\@head#1,#2\@eol{#1}
    \def\tail#1{\expandafter\@tail#1,\@eol}
    \def\@tail#1,#2\@eol{\ifx\@eol#2\@eol\relax\else\@@tail#2\@eol\fi}
    \def\@@tail#1,\@eol{#1}
    \def\last#1{\expandafter\@last#1,\@eol}
    \def\@last#1,#2\@eol{\ifx\@eol#2\@eol#1\else\@last#2\@eol\fi}
\makeatother

\definecolor{myCPlot1}{rgb} {0,    0.4470,    0.7410}
\definecolor{myCPlot2}{rgb} {0.8500,   0.3250,    0.0980}
\definecolor{myCPlot3}{rgb} {0.9290,    0.6940,    0.1250}
\definecolor{myCPlot4}{rgb} {0.4940,    0.1840,    0.5560}
\definecolor{myCPlot5}{rgb} {0.4660,    0.6740,    0.1880}
\definecolor{myCPlot6}{rgb} {0.3010,    0.7450,    0.9330}
\definecolor{myCPlot7}{rgb} {0.6350,    0.0780,    0.1840}

\def\mygrid[#1,#2,#3,#4,#5,#6,#7]{ %[coordinate_bottom_left, coordinate_top_right, #elements_x, #elements_y, line_thickness, label_id,face color]
	\coordinate (bl) at (#1);
	\coordinate (tr) at (#2);
	\coordinate (tl) at (#1 |- #2);
	\coordinate (br) at (#1 -| #2);
	\fill[#7] (bl) -- (br) -- (tr) -- (tl) -- cycle;
	\foreach \i in {0,...,#3}{
		\pgfmathsetmacro{\ii}{\i/#3}
		\coordinate (a) at ($(bl)!\ii!(br)$);
		\coordinate (b) at ($(tl)!\ii!(tr)$);
		\draw[line width=#5] (a) -- (b);
	}
	\foreach \j in {0,...,#4}{
		\pgfmathsetmacro{\jj}{\j/#4}
		\coordinate (a) at ($(bl)!\jj!(tl)$);
		\coordinate (b) at ($(br)!\jj!(tr)$);
		\draw[line width=#5] (a) -- (b);
	}
	\foreach \i in {0,...,#3}{
		\pgfmathsetmacro{\ii}{\i/#3}
		\coordinate (a) at ($(bl)!\ii!(br)$);
		\coordinate (b) at ($(tl)!\ii!(tr)$);
		\foreach \j in {0,...,#4}{
			\pgfmathsetmacro{\jj}{\j/#4}
			\coordinate (#6-\i-\j) at ($(a)!\jj!(b)$);
		}
	}
}

\newcommand{\mymanualgrid}[5]{ %[{_list_of_x_lines}, #{_list_of_y_lines}, line_thickness, line color, face color]

	\pgfmathsetmacro{\imin}{\head{#1}}
	
	\pgfmathsetmacro{\imax}{\last{#1}}

	\pgfmathsetmacro{\jmin}{\head{#2}}
	
	\pgfmathsetmacro{\jmax}{\last{#2}}

	\coordinate (bl) at (\imin,\jmin);
	\coordinate (tr) at (\imax,\jmax);
	\coordinate (tl) at (\imin,\jmax);
	\coordinate (br) at (\imax,\jmin);

	\fill[#5] (bl) -- (br) -- (tr) -- (tl) -- cycle;

	\foreach \i in {#1}{
		\draw[line width=#3,#4] (\i,\jmin) -- (\i,\jmax);
	}
	\foreach \j in {#2}{
		\draw[line width=#3,#4] (\imin,\j) -- (\imax,\j);
	}

}

\newcommand{\plotGrevillePoints}[7]{ %[{_list_of_x_lines}, #{_list_of_y_lines}, {_list_of_xidx_}, {_list_of_yidx_}, node color, node size, triangle size]
	\foreach \i in {#3}{
		\foreach \j in {#4}{
                    	\pgfmathparse{{#1}[\i]}
                    	\pgfmathsetmacro{\centralx}{\pgfmathresult}
                    	
                    	\pgfmathparse{\i-1}
                    	\pgfmathparse{{#1}[\pgfmathresult]}
                    	\pgfmathsetmacro{\lowerboundx}{\pgfmathresult}
                    	
                    	\pgfmathparse{\i+1}
                    	\pgfmathparse{{#1}[\pgfmathresult]}
                    	\pgfmathsetmacro{\upperboundx}{\pgfmathresult}
                    
                    	\pgfmathparse{(\centralx+\lowerboundx)/2}
                    	\pgfmathsetmacro{\leftx}{\pgfmathresult}
                    
                    	\pgfmathparse{(\centralx+\upperboundx)/2}
                    	\pgfmathsetmacro{\rightx}{\pgfmathresult}
                    	
                    	\pgfmathparse{{#2}[\j]}
                    	\pgfmathsetmacro{\centraly}{\pgfmathresult}
                    	
                    	\pgfmathparse{\j-1}
                    	\pgfmathparse{{#2}[\pgfmathresult]}
                    	\pgfmathsetmacro{\lowerboundy}{\pgfmathresult}
                    
                    	\pgfmathparse{\j+1}
                    	\pgfmathparse{{#2}[\pgfmathresult]}
                    	\pgfmathsetmacro{\upperboundy}{\pgfmathresult}
                    
                    	\pgfmathparse{(\centraly+\lowerboundy)/2}
                    	\pgfmathsetmacro{\bottomy}{\pgfmathresult}
                    
                    	\pgfmathparse{(\centraly+\upperboundy)/2}
                    	\pgfmathsetmacro{\topy}{\pgfmathresult}

                    	\node at (\centralx,\centraly) [circle,fill=#5,minimum size=#6]{};
                    	\node at (\leftx,\bottomy) [rectangle,fill=#5,minimum size=#6]{};
                    	\node at (\leftx,\topy) [rectangle,fill=#5,minimum size=#6]{};
                    	\node at (\rightx,\bottomy) [rectangle,fill=#5,minimum size=#6]{};
                    	\node at (\rightx,\topy) [rectangle,fill=#5,minimum size=#6]{};
                    	\node at (\centralx,\topy) [regular polygon, regular polygon sides=3,fill=#5,minimum size=#7]{};
                    	\node at (\centralx,\bottomy) [regular polygon, regular polygon sides=3,fill=#5,minimum size=#7]{};
                    	\node at (\leftx,\centraly) [regular polygon, regular polygon sides=3,fill=#5,minimum size=#7,rotate=-90]{};
                    	\node at (\rightx,\centraly) [regular polygon, regular polygon sides=3,fill=#5,minimum size=#7,rotate=-90]{};
		}
	}
}

\newtheorem*{theorem*}{Theorem}

\begin{document}
\title{Locally-verifiable sufficient conditions for exactness of the hierarchical B-spline discrete de Rham complex in $\RR^\ndim$}
\author{
	Kendrick Shepherd\thanks{ Civil \& Construction Engineering, Brigham Young University, Provo UT 84602, USA \newline \textit{E-mail address: } \texttt{kendrick\_shepherd@byu.edu}}
  \and
	Deepesh Toshniwal\thanks{Delft Institute of Applied Mathematics, Delft University of Technology, Delft, The Netherlands \newline \textit{E-mail address: } \texttt{d.toshniwal@tudelft.nl}}
}
\maketitle

\begin{abstract}
	Given a domain $\Omega \subset \RR^n$, the de Rham complex of differential forms arises naturally in the study of problems in electromagnetism and fluid mechanics defined on $\Omega$, and its discretization helps build stable numerical methods for such problems.
	For constructing such stable methods,  one critical requirement is ensuring that the discrete subcomplex is cohomologically equivalent to the continuous complex.
	When $\Omega$ is a hypercube, we thus require that the discrete subcomplex be exact.
	Focusing on such $\Omega$, we theoretically analyze the discrete de Rham complex built from hierarchical B-spline differential forms, i.e., the discrete differential forms are smooth splines and support adaptive refinements -- these properties are key to enabling accurate and efficient numerical simulations.
	We provide locally-verifiable sufficient conditions that ensure that the discrete spline complex is exact.
	Numerical tests are presented to support the theoretical results, and the examples discussed include complexes that satisfy our prescribed conditions as well as those that violate them.
\end{abstract}

\section{Introduction}

While many partial differential equations (PDEs) may be couched as minimization problems, a large swath of them are more naturally described as saddle-point problems.
Numerical solutions of these PDEs require special care since the discrete problems is not guaranteed to be well-posed for all choices of finite dimensional spaces even if the continuous problem is.
For instance, it may be necessary to verify (for each combination of finite dimensional spaces) the Ladyzhenskaya--Babuska--Brezzi condition for problems of interest such as electromagnetism and fluid flows.
On the other hand, the language of { exterior calculus} provides an abstract framework that can be used to uniformly discuss a large class of PDEs as well as their discretizations.
This framework has yielded an elegant approach, dubbed { \emph{finite element exterior calculus} }\cite{AFW06,AFW-2}, for developing well-posed discrete formulations.
In this document, and for problems posed on $n$-dimensional domains in $\RR^n$, $n \geq 1$, we provide new theoretical results that can help build such well-posed discretizations using finite dimensional spaces of locally-refinable spline functions.

The use of spline functions for numerically solving PDEs has been popularized by the emergence of \emph{isogeometric analysis} \cite{Hughes:2005}.
A generalization of the classical { finite element method} \cite{Hughes_book}, the { isogeometric analysis} philosophy relies on the use of spline functions \cite{DeBoor,Schumaker:2007} for describing both the domain on which the problem is posed as well as the discrete solution.
One objective of this approach is to greatly simplify the application of numerical methods to geometries of engineering interest, which are themselves designed within { computer-aided design} software using spline functions \cite{Farin:2002}.
The last two decades have seen this approach applied successfully to challenging problems such as full scale wind turbine simulations \cite{bazilevs2012isogeometric} and performance assessment of cardiac implants \cite{Kamensky17}.
Moreover, recent theoretical developments \cite{Bressan:2019,Sande:2019} have also given credence to the large amounts of numerical evidence that suggested that smooth
splines demonstrate better approximation behaviour per degree of freedom than less smooth or classical $C^0$ and $C^{-1}$ finite element spaces \cite{Evans_Bazilevs_Babuska_Hughes}.

Therefore, in this document we focus on methods that extend { finite element exterior calculus} by utilizing smooth splines in lieu of the classical $C^0$ and $C^{-1}$ finite element spaces.
The existing approaches \cite{BRSV11,Buffa_Sangalli_Vazquez,EvHu12} in this line of research have predominantly focused on PDEs that arise as Laplacians of the de Rham cochain complex of ($L^2$ completions of) smooth differential forms.
The said approaches either focus on identifying spline spaces that can be used to stably discretize such PDEs \cite{BRSV11,BSV14,Evans19,Johannessen15,toshniwal2021isogeometric}, or they use those spline spaces in applications of interest such as magnetohydrodynamics \cite{perse2021geometric} where exact satisfaction of physical conservation laws in the discrete setting is desirable.

When identifying stable discretizations for a PDE arising from some cochain complex, one of the main considerations in (spline-based) { finite element exterior calculus} is the identification of discrete spaces that together form a cohomologically-equivalent subcomplex of the continuous one.\footnote{The other main considerations (e.g., commuting projections) are outside the scope of the present work.}
Evans et al. \cite{Evans19} were the first to study this for locally refinable spline functions called hierarchical B-splines \cite{Kraft}.
Their main result identifies sufficient conditions for ensuring that the hierarchical B-spline spaces on a rectangular $\Omega \subset \RR^2$ form an exact subcomplex of the de Rham complex of $L^2$ differential forms on $\Omega$.
In this manuscript, we generalize and extend this to subdomains of $\RR^n$ for any $n \geq 1$.
The main contribution of this manuscript is the identification of locally-verifiable sufficient conditions that guarantee the exactness of the complex of hierarchical B-spline differential forms defined on a box (a.k.a.~a hypercube) $\Omega \subset \RR^\ndim$.

We describe our main result here (albeit a bit imprecisely) so that a reader familiar with hierarchical B-splines \cite{Kraft} and the complex of B-spline differential forms \cite{BRSV11} can get a flavor of our main result.
\begin{mdframed}[style=MyFrame]
	\begin{theorem*}
	Consider an $n$-dimensional box $\Omega \subset \RR^\ndim$ and an associated domain hierarchy 
	\begin{equation*}
		\Omega =: \Omega_0 \supseteq \Omega_1 \supseteq \cdots \supseteq \Omega_\nref \supseteq \Omega_{\nref+1} := \emptyset\;.
	\end{equation*}
	On each $\Omega_{\ell}$, define a complex of tensor-product B-spline differential forms as in \cite{BRSV11} and, assuming nestedness, use Kraft's selection mechanism\footnotemark \cite{Kraft} to build a complex of hierarchical B-spline differential forms \cite{Evans19}.
	Moreover, assume\footnotemark~the following for any $\ell < \nref$.
	\begin{itemize}
		\item $\Omega_{\ell+1}$ is the union of supports of level $\ell$ B-spline $0$-forms.
		\item  If two level $\ell$ B-spline $0$-forms $\alpha_i$ and $\alpha_j$ are supported on $\Omega_{\ell+1}$, and if the intersections of their supports has a ``minimal size'', then there exists a ``shortest chain'' of level $\ell$ B-spline $0$-forms from $\alpha_i$ to $\alpha_j$, with each B-spline in the chain supported on $\Omega_{\ell+1}$.
	\end{itemize}	
	Then, the complex of hierarchical B-spline differential forms is exact.
	\end{theorem*}
\end{mdframed}
\addtocounter{footnote}{-1}
\footnotetext{Here, Kraft's selection mechanism is the selection mechanism of Definition \ref{def:hbs}.}
\addtocounter{footnote}{1}
\footnotetext{See Definitions \ref{def:n-1-intersection}, \ref{def:chain}, \ref{def:shortest_path} and Assumption \ref{assume:shortest_path} for the precise statement. Proof of this result is given in Theorem \ref{thrm:exactness}.}

In the above, by a ``chain'' of B-splines from $\alpha_i$ to $\alpha_j$, we mean a sequence of B-splines -- with the first and the last B-splines in the sequence being $\alpha_i$ and $\alpha_j$ -- such that each B-spline in the sequence is obtained from the preceding one by a unit translation of its support in index space \cite{Hughes:2005}.
By { a} ``shortest chain'', we mean { a} sequence with the smallest number of B-splines.
Note that this assumption is locally-verifiable, i.e., it can be checked in a local manner for any given hierarchical B-spline mesh.

We start this manuscript by recalling the basics of the de Rham complex of differential forms (Section \ref{sec:deRham}) and the construction of the hierarchical B-spline complex (Section \ref{sec:notation}).
Section \ref{sec:main-result} contains our main result, which is derived by using the notion of Mayer--Vietoris sequences \cite{Hatcher}.
We discuss implementational aspects and present numerical tests that complement our theoretical results in Section \ref{sec:numerics} before concluding in Section \ref{sec:conclusions}.	
	
\section{The de Rham Complex of Differential Forms}\label{sec:deRham}
Before discussing exactness of the discrete (spline) complex of interest for numerical analysis, we first briefly introduce the continuous de Rham complex that we aim to approximate using these splines.
For the sake of clarity and brevity, we do not attempt to fully characterize all of the mathematical objects with which we work.
Instead, we emphasize only those properties that will be necessary for understanding the following mathematical presentation.
A reader more interested in a full picture may want to read  \cite{Arnold:2018} or \cite{Spivak:1995} for a more thorough introduction.

\subsection{Hilbert cochain complexes}\label{sec:prelims-complexes}
Let $V$ denote a sequence of Hilbert spaces $\left\{ V^{j} \right\}_{j = 0}^n$, and let $d$ denote a sequence $\left\{ d^{j} \right\}_{j = 0}^{n-1}$ of connecting closed, linear maps, $d^{j} : V^{j} \rightarrow V^{j+1}$.
By convention, $d^{j}$ is the zero map if $j < 0$ or $j \geq n$, and $V^{j} := 0$ if $j < 0$ or $j > n$.
If $d^{j} \circ d^{j-1} = 0$, $j = 1, \dots, n$, then $V$ and $d$ together form a \emph{Hilbert complex} $\mathcal{V} := (V, d)$,
\begin{equation}
	\begin{tikzcd}%[ampersand replacement=\&]
		\mathcal{V}~:~ V^{0} \arrow[r, "d^{0}"]& V^{1} \arrow[r, "d^{1}"] & \cdots \arrow[r, "d^{n-1}"]& V^{n}\;.
	\end{tikzcd}
	\label{eq:hilbert_complex}
\end{equation}
The connecting maps $d^{j}$ are called the differentials of the complex $\mathcal{V}$. Moreover, $\mathcal{V}$ is called bounded if its differentials are bounded linear operators, and it is called closed if the image of $d^{j}$ is closed in $V^{j+1}$ for all $j$.

The composition property of the differentials implies that the following containment holds for all $j$,
\begin{equation}
	\image{(d^{j-1})} \subseteq \kernel{(d^{j})}\;.
	\label{eq:image_in_kernel}
\end{equation}
Members of $V^{j}$ in $\kernel{(d^{j})}$ are called $j$-cocycles or closed, and the members of $V^{j}$ in $\image{(d^{j-1})}$ are called $j$-coboundaries or exact.
The $j^{th}$ cohomology space associated to the complex $\mathcal{V}$, $H^{(j)}(\mathcal{V})$, is defined as the following quotient,
\begin{equation}
	H^{(j)}(\mathcal{V}) = \kernel{(d^{j})}/\image{(d^{j-1})}\;.
	\label{eq:cohomology}
\end{equation}
The cohomology space $H^{(j)}(\mathcal{V})$ measures the extent to which the equality in Equation \eqref{eq:image_in_kernel} fails to hold.

Given two complexes $\mathcal{V} = (V, d_V)$ and $\mathcal{W} = (W, d_W)$, linear maps $f^{j} : V^{j} \rightarrow W^{j}$ are called cochain maps if they commute with the differentials for all $j$,
\begin{equation}
	d_W^{j} \circ f^{j} = f^{j+1} \circ d_V^{j}\;.
\end{equation}
Cochain maps preserve closed and exact forms and, consequently, induce maps between cohomology spaces of the two complexes, $f^{*,(j)} : H^{(j)}(\mathcal{V}) \rightarrow H^{(j)}(\mathcal{W})$.
Additionally, for all $j$, if $W^{j} \subseteq V^{j}$ and all differentials $d_W^{j}$ are obtained from $d_V^{j}$ by restriction, then the complex $\mathcal{W}$ is called a subcomplex of $\mathcal{V}$.
In this case, the inclusion $\iota : \mathcal{W} \rightarrow \mathcal{V}$ is a cochain map and induces a natural map between their cohomologies.
If, additionally, there exists a cochain projection map from $\mathcal{V}$ to $\mathcal{W}$, it induces a surjection of cohomologies.
In particular, the dimensions of $H^{(j)}(\mathcal{W})$ are then bounded from above by those of $H^{(j)}(\mathcal{V})$ for all $j$.

\subsection{The de Rham complex of differential forms}\label{sec:prelims-deRham}

Given a (sufficiently) smooth $n$-manifold $\Omega \subset \RR^n$, let $T_{\mbf{y}}\Omega$ denote the $n$-dimensional tangent space at $\mbf{y} \in \Omega$.
A smooth differential $j$-form $f$, $j = 0, \dots, n$, on $\Omega$ is a smooth field such that $f_{\mbf{y}}$ is a real-valued skew-symmetric $j$-linear form on $T_{\mbf{y}}{\Omega} \times \cdots \times T_{\mbf{y}}{\Omega}$.
Let $\Lambda^{j}(\Omega)$ denote the space of all smooth $j$-forms, $j = 0, \dots, n$.

For $j = 0, \dots, n$, and $f \in \Lambda^{j}(\Omega)$, the exterior derivative is a linear map, $d^j:\Lambda^{j}(\Omega) \rightarrow \Lambda^{j+1}(\Omega)$, such that $d^{j+1} \circ d^j = 0$.
For our purposes, we need not expose the precise definition of $d^j$, but just need to recognize that it (or rather its soon-defined extension to the $L^2$ completion) meet the criteria above.

With $L^2\Lambda^{j}(\Omega)$ denoting the completion of $\Lambda^{j}(\Omega)$ with respect to the $L^2$ inner product of $j$-forms $\inpr[L^2\Lambda^{j}(\Omega)]{\cdot}{\cdot}$, 
we define $H\Lambda^{j}(\Omega)$ as
\begin{equation}
	{
	H\Lambda^{j}(\Omega) := \left\{
	f \in L^2\Lambda^{j}(\Omega)~:~ d^j f \in L^2\Lambda^{j+1}(\Omega)
	\right\}\;.
	}
\end{equation}
With $\inpr[]{\cdot}{\cdot} := \inpr[L^2\Lambda^{j}(\Omega)]{\cdot}{\cdot}$,
we equip $H\Lambda^{j}(\Omega)$ with the following graph norm-induced inner-product,
\begin{equation}
	\inpr[H\Lambda^{j}(\Omega)]{f}{g}
	:= \inpr[]{f}{g} + \inpr[]{d^jf}{d^jg}\;.
\end{equation}
{Note that $H\Lambda^{0} = H^1(\Omega)$ and $H\Lambda^{n}(\Omega)$ = $L^2(\Omega)$.}
Then, the $L^2$ de Rham complex on $\Omega$ is the closed and bounded Hilbert complex defined as
\begin{equation}
	\begin{tikzcd}%[ampersand replacement=\&]
		\mathcal{R}~:~ H\Lambda^{0}(\Omega) \arrow[r, "d^0"]& H\Lambda^{1}(\Omega) \arrow[r, "d^1"] & \cdots \arrow[r, "d^{\ndim-1}"] & H\Lambda^{\ndim}(\Omega)\;.
	\end{tikzcd}
\end{equation}
When $\Omega$ is {contractible}, we have {$H^{(0)}(\mathcal{R}) = \RR$} and $H^{(j)}(\mathcal{R}) = 0$, $j = 1, \dots, \ndim$.

In this document, we will only consider the above complex with homogeneous boundary conditions.
This complex is be built up from differential forms of compact support and will be denoted as
\begin{equation}
	\begin{tikzcd}%[ampersand replacement=\&]
		\mathcal{R}_0~:~ H\Lambda^{0}_0(\Omega) \arrow[r, "d^0"]& H\Lambda^{1}_0(\Omega) \arrow[r, "d^1"] & \cdots \arrow[r, "d^{\ndim-1}"] & H\Lambda^{\ndim}_0(\Omega)\;,
	\end{tikzcd}
	\label{eq:deRham}
\end{equation}
where $H\Lambda^{j}_0(\Omega)$ is the space of $j$-forms with vanishing trace on $\partial \Omega$, $j = 0, \dots, n-1$, and $H\Lambda^{\ndim}_0(\Omega) = H\Lambda^{\ndim}(\Omega)$.
When $\Omega$ is {contractible}, we have $H^{(j)}(\mathcal{R}_0) = 0$, $j = 0, \dots, \ndim-1$ and {$H^{(\ndim)}(\mathcal{R}_0) = \RR$}.

\section{Notation and preliminaries}\label{sec:notation}
We now set the framework to define the hierarchical B-spline complex of discrete differential forms, proceeding largely as in \cite{Evans19}, albeit in parametric dimension $\ndim$.
To begin, we first describe the univariate scenario.
Note that all the spaces defined here already incorporate the relevant homogeneous boundary conditions.

\subsection{Univariate B-splines}
Given polynomial degree $p \geq 1$ and an integer $m \geq 1$, consider a \textknotvector, i.e.,~a non-decreasing sequence of real numbers called knots, $\kntsSym = (\knt{1},\dots,\knt{\ndofSym+\pdegSym+1})$, such that
\begin{equation}
	\knt{1} = \dots = \knt{\pdegSym} < \knt{\pdegSym+1} \leq \dots \leq \knt{\ndofSym+1} < \knt{\ndofSym+2} = \dots = \knt{\ndofSym + \pdegSym + 1}.
\end{equation}
For simplicity, we will assume that all \textknotvector{s}~are such that $\knt{1} = 0,~\knt{\ndofSym+\pdegSym+1} = 1$.
Note that the first and last knots both appear with multiplicity $p$.
We also assume that no other knot in $\kntsSym$ appears more than $\pdegSym$ times.

Given \textknotvector~$\kntsSym$, we will define two spaces of piecewise-polynomial functions.
First, denote with $\tpbs{}{0}(\kntsSym)$ the space of  piecewise-polynomial functions of degree $p$ defined on the partition of $[0, 1]$ defined by the unique values of the knots $\knt{i}$, such that:
\begin{itemize}
	\item the functions are $C^{\pdegSym - r}$ smooth at a knot $\knt{i}$ that appears in $\kntsSym$ with multiplicity $r$, and,
	\item the functions vanish at $\xi = 0$ and $\xi = 1$.
\end{itemize}
Next, denote with $\tpbs{}{1}(\kntsSym)$ the space of  piecewise-polynomial functions of degree $p-1$ defined on the partition of $[0, 1]$ defined by the unique values of the knots $\knt{i}$, such that the functions are $C^{\pdegSym - r - 1}$ smooth at a knot $\knt{i}$ that appears in $\kntsSym$ with multiplicity $r$.

The dimension of $\tpbs{}{j}(\kntsSym)$ is $\ndofSym+j$, $j = 0, 1$, and we can define $\ndofSym+j$ basis functions that span $\tpbs{}{j}(\kntsSym)$.
We will choose univariate B-splines as the basis for these spaces and we will denote them with $\gb^j_{i,\pdegSym}$, $i = 1, \dots, \ndofSym+j$; the set containing these basis functions will be denoted as $\tpb{}{j}(\kntsSym)$.
This basis has many useful properties, and the ones most interesting for this manuscript are the boundary conditions satisfied by the functions (by definition) and minimal support.
\begin{itemize}
	\item Boundary conditions: All B-splines $\gb^0_{i,\pdegSym}$ vanish at positions $\xi = 0$ and $\xi = 1$.
	Moreover, the only B-splines $\gb^1_{i,\pdegSym}$ that are non-zero at $\xi = 0$ and $\xi = 1$ are, respectively, $\gb^1_{1,\pdegSym}$ and $\gb^1_{\ndofSym+1,\pdegSym}$.
	
	\item Minimal support: We will need the fact that $\supp(\gb^j_{i,\pdegSym}) = (\knt{i},\knt{i+\pdegSym-j+1})$ and, moreover, $\gb^j_{i,\pdegSym}$ is defined by, and thus can be uniquely identified with, the following subsequence of $\kntsSym$,
	\begin{equation}
		\gb^j_{i,\pdegSym} \longleftrightarrow \lknts{i}{j}{} := (\knt{i}, \dots,  \knt{i+\pdegSym-j+1})\;.
	\end{equation}
	By convention, we will define all supports to be open sets in this manuscript.
	Hereafter, we will exclusively employ this unique identification, i.e., instead of talking about $B^j_{i,\pdegSym}$, we will only talk about $\lknts{i}{j}{}$.
\end{itemize}

The last things we define in the univariate setting are the \Bezier~and Greville meshes.
The knots in $\kntsSym$ partition $(0,1)$ into a mesh that will be called the univariate \Bezier{} mesh and denoted as $\unibmeshSym(\kntsSym)$.
Moreover, we associate the B-spline $\lknts{i}{0}{}$ with a point in $(0,1)$ called the $i$-th Greville point\footnote{
	Ignore boundary conditions and consider degree $\pdegSym$ B-splines defined on the augmented \textknotvector~$(\knt{0},\kntsSym,\knt{\ndofSym+\pdegSym+2})$, with $\knt{0} = 0$ and $\knt{\ndofSym+\pdegSym+2} = 1$.
	Then, the linear polynomial $f(\kntSym) = \kntSym$ can be expressed as a unique linear combination of the B-splines $\lknts{i}{0}{}$, $i = 0, \dots, \ndofSym+1$, and the points $\greville{i}{}$, $i = 0, \dots, \ndofSym+1$, are the corresponding coefficients of linear combination.
} which is defined as
\begin{equation}
	\grevilleE{i}{}{0} := \frac{\knt{i+1} + \cdots + \knt{i + \pdegSym}}{\pdegSym}\;,
\end{equation}
with the convention that $\greville{0}{} := 0$ and $\greville{\ndofSym+1}{} = 1$.
By the assumptions placed on the knots, $0 = \greville{0}{} < \greville{1}{} < \cdots < \greville{\ndofSym}{} < \greville{\ndofSym+1}{} = 1$.
The Greville points help partition $[0,1]$ into $(\ndofSym+1)$ intervals.
We will call the interior of the $i$-th interval in this partition a Greville edge identified with the B-spline $\lknts{i}{1}{}$, $i = 1, \dots, \ndofSym+1$,
\begin{equation}
	\grevilleE{i}{}{1} := (\greville{i-1}{}, \greville{i}{})\;.
\end{equation}
Then, collecting the Greville points in the set $\gmesh{}^0(\kntsSym)$ and the Greville edges in the set $\gmesh{}^1(\kntsSym)$, the mesh formed by them will be called a Greville mesh will be denoted by $\gmesh{}(\kntsSym)$.
Finally, $\gmap{}{j}(\kntsSym)$ will denote the map that performs the above identification of B-splines with the $j$-dimensional Greville mesh entities,
\begin{equation}
	\gmap{}{j}(\kntsSym)~:~ \tpb{}{j}(\kntsSym) \rightarrow \gmesh{}^j(\kntsSym)\;.
\end{equation}

Figure \ref{fig:one_d_isomorphism} shows an example of the B-splines in $\tpb{}{j}(\kntsSym)$ and their correspondence to the members of $\gmesh{}^j(\kntsSym)$.
Note that the evaluation of the B-spline basis functions can be performed, for instance, with the Cox--de Boor recursion formula \cite{DeBoor}.

\begin{figure}
	\centering
	\includegraphics[width=0.9\textwidth]{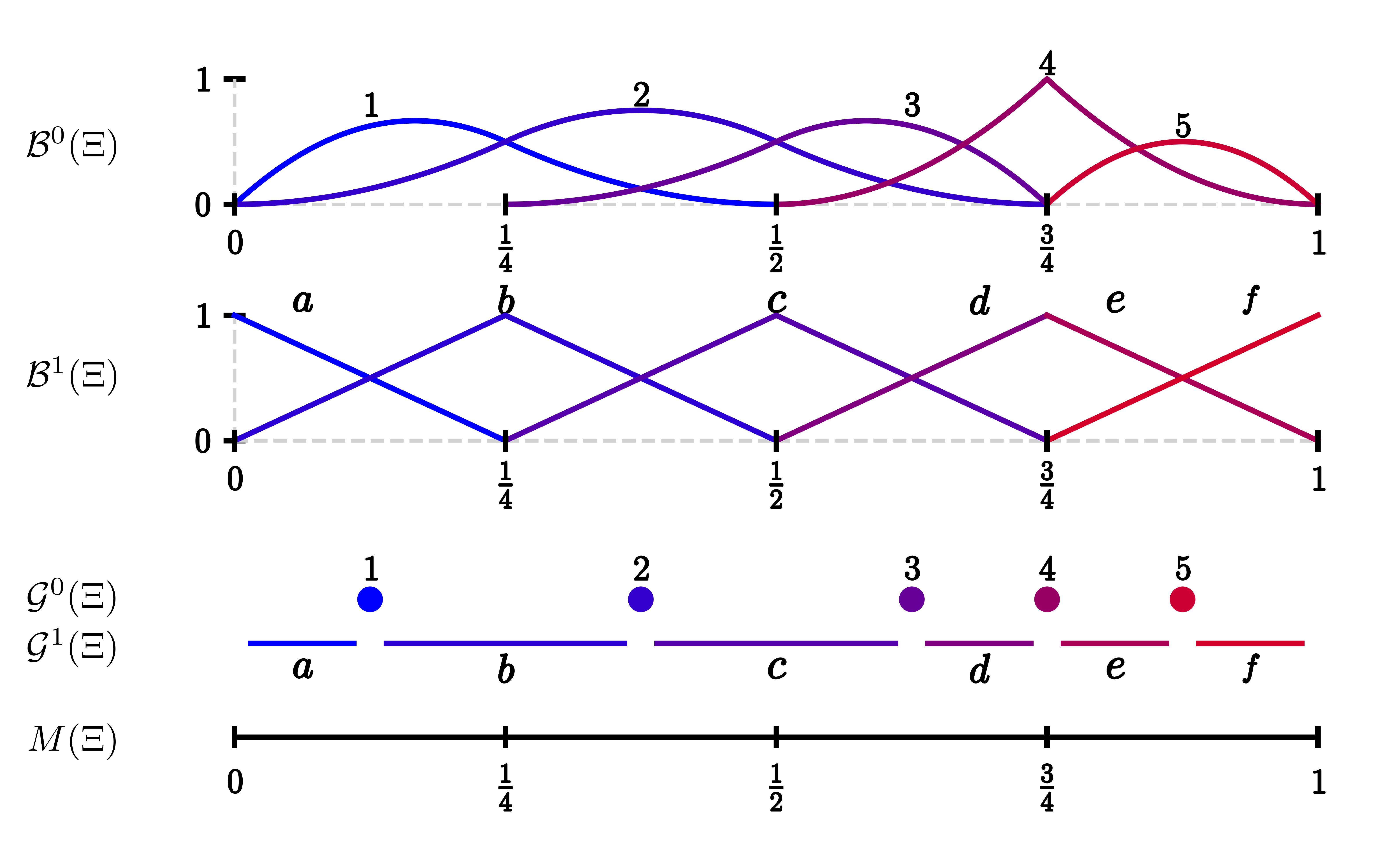}
	\caption{The one-to-one correspondence between a set of univariate splines with a Greville grid is shown here in one dimension.
	On the top, the quadratic B-spline basis $\tpb{}{0}(\kntsSym)$ corresponding to $\kntsSym = \{0,0,\frac{1}{4},\frac{1}{2},\frac{3}{4},\frac{3}{4},1,1\}$ is shown. 
	On the bottom, the spline basis $\tpb{}{1}(\kntsSym)$ is shown.
	The Greville points $\gmesh{}^0(\kntsSym)$ and edges $\gmesh{}^1(\kntsSym)$ are shown next.
	Basis functions of $\tpb{}{0}(\kntsSym)$ are in one-to-one correspondence with points of $\gmesh{}^0(\kntsSym)$ as indicated with the numerical labeling, while basis functions of $\tpb{}{1}(\kntsSym)$ are in one-to-one-correspondence with edges of $\gmesh{}^1(\kntsSym)$ through the alphabetically labeled relationship.
	Finally, the univariate B\'ezier mesh, $\unibmeshSym(\kntsSym),$ is shown at the bottom, with ticks indicating vertices.
	}\label{fig:one_d_isomorphism}
\end{figure}

\subsection{Tensor-product B-splines}\label{sec:notation-tpbs}

Using the univariate splines defined above, tensor products can be used to define multivariate spline spaces and a basis for them.
Specifically, on $\Omega = (0,1)^\ndim \subset \RR^\ndim$, a tensor product spline space and its B-spline basis are defined by choosing $\ndim$ univariate \textknotvector{s} and taking the tensor-products of the associated univariate spline spaces and their B-spline bases, respectively.

With a view toward the next sections where we introduce the hierarchical spline spaces, we describe here the construction of $\nref+1$ nested tensor-product spline spaces, $\nref \geq 0$.
Given $0 \leq \ell \leq \nref$, let the \textknotvector{s}~in the $k$-th direction be denoted by $\knts{\ell,k}$, $k = 1, \dots, \ndim$.
We assume that these univariate \textknotvector{s}~satisfy the assumptions placed on \textknotvector{s}~in the previous subsection.
With $\pdeg{\ell}{k}$ denoting the corresponding polynomial degree, we will define the spaces $\tpbs{}{j_k}(\knts{\ell,k})$, $j_k \in \{0,1\}$ as in the previous subsection, and the dimension of $\tpbs{}{0}(\knts{\ell,k})$ will be denoted by $\ndof{\ell}{k}$; the dimension of $\tpbs{}{1}(\knts{\ell,k})$ will thus be $\ndof{\ell}{k}+1$.

Then, given an $\ndim$-tuple $\mbf{j} = (j_1,\dots,j_\ndim)$ such that $j_k \in \{0,1\}$ for all $k$, we define the tensor-product spline space $\tpbs{\ell}{\mbf{j}}$ as the following tensor product of univariate spline spaces,
\begin{equation}
	\tpbs{\ell}{\mbf{j}} := \tpbs{}{j_1}(\knts{\ell,1}) \otimes \cdots \otimes \tpbs{}{j_\ndim}(\knts{\ell,\ndim})\;.
\end{equation}

\begin{remark}
	To differentiate notation in the univariate setting from the general tensor-product setting, we will explicitly incorporate the (local) univariate \textknotvector{s} in the former; the multidimensional \textknotvector{s} will not be included in the latter notation.
	For instance, we have suppressed the explicit dependence of $\tpbs{\ell}{\mbf{j}}$ on the underlying knot sequences.
\end{remark}

\begin{assumption}\label{assume0}
	We assume that the \textknotvector{s}~are nested.
	That is, for all $k$ and $\ell' > \ell$, if a knot appears in $\knts{\ell,k}$ with multiplicity $r$, then its multiplicity in $\knts{\ell',k}$ is at least $r$.
	This automatically implies that, for all $\mbf{j}$,
	\begin{equation}
		\tpbs{0}{\mbf{j}} \subset \tpbs{1}{\mbf{j}} \subset \cdots \subset \tpbs{\nref}{\mbf{j}}\;.
	\end{equation}
\end{assumption}

Let us now dive into the details of these spaces and identify the B-splines and the Greville meshes by considering a fixed level $0 \leq \ell \leq \nref$.
The $i$-th knot in $\knts{\ell,k}$ will be denoted with $\knt{i,\ell,k}$.
For the $k$-th parametric direction, the $i_k$-th univariate B-spline will be uniquely identified with the local \textknotvector~$\lknts{i_k}{j_k}{\ell,k}$.
Consequently, the $\mbf{i}$-th tensor-product B-spline $\gb^{\mbf{j}}_{\mbf{i}} \in \tpbs{\ell}{\mbf{j}}$, $1 \leq i_k \leq \ndof{\ell}{k} + j_k$ for all $k$, will be identified with the following Cartesian product of the local \textknotvector{s} that define it,
\begin{equation}
	\gb^{\mbf{j}}_{\mbf{i}} \longleftrightarrow \bigtimes_{k=1}^\ndim  \lknts{i_k}{j_k}{\ell,k}\;.
\end{equation}
All such B-splines in $\tpbs{\ell}{\mbf{j}}$ will be collected in the set $\tpb{\ell}{\mbf{j}}$.
Their support is the Cartesian product of the univariate B-splines' support,
\begin{equation}
	\begin{split}
		\supp(\phi^\mbf{j}_\mbf{i}) = \bigtimes_{k=1}^\ndim\supp(\lknts{i_k}{j_k}{\ell,k})\;.
	\end{split}
\end{equation}
The Greville entity in $\RR^\ndim$ associated to $\phi^\mbf{j}_\mbf{j}$ is simply the Cartesian product of those associated to the univariate B-splines,
\begin{equation}
	\mu\left( \phi^\mbf{j}_\mbf{i} \right) = 
	\bigtimes_{k=1}^n \grevilleE{i_k}{\ell,k}{j_k}\;.
\end{equation}
The set $\gmesh{\ell}^j$ will contain all Greville entities $\mu\left( \phi^\mbf{j}_\mbf{i} \right)$ such that $|\mbf{j}| = j$, where $|\mbf{j}| = \sum_{k=1}^n j_k$.
The map $\gmap{\ell}{j}$ will perform the one-to-one identification of the B-splines with the corresponding elements of $\gmesh{\ell}^j$,
\begin{equation}
	\gmap{\ell}{j}:\bigcup_{|\mbf{j}| = j} \tpb{\ell}{\mbf{j}} \rightarrow \gmesh{\ell}^j\;.
\end{equation}
Combining $\gmesh{\ell}^j$ for all $j$, we obtain the cuboidal Greville mesh $\gmesh{\ell}$.
Similarly, the tensor-product \Bezier~mesh $\unibmeshSym$ will be defined as the Cartesian product of the univariate meshes defined by the \textknotvector{s} $\knts{\ell,k}$,
\begin{equation}
	\unibmeshSym_\ell := \bigtimes_{k=1}^n \unibmeshSym(\knts{\ell,k})\;.
\end{equation}
	
\subsection{Tensor-product spline differential forms}
Using the tensor-product spline spaces $\tpbs{\ell}{\mbf{j}}$, we can now define the space of tensor-product spline differential $j$-forms at the $\ell$-th refinement level as below,
\begin{equation}
	\tpbs{\ell}{j} := \bigtimes_{|\mbf{j}| = j}\tpbs{\ell}{\mbf{j}}\;.
\end{equation}
Equivalently, we can rewrite the above as follows, noting that all superscripts on the right are $n$-tuples,
\begin{equation}
	\begin{split}
		0\text{-forms, }\tpbs{\ell}{0}\;\;\;\; &:= \tpbs{\ell}{(0,0,\dots,0)}\;,\\
		1\text{-forms, }\tpbs{\ell}{1}\;\;\;\; &:= \tpbs{\ell}{(1,0,\dots,0)} \times \tpbs{\ell}{(0,1,0,\dots,0)} \times \cdots \times \tpbs{\ell}{(0,\dots,0,1)}\;,\\
		&\vdots \\
		\ndim\text{-forms, }\tpbs{\ell}{\ndim}\;\;\;\; &:= \tpbs{\ell}{(1,1,\dots,1)}\;.
	\end{split}
\end{equation}
As noted in the previous section, $j$-forms can be naturally identified with $j$-cells in the Greville mesh.

\subsection{Hierarchical B-splines}
For all $j$-forms, $j = 0,\dots,\ndim$, we will use the $(\nref+1)$ levels of nested tensor-product spline spaces defined above to build a space of hierarchical B-splines.
To do so, we first build a domain hierarchy
\begin{equation}
	\Omega =: \Omega_0 \supseteq \Omega_1 \supseteq \cdots \supseteq \Omega_\nref \supseteq \Omega_{\nref+1} := \emptyset\;,
\end{equation}
that helps indicate which tensor-product B-splines contribute to the hierarchical space.

\begin{definition}[Hierarchical B-splines]\label{def:hbs}
	The set of hierarchical B-splines $\hb{\nref}{\mbf{j}}$ is constructed using the following recursive algorithm \cite{Evans19,Kraft,Vuong_giannelli_juttler_simeon}:
	\begin{enumerate}
		\item Initialization, $\ell=0$: $\hb{0}{\mbf{j}} := \tpb{0}{\mbf{j}}$.
		\item Recursion, $\ell=0, \dots, \nref-1$:
		Construct $\hb{\ell+1}{\mbf{j}}$ from $\hb{\ell}{\mbf{j}}$ by setting
		\[
		\hb{\ell+1}{\mbf{j}} = \hb{\ell+1,c}{\mbf{j}}\cup \hb{\ell+1,f}{\mbf{j}}\;,
		\]
		where
		\begin{align*}
			\hb{\ell+1,c}{\mbf{j}} &:= \{\gb \in \hb{\ell}{\mbf{j}}: \supp(\gb) \not\subset \Omega_{\ell+1}\}
			\;,\\
			\hb{\ell+1,f}{\mbf{j}} &:= \{\gb \in \tpb{\ell+1}{\mbf{j}}:\supp(\gb) \subset \Omega_{\ell+1}\}\;.
		\end{align*}
	\end{enumerate}
	For $\ell = 0, \dots, \nref$, the spaces of hierarchical splines $\hbs{\ell}{\mbf{j}}$ are simply defined to be the span of $\hb{\ell}{\mbf{j}}$,
	\begin{equation}
		\hbs{\ell}{\mbf{j}} := \vspan{\hb{\ell}{\mbf{j}}}\;.
	\end{equation}
\end{definition}

Taking inspiration from \cite{Evans19,Vuong_giannelli_juttler_simeon}, we place the following assumption on these domains.
This is a common assumption on the construction of hierarchical spline spaces and places a (level-dependent) lower-bound on the sizes of the refinement domains.
\begin{assumption}\label{assume:zero_union}
	The subdomain $\Omega_{\ell+1}$ is given as the union of supports of $0$-form basis functions from the previous level.
	That is, for $\ell = 0,\dots,\nref-1$, there exists $S \subset \tpb{\ell}{\mbf{0}}$ such that
	\begin{equation}
		\Omega_{\ell+1} = \bigcup_{\zerob \in S} \supp(\zerob)\;.
	\end{equation}
\end{assumption}
	
\subsection{Hierarchical spline differential forms}
Using the above definition, and similarly to the tensor-product case, we define the spaces of hierarchical B-spline differential forms, $\hbs{\ell}{j}$, $j = 0, \dots, n$, $\ell = 0, \dots, \nref$, as
\begin{equation}
	\hbs{\ell}{j} := \bigtimes_{|\mbf{j}| = j}\hbs{\ell}{\mbf{j}}\;.
\end{equation}
Equivalently, we can rewrite the above as follows, noting that all superscripts on the right are $n$-tuples,
\begin{equation}
	\begin{split}
		0\text{-forms, }\hbs{\ell}{0} &:= \hbs{\ell}{(0,0,\dots,0)}\;,\\
		1\text{-forms, }\hbs{\ell}{1} &:= \hbs{\ell}{(1,0,\dots,0)} \times \hbs{\ell}{(0,1,0,\dots,0)} \times \cdots \times \hbs{\ell}{(0,\dots,0,1)}\;,\\
		&\vdots \\
		n\text{-forms, }\hbs{\ell}{n} &:= \hbs{\ell}{(1,1,\dots,1)}\;.
	\end{split}
\end{equation}

For ease of reading, we will reserve certain symbols for referring to (hierarchical) B-splines and general (hierarchical) splines for certain choices of $j$; we tabulate this below.
This notation will be reused for the entirety of this document, and is also consistent with the notation used thus far.
\begin{table}[ht]
	\centering
	\begin{tabular}{|c|c|c|}
			\hline
			\hline
			\rowcolor{gray!14}
			$j$-forms & (Hierarchical) B-splines & General splines\\\hline
			$j = 0$ & $\zerob$ & $\zerof$\\\hline
			$j = \ndim$ & $\threeb$ & $\threef$\\\hline
			Unspecified $k$ & $\gb$ & $\gf$\\\hline
			\hline
		\end{tabular}
		\caption{To simplify reading of the text, the above symbols will be reserved and consistently reused for talking about (B-)spline differential forms.}
\end{table}

\section{Exactness of the $n$-Dimensional Hierarchical Spline Complex of Differential Forms}\label{sec:main-result}
We now present several results that will be useful for proving exactness of the hierarchical spline de Rham complex in $n$ dimensions.
Section \ref{sec:mayer-vietoris} collects some elementary results from algebraic topology without proof, since their proofs can be found in multiple references; we will use \cite{Hatcher} as our main reference.
Section \ref{sec:exactness} contains the main result of this paper -- a proof of exactness for the $n$-dimensional hierarchical B-spline complex under suitable assumptions.

\subsection{Mayer--Vietoris on Spline Functions}\label{sec:mayer-vietoris}
{
We start by defining here some notation and associated shorthand that will be used extensively in the following.
The purpose is to collect and highlight the definitions so that the reader can easily find them, as well as compare them against similar notations that will be introduced later in Section \ref{sec:mayer-vietoris}.
We define a special subset of B-splines in $\tpb{\ell}{\mbf{j}}$ (and the associated spline space, Greville submesh and the space of $j$-forms), denoted $\tpb{\ell}{\mbf{j}}(Y)$, that contains B-splines in $\tpb{\ell}{\mbf{j}}$ that are supported on $Y \subset \Omega$, i.e.,
\begin{subequations}
\begin{align}
	\tpb{\ell}{\mbf{j}}(Y) &:= \left\{ \gb \in \tpb{\ell}{\mbf{j}}~:~ \supp(\gb) \subset Y\right\}\;,\\
		\tpbs{\ell}{\mbf{j}}(Y) &:= \vspan{\tpb{\ell}{\mbf{j}}(Y)}\;,\\
		\tpbs{\ell}{j} &:= \bigtimes_{|\mbf{j}| = j}\tpbs{\ell}{\mbf{j}}(Y)\;,\\
		\gmesh{\ell}(Y) &:= \bigcup_{\substack{\gb \in \tpb{\ell}{\mbf{j}}(Y)\\ 0\leq |\mbf{j}|\leq n}} \gmap{\ell}{|\mbf{j}|}(\gb)\;.
\end{align}\label{eq:bspl_l_subset}
\end{subequations}
This notation and the associated shorthand is collected in Table \ref{tab:bspl_l_subset} for convenience.
Observe that the last row of Table \ref{tab:bspl_l_subset} coincides with definitions introduced previously.
\begin{table}[ht]
	\centering
	\begin{tabular}{|>{\centering\arraybackslash}m{2.8cm}|>{\centering\arraybackslash}m{2cm}|>{\centering\arraybackslash}m{2cm}|>{\centering\arraybackslash}m{2cm}|>{\centering\arraybackslash}m{2cm}|}
		\hline\hline
		\textsc{Sub-domain} & \multicolumn{4}{c|}{\textsc{(Shorthand) Notation for \eqref{eq:bspl_l_subset}}}\\\hline
		\cellcolor[gray]{0.9}$Y \subset \Omega$ & $\tpb{\ell}{\mbf{j}}(Y)$ & $\tpbs{\ell}{\mbf{j}}(Y)$ & $\tpbs{\ell}{j}(Y)$ & $\gmesh{\ell}(Y)$ \\\hline
		\cellcolor[gray]{0.9}$Y = \Omega_{\ell'},~\ell' \geq \ell$ & $\tpb{\ell,\ell'}{\mbf{j}}$ & $\tpbs{\ell,\ell'}{\mbf{j}}$ & $\tpbs{\ell,\ell'}{j}$ & $\gmesh{\ell,\ell'}$ \\\hline
		\cellcolor[gray]{0.9}$Y = \Omega$ & $\tpb{\ell}{\mbf{j}}$ & $\tpbs{\ell}{\mbf{j}}$ & $\tpbs{\ell}{j}$ & $\gmesh{\ell}$ \\\hline
		\hline
	\end{tabular}
\caption{General and shorthand notation corresponding to Equation \eqref{eq:bspl_l_subset} that will be used extensively in the following sections for different choices of $Y \subset \Omega$.
Observe that the notation in the last row above coincides with definitions introduced previously in Section \ref{sec:notation-tpbs}.
	}\label{tab:bspl_l_subset}
\end{table}
}

\subsubsection{Local Cochain Complexes}
For the domain $\Omega$, the coboundary operator for the de Rham cochain complex of differential forms is the exterior derivative,
$\xder{}{j}~:~H\Lambda^j_0(\Omega) \rightarrow H\Lambda^{j+1}_0(\Omega)$, and it maps $j$-forms to $(j+1)$-forms and enjoys the property
\begin{equation}\label{eq:tp_xder_comp}
	\xder{}{j+1} \circ \xder{}{j} = 0\;.
\end{equation}
For any subspace of $H\Lambda^j_0(\Omega)$ (e.g., $\tpbs{\ell}{j}$) the exterior derivative is obtained by restriction of $\xder{}{}$ to the subspace.
Note that the combination of the (restrictions of) the exterior derivative with the spaces $\tpbs{\ell}{j}$ for $j=0,\dots,n$, we obtain a cochain complex $\cocomplex{\tpbs{\ell}{}}{\xder{}{}}$ or simply $\tpbs{\ell}{}$,
\begin{equation}
	\begin{tikzcd}
		\tpbs{\ell}{}~:~\tpbs{\ell}{0} \arrow{r}{d^0} & \tpbs{\ell}{1} \arrow{r}{d^1} & \cdots \arrow{r}{d^{n-1}}  & \tpbs{\ell}{n}\;.
	\end{tikzcd}
\end{equation}

\begin{remark}
For simplicity of notation, and because it will always be clear from the context, we will remove the superscript from the exterior derivative.
\end{remark}

By restricting to a subdomain, $Y \subset \Omega$, we can create subcomplexes of the above complex.
This is encapsulated in the following elementary results.

\begin{lemma}\label{lem:subspace_complex_defined}
	For $Y \subset \Omega$, the set of spline spaces $\tpbs{\ell}{j}(Y)$ for $j=0,\dots,n$ and (restriction of) exterior derivative define a cochain complex $\cocomplex{\tpbs{\ell}{}(Y)}{\xder{}{}}$, that is a subcomplex of $\cocomplex{\tpbs{\ell}{}}{\xder{}{}}$.
\end{lemma}

\begin{corollary}\label{corol:subdomain_complexes}
	Let $X \subset Y \subset \Omega$. Then the cochain complex 
		$\cocomplex{\tpbs{\ell}{}(X)}{\xder{}{}}$
	is a subcomplex of $\cocomplex{\tpbs{\ell}{}(Y)}{\xder{}{}}$.
\end{corollary}

With the above definitions in place for the cochain complex $\cocomplex{\tpbs{\ell}{}(Y)}{\xder{}{}}$, we define the following spaces that respectively contain the $\ell$-th level $j$-forms that are in the kernel and range of the exterior derivative,
\begin{equation}
	\begin{split}
		\kernel^j_\ell(Y) := \left\{ \gf \in \tpbs{\ell}{j}(Y)~:~d\gf = 0\right\}\;,\quad
		\image^{j-1}_\ell(Y) := \left\{ d\gf ~:~\gf \in \tpbs{\ell}{j-1}(Y)\right\}\;.
	\end{split}
\end{equation}
The corresponding $j$-th cohomology group is then defined as
\begin{equation}
	{H^{(j)}_{\ell}(Y) := \kernel^j_\ell(Y) / \image^{j-1}_\ell(Y)\;.}
\end{equation}

% -----------------------
% subsubsection 
% -----------------------
\subsubsection{Construction of Mayer--Vietoris Sequences}

Let $A,B \subset \Omega$ be domains with cochain complexes $\cocomplex{\tpbs{\ell}{}(A)}{\xder{}{}}$ and $\cocomplex{\tpbs{\ell}{}(B)}{\xder{}{}}$.
Similarly, define the spline complexes $\cocomplex{\tpbs{\ell}{}(A\cup B)}{\xder{}{}}$ and $\cocomplex{\tpbs{\ell}{}(A\cap B)}{\xder{}{}}$.
Also, we define 
\begin{subequations}
\begin{align}
	\tpb{\ell}{j}(A \boxplus B) &:= \tpb{\ell}{j}(A) \cup \tpb{\ell}{j}(B) \;,\\
	\tpbs{\ell}{j}(A \boxplus B) &:= \left\langle \tpb{\ell}{j}(A \boxplus B) \right\rangle = \tpbs{\ell}{j}(A) + \tpbs{\ell}{j}(B)\;.
\end{align}
\end{subequations}
Then, we obtain a short exact sequence of chain complexes,
\begin{equation}
	\tpbs{\ell}{}(A \cap B) \rightarrow \tpbs{\ell}{}(A) \oplus \tpbs{\ell}{}(B) \rightarrow \tpbs{\ell}{}(A \boxplus B)\;,
\end{equation}
as shown explicitly in the following commutative diagram; see \cite[p.~149]{Hatcher} for details.

\begin{lemma}\label{lem:Mayer_vietoris}
	The following commutative diagram of short exact sequences holds\\
	\begin{tikzcd}
		 & 0 \arrow{d} & 0 \arrow{d} & 0 \arrow{d} & \\
		\cdots \arrow{r}{d} & \tpbs{\ell}{j-1}(A \cap B) \arrow{r}{d}\arrow{d}{\varphi} & \tpbs{\ell}{j}(A \cap B) \arrow{r}{d} \arrow{d}{\varphi} & \tpbs{\ell}{j+1}(A \cap B) \arrow{r}{d} \arrow{d}{\varphi} & \cdots\\
		\cdots \arrow{r}{d} & \tpbs{\ell}{j-1}(A)\oplus \tpbs{\ell}{j-1}(B) \arrow{r}{d} \arrow{d}{\psi} & \tpbs{\ell}{j}(A)\oplus \tpbs{\ell}{j}(B) \arrow{r}{d} \arrow{d}{\psi} & \tpbs{\ell}{j+1}(A)\oplus \tpbs{\ell}{j+1}(B) \arrow{r}{d} \arrow{d}{\psi} & \cdots\\
		\cdots \arrow{r}{d} & \tpbs{\ell}{j-1}(A \boxplus B) \arrow{r}{d}\arrow{d} & \tpbs{\ell}{j}(A \boxplus B) \arrow{r}{d} \arrow{d} & \tpbs{\ell}{j+1}(A \boxplus B) \arrow{r}{d} \arrow{d}& \cdots\\
		& 0 & 0 & 0 & 
	\end{tikzcd}
	where
	\begin{equation*}
		\varphi~:~\phi \mapsto (\phi, \phi) \;,
		\qquad
	 	\psi~:~(\phi_1,\phi_2) \mapsto \phi_1-\phi_2\;.
	\end{equation*}
\end{lemma}

Moreover, from the above commutative diagram and the Snake Lemma \cite[p.~116-117]{Hatcher}, we obtain a long exact sequence connecting the different cohomologies of the three complexes.

\begin{corollary}\label{corol:Mayer_vietoris}
	The following long exact sequence connects the cohomologies and is called the Mayer--Vietoris sequence,
	\[
	\begin{tikzcd}[column sep=0.13in]
		\cdots \arrow{r} 
		& 
		{H^{(j)}_{\ell}(A\cap B) \arrow{r}{\varphi^*} }
		&
		{H^{(j)}_{\ell}(A) \oplus H^{(j)}_{\ell}(B)\arrow{r}{\psi^*}}
		 &
		 {H^{(j)}_{\ell}(A \boxplus B)\arrow{r}{\partial^*} }
		 &
		 {H^{(j+1)}_{\ell}(A \cap B) \arrow{r}}
		 & \cdots
	\end{tikzcd}
	\]
\end{corollary}

Due to nestedness of the multi-level spline spaces, the level $\ell$ spline spaces are contained in the level $\ell'$ spline spaces, $\ell' \geq \ell$.
These inclusions imply the following commutative diagram between the corresponding long exact sequences of cohomologies; see \cite[p.~127-128]{Hatcher}.

\begin{theorem}\label{thrm:natural_Mayer_vietoris}
	For domains $A,B \subset \Omega$, $\ell \leq \ell' \leq \nref$, and the spaces that appear in the Mayer--Vietoris sequence above, the inclusion operator 
	$\splinelevelinclusion{\ell,\ell'}:\tpbs{\ell}{}(\cdot)\rightarrow \tpbs{\ell'}{}(\cdot)$
	 induces a function on cohomology {$\splinelevelinclusion{\ell,\ell'}^*:H^{(j)}_{\ell}(\cdot)\rightarrow H^{(j)}_{\ell'}(\cdot)$} for $j=0,\dots,n$, such that the following Mayer--Vietoris sequences commute:
	\footnotesize
	\[
	\begin{tikzcd}
		\cdots \arrow{r} & H^{(j)}_{\ell}(A \cap B) \arrow{r}{\varphi^*} \arrow{d}{\splinelevelinclusion{\ell,\ell'}^*} &H^{(j)}_{\ell}(A) \oplus H^{(j)}_{\ell}(B) \arrow{r}{\psi^*} \arrow{d}{\splinelevelinclusion{\ell,\ell'}^*} & H^{(j)}_{\ell}(A \boxplus B) \arrow{r}{\partial^*} \arrow{d}{\splinelevelinclusion{\ell,\ell'}^*} & H^{(j+1)}_{\ell}(A \cap B) \arrow{r} \arrow{d}{\splinelevelinclusion{\ell,\ell'}^*} & \cdots\\
		\cdots \arrow{r} & H^{(j)}_{\ell'}(A\cap B) \arrow{r}{\varphi^*} &H^{(j)}_{\ell'}(A) \oplus H^{(j)}_{\ell'}(B) \arrow{r}{\psi^*}  & H^{(j)}_{\ell'}(A \boxplus B) \arrow{r}{\partial^*} & H^{(j+1)}_{\ell'}(A \cap B) \arrow{r} & \cdots
	\end{tikzcd}
	\]
\end{theorem}

Finally, the following result of this section follows from the Five Lemma \cite[p.~129]{Hatcher} applied to the commuting diagram from Theorem \ref{thrm:natural_Mayer_vietoris}.

\begin{corollary}\label{cor:natural_Mayer_vietoris}
	Consider the commuting diagram of Mayer--Vietoris sequences from Theorem \ref{thrm:natural_Mayer_vietoris}.
	If any two of the following vertical maps are isomorphisms for all $j$,
	\begin{itemize}
		\item $H^{(j)}_\ell (A \cap B) \rightarrow H^{(j)}_{\ell'} (A \cap B)$,
		\item $H^{(j)}_\ell (A) \oplus H^{(j)}_\ell (B) \rightarrow H^{(j)}_{\ell'} (A) \oplus H^{(j)}_{\ell'} (B)$,
		\item $H^{(j)}_\ell (A \boxplus B) \rightarrow H^{(j)}_{\ell'} (A \boxplus B)$,
	\end{itemize}
	then so is the third map.
\end{corollary}

\subsection{Application of Mayer--Vietoris to Exactness}\label{sec:exactness}
{
We are now ready to embark on the exactness proof.
The proof has three main steps: decomposition of $\Omega_{\ell+1}$ into small overlapping subdomains; proof that the level $\ell$ and $\ell+1$ spline complexes on these subdomains are cohomologically equivalent; and, finally, piecing together the subdomains to show the same for the level $\ell$ and $\ell+1$ spline complexes on all of $\Omega_{\ell+1}$.
Along the way, we will need to introduce one additional assumption for the third step.

\begin{remark}
	All our results are based on the relationship between the refinement domains at two successive levels, and therefore all our figures will only show two levels of refinement.
\end{remark}

As in Section \ref{sec:exactness}, we start by defining some notation and associated shorthand that will used extensively in what follows. In particular, for $s \in \{0, 1\}$, we define a subset of B-splines in  $\tpb{\ell+s}{\mbf{j}}$ (and the associated spline space, Greville submesh and the space of $j$-forms), denoted $\tpb{\ell+s,\ell+1}{\mbf{j}}(Y)$, that contains the B-splines in $\tpb{\ell+s}{\mbf{j}}$ that are supported on $\supp(\zerob)$ where $\zerob \in \tpb{\ell}{\mbf{0}}$ is itself supported on $Y \cap \Omega_{\ell+1}$,
\begin{subequations}
	\begin{align}
		\tpb{\ell+s,\ell+1}{\mbf{j}}(Y) &:= 
		\left\{
		\phi \in \tpb{\ell+s}{\mbf{j}}~:~
		\exists \zerob \in \tpb{\ell}{\mbf{0}}\;,~\supp(\phi) \subset \supp(\zerob) \subset Y \cap \Omega_{\ell+1}
		\right\}\;,\\
		\tpbs{\ell+s,\ell+1}{\mbf{j}}(Y) &:= 
		\vspan{\tpb{\ell+s,\ell+1}{\mbf{j}}(Y)}\;,\\
		\tpbs{\ell+s,\ell+1}{j}(Y) &:= \bigtimes_{|\mbf{j}| = j}\tpbs{\ell+s,\ell+1}{\mbf{j}}(Y)\;,\\
		\gmesh{\ell+s,\ell+1}(Y)&:= \bigcup_{\substack{\gb \in \tpb{\ell+s,\ell+1}{\mbf{j}}(Y)\\ 0\leq |\mbf{j}|\leq n}} \gmap{\ell}{|\mbf{j}|}(\gb)\;.
	\end{align}\label{eq:bspl_ll_subset}
\end{subequations}
This notation and the associated shorthand is collected in Table \ref{tab:bspl_l_subset} for convenience.
Observe that the shorthand notation in the second row of Table \ref{tab:bspl_ll_subset} is the same as the notation in the second and third rows of Table \ref{tab:bspl_l_subset}---they refer to the same objects that can be defined using both Equations \eqref{eq:bspl_l_subset} and \eqref{eq:bspl_ll_subset}.
\begin{table}[ht]
	\centering
	\begin{tabular}{|>{\centering\arraybackslash}m{4.2cm}|>{\centering\arraybackslash}m{2cm}|>{\centering\arraybackslash}m{2cm}|>{\centering\arraybackslash}m{2cm}|}
		\hline\hline
		\textsc{Sub-domain \& Levels} & \multicolumn{3}{c|}{\textsc{Notation for \eqref{eq:bspl_ll_subset}}}\\\hline
		\cellcolor[gray]{0.9}$Y \subset \Omega$~,~$s \in \{0, 1\}$ & $\tpb{\ell+s,\ell+1}{\mbf{j}}(Y)$ & $\tpbs{\ell+s,\ell+1}{\mbf{j}}(Y)$ & $\gmesh{\ell+s,\ell+1}(Y)$ \\\hline
		\cellcolor[gray]{0.9}$Y = \Omega_{\ell+1}$~,~$s = 0$ & $\tpb{\ell,\ell+1}{\mbf{j}}$ & $\tpbs{\ell,\ell+1}{\mbf{j}}$ & $\gmesh{\ell,\ell+1}$ \\\hline
		\cellcolor[gray]{0.9}$Y = \Omega_{\ell+1}$~,~$s = 1$ & $\tpb{\ell+1,\ell+1}{\mbf{j}}$ & $\tpbs{\ell+1,\ell+1}{\mbf{j}}$ & $\gmesh{\ell+1,\ell+1}$ \\\hline\hline
	\end{tabular}
\caption{General and shorthand notation corresponding to equation \eqref{eq:bspl_ll_subset} that will be used extensively in the following sections for different choices of $Y \subset \Omega$ and $s \in \{0,1\}$.
Observe that the shorthand notation in the second and third rows above is the same as the notation in the second row of Table \ref{tab:bspl_l_subset}.
}\label{tab:bspl_ll_subset}
\end{table}

We will also make use of \textextendedknotdomain{s} which are defined for each parametric direction as
\begin{equation}
	\xgb^{\mbf{j}}_{\mbf{i}} := \bigtimes_{k=1}^\ndim  \xknts{i_k}{j_k}{\ell,k},
\end{equation}
where the \textextendeduniknotdomain{s} $\xknts{i_k}{j_k}{\ell,k}$ are defined as
\begin{equation}
	\begin{split}
		\xknts{i_k}{j_k}{\ell,k} &:= \left(\knt{i_k,\ell,k}, \knt{i_k+\pdeg{\ell}{k}-j_k+2,\ell,k}\right) \subset \RR\;.
	\end{split}
\end{equation}
With $\mbf{j} = (j_1, j_2, \dots, j_n)$, we will denote the set of all potential extended local knot domains as
\begin{equation}\label{eq:space_of_extended_domains}
	\xtpb{\ell}{\mbf{j}} = \left\{\xgb^{\mbf{j}}_{\mbf{i}} ~:~ 1 \leq i_k \leq  \ndof{\ell}{k}+j_k-1\;,\;k = 1, 2, \dots, n\right\}\;,
\end{equation}

\begin{remark}
	For any practical numerical analysis problem, one would pick spline spaces such that $\ndof{\ell}{k} \geq 2$ for all $k$ and $\ell$, thus making the set $\xtpb{\ell}{\mbf{j}}$ non-empty.
	We therefore assume that $\ndof{\ell}{k} \geq 2$ for the rest of the document.
	However, for completeness we comment here on the special case when $\ndof{\ell}{k} = 1$ for at least one $k$ --- this will mean $\xtpb{\ell}{\mbf{j}}$ is an empty set when $j_k = 0$.
	In such cases, the results of this work still hold but the proofs (which use finite induction on the number of parametric dimensions) would need to be modified to omit the directions where $\ndof{\ell}{k} = 1$ since the inductive argument holds automatically in such directions.
	To keep notation clean from this special case, we thus assume that $\ndof{\ell}{k} \geq 2$ henceforth.
\end{remark}

\noindent
\emph{\textbf{Part 1: Decomposing $\Omega_{\ell+1}$ into small subdomains}}\\
\noindent
The proof will utilize a decomposition of $\Omega_{\ell+1}$ into overlapping \textextendedknotdomains; so we begin by defining certain index sets associated to those \textextendedknotdomains, among other notation.

For any $\zerob_{\overline{\mbf{i}}} \in \tpb{\ell}{\mbf{0}}$, define $\itpb{\ell}{\mbf{j}}(\overline{\mbf{i}})$ as the set of indices of $\mbf{j}$-\textextendedknotdomains~ that are index-space neighbours of $\zerob_{\overline{\mbf{i}}}$,
\begin{equation}
	\itpb{\ell}{\mbf{j}}(\overline{\mbf{i}}) :=
	\left\{
	{\mbf{i}}~:~
	j_k-1 \leq {i}_k-\overline{i}_k \leq 0~~,\;k = 1, \dots, n
	\right\}\;.
\end{equation}
We use the convention that, if $\mbf{i} \in \itpb{\ell}{\mbf{j}}(\overline{\mbf{i}})$ is such that any $i_k \leq 0$ or $i_k \geq \ndof{\ell}{k}+j_k-1$, then $\xgb_{{\mbf{i}}}^\mbf{j} := \emptyset$.
Then, if ${\mbf{i}} \in \itpb{\ell}{\mbf{j}}(\overline{\mbf{i}})$ then $\xgb_{{\mbf{i}}}^\mbf{j} \supseteq \supp(\zerob_{\overline{\mbf{i}}})$ for any $\mbf{j}$; the converse is not true in general.
Moreover, $\itpb{\ell}{\mbf{0}}(\overline{\mbf{i}})$ has cardinality $\leq 2^n$ (the inequality holding only when $\overline{\mbf{i}}$ corresponds to a boundary-adjacent zero form B-spline) while $\itpb{\ell}{\mbf{1}}(\overline{\mbf{i}})$ has cardinality $1$.

Denote with $\itpb{\ell,\ell+1}{\mbf{j}}$ the union of such index-space neighbours for all $\zerob_{\overline{\mbf{i}}} \in \tpb{\ell,\ell+1}{\mbf{0}}$, i.e.,
\begin{equation}
	\itpb{\ell,\ell+1}{\mbf{j}} := \bigcup_{\zerob_{\overline{\mbf{i}}}\in \tpb{\ell,\ell+1}{\mbf{0}}}\itpb{\ell}{\mbf{j}}(\overline{\mbf{i}})\;.
\end{equation}
Next, if ${\mbf{i}}$ is the index of an index-space neighbour, define $\deriv{\ell,\ell+1}^\mbf{j}({\mbf{i}})$ as the set of $\mbf{1}$-extended knot domains that are supported on $\Omega_{\ell+1}$ as well as the given index-space neighbour,
\begin{equation}
	\deriv{\ell,\ell+1}^\mbf{j}({\mbf{i}}) := 
	\left\{
	\overline{\mbf{i}}~:~
	\xgb_\mbf{\overline{\mbf{i}}}^\mbf{1} \subset \Omega_{\ell+1}
	\text{~~and~~}
	\forall k~,~
	0 \leq \overline{i}_k-{i}_k \leq 1 - j_k
	\right\}\;.
\end{equation}
Observe that $\deriv{\ell,\ell+1}^\mbf{1}({\mbf{i}})$ has cardinality $\leq 1$ while $\deriv{\ell,\ell+1}^\mbf{0}({\mbf{i}})$ has cardinality $\leq 2^n$.

\begin{remark}
	Since each $\mbf{1}$-extended knot domain can be interpreted as the support of a zero form B-spline, notice that $\itpb{\ell,\ell+1}{\mbf{j}}$ and $\deriv{\ell,\ell+1}^\mbf{j}({\mbf{i}})$ are sort of ``inverses'' of each other -- the former contains the indices of all $\mbf{j}$-extended knot domains that are index-space neighbours of zero form B-splines supported on $\Omega_{\ell+1}$, while the latter contains the indices of all zero form B-splines supported on $\Omega_{\ell+1}$ for which a given $\mbf{j}$-extended knot domain is an index-space neighbour.
	This observation is encapsulated in Lemma \ref{lem:index_inequalities} below.
\end{remark}

\begin{lemma}\label{lem:index_inequalities}
	Let $\mbf{i}$ and $\overline{\mbf{i}}$ be such that, for all $k$,
	\begin{equation*}
		j_k-1 \leq {i}_k - \overline{i}_k \leq 0\;.
	\end{equation*}
	If $\zerob_{\overline{\mbf{i}}} \in \tpb{\ell,\ell+1}{\mbf{0}}$, then $\mbf{i} \in \itpb{\ell,\ell+1}{\mbf{j}}$ and $\overline{\mbf{i}} \in \deriv{\ell,\ell+1}^\mbf{j}({\mbf{i}})$.
\end{lemma}
\begin{proof}
	By definition, $\xgb_{{\mbf{i}}}^\mbf{j}$ is an index space neighbour of $\zerob_{\overline{\mbf{i}}}$ and the latter is supported on $\Omega_{\ell+1}$, thus implying $\mbf{i} \in \itpb{\ell,\ell+1}{\mbf{j}}$.
	Then, combining with the given inequality, we see that, for all $k$,
	\begin{equation*}
		0 \leq \overline{i}_k - {i}_k \leq 1-j_k\;.
	\end{equation*}
	Clearly, $\xgb_{\overline{\mbf{i}}}^\mbf{1} = \supp(\zerob_{\overline{\mbf{i}}}) \subset \Omega_{\ell+1}$, and thus $\overline{\mbf{i}} \in \deriv{\ell,\ell+1}^\mbf{j}({\mbf{i}})$.
\end{proof}

The union of all $\mbf{1}$-\textextendedknotdomains~with indices in $\deriv{\ell,\ell+1}^\mbf{j}({\mbf{i}})$ forms a subset of $\Omega_{\ell+1}$, we denote this subset with $\Omega^\mbf{j}_{\ell+1}({\mbf{i}})$,
\begin{equation}
	\begin{split}
		\Omega_{\ell+1} \supset \Omega_{\ell+1}^\mbf{j}({\mbf{i}}) &:=
		\begin{dcases}
			\bigcup_{\overline{\mbf{i}} \in \deriv{\ell,\ell+1}^\mbf{j}({\mbf{i}})} \xgb_{\overline{\mbf{i}}}^\mbf{1}\;, & {\mbf{i}} \in \itpb{\ell,\ell+1}{\mbf{j}}\;,\\
			\emptyset\;, & \text{otherwise}\;.
		\end{dcases}
	\end{split}
\end{equation}
Figure \ref{fig:disconnected_components}, in conjunction with the following results, helps illustrate the motivation for this definition.

\begin{figure}
	\includegraphics[width=1\textwidth]{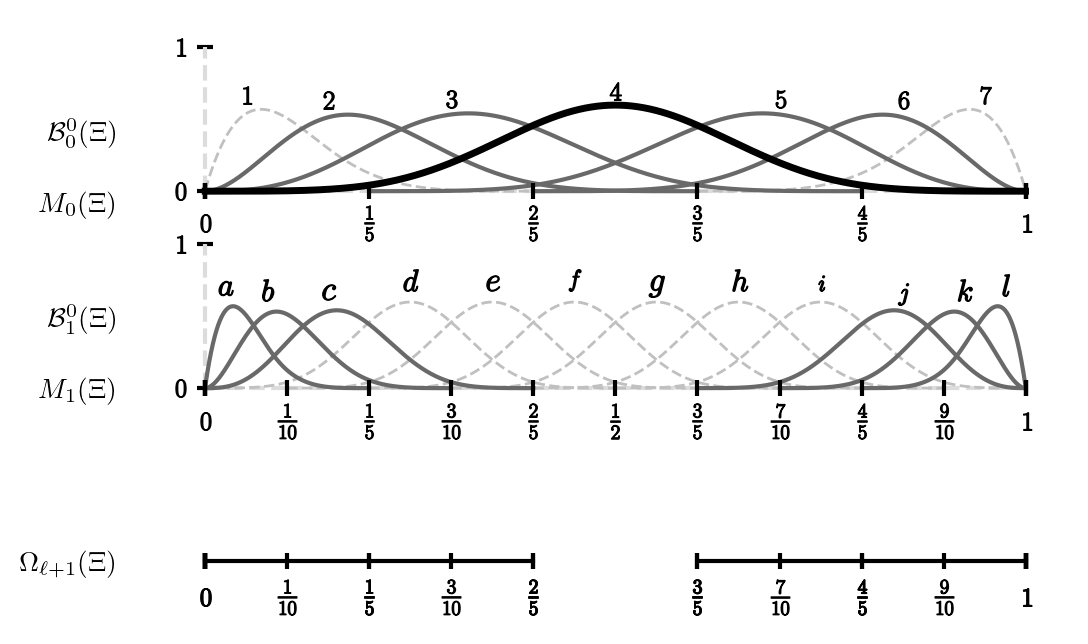}
	\caption{
		Given a knot vector $\kntsSym_0 = \{0,0,0,0,\frac{1}{5},\frac{2}{5},\frac{3}{5},\frac{4}{5},1,1,1,1\}$ associated with univariate splines of degree $p=4$ on level $\ell=0$ and a spline space of the same polynomial degree defined by dyadic refinement of level $\ell=0$ for level $\ell=1$, splines in $\hb{1}{\mbf{j}}$ are shown, with $\Omega_{1}=(0,\frac{3}{10}) \cup (\frac{3}{5},1)$.
	Under this refinement, the domain $\supp(\alpha_4) \cap \Omega_{\ell+1}$ is composed of two disconnected components.
	Nonetheless, based on its definition, $\Omega^j_1(4) = \emptyset.$
	Indeed, the definition of $\deriv{\ell,\ell+1}^j(i)$ assures that  $\Omega^j_1(i)$ is always either empty or a topological ball.
	This result will help us build the hierarchical B-spline complex by piecing together complexes defined on such subdomains.}\label{fig:disconnected_components}
\end{figure}

\begin{corollary}\label{cor:ball_definition}
	If ${\mbf{i}} \in \itpb{\ell,\ell+1}{\mbf{j}}$ then $\Omega_{\ell+1}^\mbf{j}({\mbf{i}})$ is non-empty for all $\mbf{j}$.
\end{corollary}
\begin{proof}
	Since ${\mbf{i}} \in \itpb{\ell,\ell+1}{\mbf{j}}$, there is a 0-form B-spline $\zerob_{\overline{\mbf{i}}}$ supported on $\xgb_{{\mbf{i}}}^\mbf{j} \cap \Omega_{\ell+1}$ such that, for all $k$,
	\begin{equation*}
		j_k-1 \leq {i}_k - \overline{i}_k \leq 0\;.
	\end{equation*}
	Then, from Lemma \ref{lem:index_inequalities}, $\Omega_{\ell+1}^\mbf{j}({\mbf{i}}) \supset \xgb_{\overline{\mbf{i}}}^\mbf{1}$.
\end{proof}

\begin{lemma}\label{lem:partition}
	The domain $\Omega_{\ell+1}$ can be constructed by taking the union of $\Omega_{\ell+1}^\mbf{j}({\mbf{i}})$ over all ${\mbf{i}} \in \itpb{\ell,\ell+1}{\mbf{j}}$ for any $\mbf{j}$, i.e.,
	\begin{equation}
		\Omega_{\ell+1} = \bigcup_{{\mbf{i}} \in \itpb{\ell,\ell+1}{\mbf{j}}}\Omega_{\ell+1}^\mbf{j}({\mbf{i}})\;.
	\end{equation}
\end{lemma}
\begin{proof}
	The claim follows from Lemma \ref{lem:index_inequalities} and Corollary \ref{cor:ball_definition}.
\end{proof}

\begin{proposition}\label{prop:is_ball}
	If ${\mbf{i}} \in \itpb{\ell,\ell+1}{\mbf{j}}$ then $\Omega_{\ell+1}^\mbf{j}({\mbf{i}})$ is a ball.
\end{proposition}
\begin{proof}
	As ${\mbf{i}} \in \itpb{\ell,\ell+1}{\mbf{j}}$, $\Omega_{\ell+1}^\mbf{j}({\mbf{i}})$ is not empty.
	In fact, by definition, $\Omega_{\ell+1}^\mbf{j}({\mbf{i}})$ is equal to the union of at most $2^n$ $\mbf{1}$-\textextendedknotdomains.
	Specifically, $\Omega_{\ell+1}^\mbf{j}({\mbf{i}})$ is the union of one or more members of the set
	\begin{equation*}
		\left\{
		\bigtimes_{k=1}^n  \xknts{\overline{i}_k}{1}{\ell,k}~:~\forall k~,~0 \leq \overline{i}_k - {i}_k \leq 1
		\right\}\;.
	\end{equation*}
	Observe that all members of the above set are $n$-dimensional hypercubes (and thus are all convex) and contain $\supp(\threeb)$, where $\threeb$ is an $n$-form B-spline defined as
	\begin{equation*}
		\threeb := \bigtimes_{k=1}^n \lknts{{i}_k+1}{1}{\ell,k}\;.
	\end{equation*}
	Therefore, $\Omega_{\ell+1}^\mbf{j}({\mbf{i}})$ is open and star-shaped with respect to any point in $\supp(\threeb)$.
	As a result \cite{Ferus:2008}, $\Omega_{\ell+1}^\mbf{j}(\mbf{i})$ is an open ball.
\end{proof}

~\\

\noindent
\emph{\textbf{Part 2: Cohomologically equivalent spline complexes on the subdomains}}\\
\noindent
We will now show that not only are $\Omega_{\ell+1}^\mbf{j}({\mbf{i}})$ nice, topologically trivial domains, the level $\ell$ and $\ell+1$ splines supported on them form cohomologically equivalent subcomplexes.

\begin{lemma}\label{lem:g_ball_definition}
	If ${\mbf{i}} \in \itpb{\ell,\ell+1}{\mbf{j}}$ then $\gmesh{\ell+s,\ell+1}\big(A\big)$, $s = 0, 1$, is non-empty for $A = \Omega_{\ell+1}^\mbf{j}({\mbf{i}})$.
\end{lemma}
\begin{proof}
	The claim follows from Corollary \ref{cor:ball_definition}.
\end{proof} 
\begin{lemma}\label{lem:space_containment}
	If $B \subset A$, then
	\begin{equation*}
		\begin{split}
			\tpb{\ell+s,\ell+1}{\mbf{j}}(B) &\subset \tpb{\ell+s,\ell+1}{\mbf{j}}(A)\;,\\
			\tpbs{\ell+s,\ell+1}{\mbf{j}}(B) &\subset \tpbs{\ell+s,\ell+1}{\mbf{j}}(A)\;,\\
			\gmesh{\ell+s,\ell+1}(B) &\subset \gmesh{\ell+s,\ell+1}(A)\;.
		\end{split}
	\end{equation*}
\end{lemma}
\begin{proof}
	The claim follows from the definitions of the three objects.
\end{proof}

\begin{proposition}\label{prop:g_is_ball}
	If ${\mbf{i}} \in \itpb{\ell,\ell+1}{\mbf{j}}$ then $\gmesh{\ell+s,\ell+1}\big(A\big)$, $s = 0, 1$, is also a ball for $A = \Omega_{\ell+1}^\mbf{j}({\mbf{i}})$.
\end{proposition}
\begin{proof}
	We proceed as in Proposition \ref{prop:is_ball}.
	Recall from that proof that $A$ is the union of at most $2^n$ $\mbf{1}$-\textextendedknotdomain{s}, and note that each $\mbf{1}$-\textextendedknotdomain~contains a zero form whose support is the same as that of the \textextendedknotdomain.
	Then, for each such zero form $\zerob$, and with $\threeb$  defined to be the $n$-form B-spline used in the proof of Proposition  \ref{prop:is_ball}, $\supp(\threeb) \subset \supp(\zerob) \subset A$.
	Observe that each $\gmesh{\ell+s,\ell+1}\big(\supp(\zerob)\big)$ is an $n$-dimensional hypercube (and thus convex) and, from Lemma \ref{lem:space_containment}, contains $\gmesh{\ell+s,\ell+1}\big(\supp(\threeb)\big)$.
	Therefore, as before, $\gmesh{\ell+s,\ell+1}(A)$ is open and star-shaped with respect to any point in $\gmesh{\ell+s,\ell+1}\big(\supp(\threeb)\big)$, and is thus an open ball.
\end{proof}

Using the above and given a domain $A$, we now define spline complex $\tpbs{\ell+s,\ell+1}{}(A)$, $s = 0, 1$, as
\begin{equation}
	\begin{tikzcd}[column sep=0.28in]
		\tpbs{\ell+s,\ell+1}{}(A)~:~\tpbs{\ell+s,\ell+1}{0}(A) \arrow{r}{d} & \tpbs{\ell+s,\ell+1}{1}(A) \arrow{r}{d} & \cdots \arrow{r}{d}  & \tpbs{\ell+s,\ell+1}{n}(A)\;.
	\end{tikzcd}
\end{equation}
Lemma \ref{lem:complex_equality} shows that, for certain domains $A$, the complex $\tpbs{\ell+s,\ell+1}{}(A)$ is the same as $\tpbs{\ell+s}{}(A)$.
Moreover, Corollary \ref{corol:3d_local_isomorphism} shows that these complexes are exact.
\begin{lemma}\label{lem:complex_equality}
	If $A \subset \Omega_{\ell+1}$ is a union of supports of a subset of $\tpb{\ell,\ell+1}{\mbf{0}}$ then, for all $j$,
	\begin{equation}
		\tpbs{\ell+s}{j}(A) = \tpbs{\ell+s,\ell+1}{j}(A)\;.
	\end{equation}
\end{lemma}
\begin{proof}
	It is clear that $\tpbs{\ell+s,\ell+1}{j}(A) \subset \tpbs{\ell+s}{j}(A)$ from the definitions.
	For the inclusion in the other direction, notice that if there is a B-spline $\phi \in \tpbs{\ell+s}{j}(A)$, then there exists a $\mbf{j}$, $|\mbf{j}|=j$, such that $\phi \in \tpb{\ell+s}{\mbf{j}}(A)$.
	However, by the assumption on $A$, there exists an $\zerob \in \tpb{\ell,\ell+1}{\mbf{0}}(A)$ such that $\supp(\phi) \subset \supp(\zerob)$.
	Then, by definition, $\phi \in \tpbs{\ell+s,\ell+1}{j}(A)$.
\end{proof}

When $\tpbs{\ell+s,\ell+1}{}(A) = \tpbs{\ell+s}{}(A)$ for some choice of $A$, then we will denote the $k$-cohomology of both $\tpbs{\ell+s,\ell+1}{}(A)$ and $\tpbs{\ell+s}{}(A)$ with $H^{(j)}_{\ell+s}(A)$.
One such choice is when $A$ is a union of supports of splines in $\tpb{\ell,\ell+1}{\mbf{0}}$, as in Lemma \ref{lem:complex_equality}.

\begin{corollary}\label{corol:3d_local_isomorphism}
	For any ${\mbf{i}} \in \itpb{\ell,\ell+1}{\mbf{j}}$ and $A = \Omega_{\ell+1}^\mbf{j}({\mbf{i}})$, $\tpbs{\ell+s,\ell+1}{}(A)$ is exact, $s = 0, 1$,
	i.e.,
	\begin{equation}
		H^{(j)}_\ell(A) = H^{(j)}_{\ell+1}(A) = 
		\begin{dcases}
			0 \;, & j = 0, 1, \dots, n-1\;,\\
			\RR \;, & j = n\;.
		\end{dcases}
	\end{equation}
\end{corollary}
\begin{proof}
	The proof follows the same argument as in \cite[Corollary 5.12]{Evans19} and uses standard results from homology and cohomology.
	From Propositions \ref{prop:is_ball} and \ref{prop:g_is_ball} and direct computation \cite[Chapter 2]{Hatcher},
	\begin{equation*}
		H_{(n-j)}(A) \approx H_{(n-j)}\big(\gmesh{\ell,\ell+1}(A)\big) \approx
		H_{(n-j)}\big(\gmesh{\ell+1,\ell+1}(A)\big) \approx 
		\begin{dcases}
			\RR\;,& j = n\;,\\
			0\;, & \text{otherwise}\;.
		\end{dcases}
	\end{equation*}
	By the correspondence between the tensor-product splines and their Greville grids \cite{Evans19}, we also have for $s = 0, 1$,
	\begin{equation*}
		H^{(j)}_{\ell+s}(A) \approx H^{(j)}\big(\gmesh{\ell+s,\ell+1}(A),\partial\gmesh{\ell+s,\ell+1}(A)\big)\;,\;s = 0, 1\;.
	\end{equation*}
	Then, the claim follows from the the Lefschetz duality theorem \cite[Theorem 3.43]{Hatcher} which states that
	\begin{equation*}
		H^{(j)}\big(\gmesh{\ell+s,\ell+1}(A),\partial\gmesh{\ell+s,\ell+1}(A)\big) \approx H_{(n-j)}\big(\gmesh{\ell,\ell+1}(A)\big)\;,\;s = 0, 1\;.
	\end{equation*}	
\end{proof}

\begin{corollary}\label{corol:injection_cohomology}
	Let $A = \Omega_{\ell+1}^\mbf{j}({\mbf{i}})$ for any multiindex ${\mbf{i}}$.
	Then, the inclusion operator $\splinelevelinclusion{\ell,\ell+1}:\tpbs{\ell,\ell+1}{}(A)\rightarrow\tpbs{\ell+1,\ell+1}{}(A)$ induces an isomorphism on cohomology.
\end{corollary}
\begin{proof}
	For ${\mbf{i}}  \in \itpb{\ell,\ell+1}{\mbf{j}}$, this follows directly from Corollary \ref{corol:3d_local_isomorphism} since the complexes are exact.
	Otherwise, if ${\mbf{i}}  \notin \itpb{\ell,\ell+1}{\mbf{j}}$, then $A = \emptyset$ and the result holds trivially.	
\end{proof}

\begin{remark}
	As was done in \cite{Evans19}, it is also possible to show Corollary \ref{corol:injection_cohomology} by defining projection operators $\hat{\Pi}^j:\tpbs{\ell+1,\ell+1}{j}(A) \rightarrow \tpbs{\ell,\ell+1}{j}(A)$ via the following problem,
	\begin{subequations}
	\begin{align}
		\Big(\hat{\Pi}^j\gf,d^{j-1} g\Big) &= (\gf, d^{j-1}g)\;,\quad \forall g \in \tpbs{\ell,\ell+1}{j-1}(A)\;,\\
		\Big(d^j\hat{\Pi}^j\gf,d^j g\Big) &= (d^j\gf, d^j g)\;,\quad \forall g \in \tpbs{\ell,\ell+1}{j}(A)\;,\;j = 0, \dots, n-1\;,\\
		\int_{A}\hat{\Pi}^n\gf &= \int_{A}\gf\;.
	\end{align}
	\end{subequations}
	In the above system, the first sub-equation helps determine the part of $\hat{\Pi}^jf$ that is exact.
	Then, from Propositions \ref{prop:is_ball} and \ref{prop:g_is_ball}, the remaining part of $\hat{\Pi}^jf$ that needs to be determined is co-exact for $j = 0, \dots, n-1$, and harmonic for $j = n$; these are respectively determined by the second (recall that we are working with homogeneous boundary conditions here) and third sub-equations.
\end{remark}
~\\

\noindent
\emph{\textbf{Part 3: Piecing together the subdomain complexes}}\\
\noindent
We now piece-together the subdomain complexes to show that the hierarchical spline complex is exact, given that an additional assumption is satisfied.

\begin{definition}[An $(n-1,\ell+1)$-intersection]\label{def:n-1-intersection}
	Let $\zerob_{\overline{\mbf{i}}}, \zerob_{\overline{\mbf{i}}+\Delta\overline{\mbf{i}}} \in \tpb{\ell}{\mbf{0}}(\Omega_{\ell+1})$ and, without loss of generality, assume that $\Delta\overline{\mbf{i}}$ is component-wise non-negative.
	We say that $\zerob_{\overline{\mbf{i}}}$ and $\zerob_{\overline{\mbf{i}}+\Delta\overline{\mbf{i}}}$ share an $(n-1,\ell+1)$-intersection if $\exists \overline{\mbf{i}}', k_0$ such that
	\begin{subequations}
	\begin{align}
		\csupp(\zerob_{\overline{\mbf{i}}}) \cap \csupp(\zerob_{\overline{\mbf{i}}+\Delta\overline{\mbf{i}}}) &\supseteq \bigtimes_{k=1}^n I_k\;,\\
		I_k &:= 
			\begin{dcases}
				\left(\knt{\overline{i}'_k,\ell+1,k}, \knt{\overline{i}'_k+\pdeg{\ell+1}{k},\ell+1,k}\right)\;, & k \neq k_0\;,\\
				\{\kntSym_{\overline{i}'_k,\ell+1,k}\}\;, & k = k_0\;.\
			\end{dcases}\;.
	\end{align}
	\end{subequations}
	If needed, we will say that $\zerob_{\overline{\mbf{i}}}$ and  $\zerob_{\overline{\mbf{i}}+\Delta\overline{\mbf{i}}}$ share an $(n-1,\ell+1)$-intersection w.r.t. the direction $k_0$.
\end{definition}

\begin{definition}[A chain]\label{def:chain}
	{Let $\zerob_{\overline{\mbf{i}}}, \zerob_{\overline{\mbf{i}}+\Delta\overline{\mbf{i}}} \in \tpb{\ell,\ell+1}{\mbf{0}}$.
	There is said to be a chain between $\zerob_{\overline{\mbf{i}}}$ and $\zerob_{\overline{\mbf{i}}+\Delta\overline{\mbf{i}}}$ if there is some positive integer $r$ and a set of B-splines $\zerob_{\overline{\mbf{i}} + \Delta\overline{\mbf{i}}_l} \in \tpb{\ell,\ell+1}{\mbf{0}}, l = 0,\dots,r$ such that
	\begin{itemize}
		\item $\Delta\overline{\mbf{i}}_0 = \mbf{0}$, $\Delta\overline{\mbf{i}}_r = \Delta\overline{\mbf{i}}$;
		\item $\Delta\overline{\mbf{i}}_l - \Delta\overline{\mbf{i}}_{l-1}$ is zero in each component but one, and the sum of all components equal to $\pm1$ for all $l = 1, \dots, r$.
	\end{itemize}
	}
\end{definition}

\begin{definition}[A shortest chain]\label{def:shortest_path}
	Let $\zerob_{\overline{\mbf{i}}}, \zerob_{\overline{\mbf{i}}+\Delta\overline{\mbf{i}}} \in \tpb{\ell,\ell+1}{\mbf{0}}$ and, without loss of generality, assume that $\Delta\overline{\mbf{i}}$ is component-wise non-negative.
	A chain between $\zerob_{\overline{\mbf{i}}}$ and $\zerob_{\overline{\mbf{i}}+\Delta\overline{\mbf{i}}}$ is a shortest chain if 
	\begin{itemize}
		\item$r := \sum_k \Delta \overline{i}_k$;
		\item $\Delta\overline{\mbf{i}}_l - \Delta\overline{\mbf{i}}_{l-1}$ is component-wise non-negative, with the sum of all components being equal to $1$ for all $l = 1, \dots, r$.
	\end{itemize}
\end{definition}

\begin{lemma}\label{lem:shortest_chain_intersection}
	Let there be a shortest chain between $\zerob_{\overline{\mbf{i}}}, \zerob_{\overline{\mbf{i}}+\Delta\overline{\mbf{i}}} \in \tpb{\ell}{\mbf{0}}(\Omega_{\ell+1})$.
	Then the closure of the support of any B-spline in the shortest chain contains $\csupp(\zerob_{\overline{\mbf{i}}}) \cap \csupp(\zerob_{\overline{\mbf{i}}+\Delta\overline{\mbf{i}}})$.
\end{lemma}

\begin{assumption}\label{assume:shortest_path}
	Let $\zerob_{\overline{\mbf{i}}}, \zerob_{\overline{\mbf{i}}+\Delta\overline{\mbf{i}}} \in \tpb{\ell}{\mbf{0}}(\Omega_{\ell+1})$.
	If $\zerob_{\overline{\mbf{i}}}$ and $\zerob_{\overline{\mbf{i}}+\Delta\overline{\mbf{i}}}$ share an $(n-1,\ell+1)$-intersection, then there exists a shortest chain between them.
\end{assumption}

Figure \ref{fig:shortest_chain} illustrates the concept of $(\ndim-1,\ell+1)$ intersections, chains, and shortest chains between two $0$-forms under different refinement scenarios.
Of these configurations, only that of subfigure \ref{fig:shortest_chain}(b) satisfies the conditions of Assumption \ref{assume:shortest_path}.

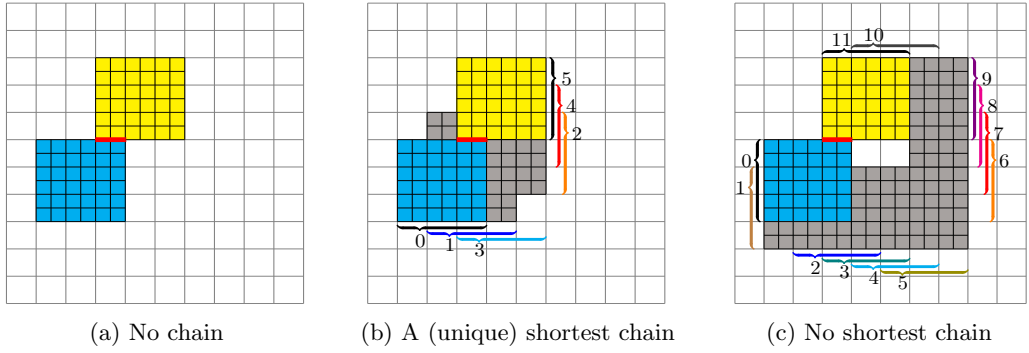
\begin{figure}
    \begin{subfigure}{0.32\textwidth}
    	\centering
    		\resizebox{\xBoxDim}{\yBoxDim}{%	
               	\begin{tikzpicture}
                		\mymanualgrid{0,...,10}{-2,...,9}{0.25}{gray}{white};
    			\mymanualgrid{1,1.5,2,2.5,3,3.5,4}{1,1.5,2,2.5,3,3.5,4}{1}{black}{cyan}
    			\mymanualgrid{3,3.5,4,4.5,5,5.5,6}{4,4.5,5,5.5,6,6.5,7}{1}{black}{yellow}
    			\mymanualgrid{3,4}{4}{5}{red}{red}
               	 \end{tikzpicture}
               	 }
	 \caption{No chain}
    \end{subfigure}
    \;
    \begin{subfigure}{0.32\textwidth}
    	\centering
    		\resizebox{\xBoxDim}{\yBoxDim}{%	
               	\begin{tikzpicture}
                		\mymanualgrid{0,...,10}{-2,...,9}{0.25}{gray}{white};
    			\mymanualgrid{1,1.5,2,2.5,3,3.5,4}{1,1.5,2,2.5,3,3.5,4}{1}{black}{cyan}
    			\mymanualgrid{3,3.5,4,4.5,5,5.5,6}{4,4.5,5,5.5,6,6.5,7}{1}{black}{yellow}
			\mymanualgrid{4,4.5,5}{1,1.5,2,2.5,3,3.5,4}{1}{black}{lightgray}
			\mymanualgrid{5,5.5,6}{2,2.5,3,3.5,4}{1}{black}{lightgray}
			\mymanualgrid{2,2.5,3}{4,4.5,5}{1}{black}{lightgray}			
    			\mymanualgrid{3,4}{4}{5}{red}{red}
			
			\draw [decorate, decoration = {calligraphic brace, raise=4pt, amplitude=5pt, aspect=0.75}, line width=3, black] (4,1) --  (1,1);
			\draw [decorate, decoration = {calligraphic brace, raise=10pt, amplitude=5pt, aspect=0.75}, line width=3, pen colour={blue}] (5,1) --  (2,1);
			\draw [decorate, decoration = {calligraphic brace, raise=16pt, amplitude=5pt, aspect=0.25}, line width=3, pen colour={orange}] (6,5) --  (6,2);
			\draw [decorate, decoration = {calligraphic brace, raise=16pt, amplitude=5pt, aspect=0.75}, line width=3, pen colour={cyan}] (6,1) --  (3,1);
			\draw [decorate, decoration = {calligraphic brace, raise=10pt, amplitude=5pt, aspect=0.25}, line width=3, pen colour={red}] (6,6) --  (6,3);
			\draw [decorate, decoration = {calligraphic brace, raise=4pt, amplitude=5pt, aspect=0.25}, line width=3, pen colour={black}] (6,7) --  (6,4);
			
			\node[scale=2] at (1.77,0.3) {0};
			\node[scale=2] at (2.77,0.15) {1};
			\node[scale=2] at (3.77,-0.05) {3};

			\node[scale=2] at (7.05,4.23) {2};
			\node[scale=2] at (6.85,5.23) {4};
			\node[scale=2] at (6.7,6.23) {5};

               	 \end{tikzpicture}
               	 }
	 \caption{A (unique) shortest chain}
    \end{subfigure}
    \;
    \begin{subfigure}{0.32\textwidth}
    	\centering
    		\resizebox{\xBoxDim}{\yBoxDim}{%	
               	\begin{tikzpicture}
                		\mymanualgrid{0,...,10}{-2,...,9}{0.25}{gray}{white};
    			\mymanualgrid{1,1.5,2,2.5,3,3.5,4}{1,1.5,2,2.5,3,3.5,4}{1}{black}{cyan}
    			\mymanualgrid{3,3.5,4,4.5,5,5.5,6}{4,4.5,5,5.5,6,6.5,7}{1}{black}{yellow}
			\mymanualgrid{1,1.5,2,2.5,3,3.5,4}{0,0.5,1}{1}{black}{lightgray}
			\mymanualgrid{4,4.5,5,5.5,6,6.5,7,7.5,8}{0,0.5,1,1.5,2,2.5,3}{1}{black}{lightgray}
			\mymanualgrid{6,6.5,7,7.5,8}{3,3.5,4,4.5,5,5.5,6,6.5,7}{1}{black}{lightgray}			
    			\mymanualgrid{3,4}{4}{5}{red}{red}
			
			\draw [decorate, decoration = {calligraphic brace, raise=4pt, amplitude=5pt, aspect=0.75}, line width=3, pen colour={black}] (1,1) --  (1,4);
			\draw [decorate, decoration = {calligraphic brace, raise=10pt, amplitude=5pt, aspect=0.75}, line width=3, pen colour={brown}] (1,0) --  (1,3);
%			\draw [decorate, decoration = {calligraphic brace, raise=4pt, amplitude=5pt, aspect=0.75}, line width=3, pen colour={violet}] (4,0) --  (1,0);
			\draw [decorate, decoration = {calligraphic brace, raise=4pt, amplitude=5pt, aspect=0.75}, line width=3, pen colour={blue}] (5,0) --  (2,0);
			\draw [decorate, decoration = {calligraphic brace, raise=10pt, amplitude=5pt, aspect=0.75}, line width=3, pen colour={teal}] (6,0) --  (3,0);
			\draw [decorate, decoration = {calligraphic brace, raise=16pt, amplitude=5pt, aspect=0.75}, line width=3, pen colour={cyan}] (7,0) --  (4,0);
			\draw [decorate, decoration = {calligraphic brace, raise=22pt, amplitude=5pt, aspect=0.75}, line width=3, pen colour={olive}] (8,0) --  (5,0);
%			\draw [decorate, decoration = {calligraphic brace, raise=4pt, amplitude=5pt, aspect=0.75}, line width=3, pen colour={brown}] (8,3) --  (8,0);
			\draw [decorate, decoration = {calligraphic brace, raise=22pt, amplitude=5pt, aspect=0.25}, line width=3, pen colour={orange}] (8,4) --  (8,1);
			\draw [decorate, decoration = {calligraphic brace, raise=16pt, amplitude=5pt, aspect=0.25}, line width=3, pen colour={red}] (8,5) --  (8,2);
			\draw [decorate, decoration = {calligraphic brace, raise=10pt, amplitude=5pt, aspect=0.25}, line width=3, pen colour={magenta}] (8,6) --  (8,3);
			\draw [decorate, decoration = {calligraphic brace, raise=4pt, amplitude=5pt, aspect=0.25}, line width=3, pen colour={violet}] (8,7) --  (8,4);
%			\draw [decorate, decoration = {calligraphic brace, raise=16pt, amplitude=5pt, aspect=0.25}, line width=3, pen colour={gray}] (5,7) --  (8,7);
			\draw [decorate, decoration = {calligraphic brace, raise=10pt, amplitude=5pt, aspect=0.25}, line width=3, pen colour={darkgray}] (4,7) --  (7,7);
			\draw [decorate, decoration = {calligraphic brace, raise=4pt, amplitude=5pt, aspect=0.25}, line width=3, pen colour={black}] (3,7) --  (6,7);

			\node[scale=2] at (0.4,3.23) {0};
			\node[scale=2] at (0.2,2.23) {1};
			
			\node[scale=2] at (2.77,-0.7) {2};
			\node[scale=2] at (3.77,-0.85) {3};
			\node[scale=2] at (4.77,-1.05) {4};
			\node[scale=2] at (5.77,-1.25) {5};

			\node[scale=2] at (9.25,3.23) {6};
			\node[scale=2] at (9.05,4.23) {7};
			\node[scale=2] at (8.85,5.23) {8};
			\node[scale=2] at (8.65,6.23) {9};
			
			\node[scale=2] at (4.77,7.85) {10};
			\node[scale=2] at (3.67,7.65) {11};
               	 \end{tikzpicture}
               	 }
	 \caption{No shortest chain}
    \end{subfigure}
    \caption{
   	In the above figures, we consider two levels of a maximally regular hierarchical spline space defined using $\pdeg{\ell}{k} = 2$ for all $\ell$ and $k$.
   	The shaded cells (of any colour) constitute $\Omega_{\ell+1}$.
   	In particular, the two biquadratic 0-form B-splines $\zerob_{\overline{\mbf{i}}},~\zerob_{\overline{\mbf{i}}+(2,3)} \in \tpb{\ell,\ell+1}{\mbf{0}}$ share an $(\ndim-1,\ell+1)$ intersection, depicted in red.
    Here, $\zerob_{\overline{\mbf{i}}}$ is on the lower-left and its support is shaded blue, while $\zerob_{\overline{\mbf{i}}+(2,3)} \in \tpb{\ell,\ell+1}{\mbf{0}}$ is on the upper-right and its support is shaded yellow; the rest of $\Omega_{\ell+1}$ is shown in grey.
    In figure (a), there are no other B-splines in $\tpb{\ell,\ell+1}{\mbf{0}}$, and thus no chain of indices traversing from one to the other.
    In figure (b), the refinement pattern supports a shortest chain between the $0$-form B-splines, with numbers indicating the magnitude of $\Delta\overline{\mbf{i}}_l$, as in Definition \ref{def:shortest_path}.
    On the right, the refinement supports a chain, with numbers also indicating the number $l$ in $\Delta\overline{\mbf{i}}_l$, but the chain is not a shortest chain.
    }
    \label{fig:shortest_chain}
\end{figure}

For this final part of the proof and for given $0 \leq \ell \leq L-1$, define the following ``slices'' of $\Omega_{\ell+1}$, where $\mbf{i}$, $\mbf{j}$, $\mbf{r}$ and $\mbf{m}$ are $n$-tuples,
\begin{equation}
	\begin{split}
		S_\mbf{i}^\mbf{j}[\mbf{m}] &:= \bigcup_{r_1=0}^{m_1}\cdots\bigcup_{r_n=0}^{m_n}\Omega^{\mbf{j}}_{\ell+1}(\mbf{i}+\mbf{r})\;,
	\end{split}
\end{equation}
and recall that $\Omega^{\mbf{j}}_{\ell+1}(\mbf{i}+\mbf{r}) = \emptyset$ by definition if $\mbf{i}+\mbf{r} \notin \itpb{\ell,\ell+1}{\mbf{j}}$.
Note that, from Lemma \ref{lem:partition}, if $m_i = \ndof{\ell}{i}$ for all $i$ then
\begin{equation}
	S_\mbf{1}^\mbf{j}[\mbf{m}]
	=  \bigcup_{r_1=0}^{\ndof{\ell}{1}}\cdots\bigcup_{r_n=0}^{\ndof{\ell}{n}}\Omega^\mbf{j}_{\ell+1}(\mbf{1}+\mbf{r})
	= \Omega_{\ell+1}\;.
	\label{eq:sliced_partition}
\end{equation}

\begin{lemma}\label{lem:unions_intersections_domains}
	Let Assumption \ref{assume:shortest_path} hold.
	For a fixed $1 \leq k \leq n$, let $\delta_k$ be an $n$-tuple with the value $1$ in index $k$ and $0$ elsewhere, and let $\mbf{i}$, $\mbf{j}$ and $\mbf{m}$ be given.
	Then, with $\overline{\mbf{r}} := (r_1,\dots,r_{k-1},0,r_{k+1},\dots,r_n)$ and with
	\begin{equation*}
		\begin{split}
			A_\mbf{i}^\mbf{j} &:= S_\mbf{i}^\mbf{j}[\mbf{m}]\;,\\
			B^\mbf{j}_\mbf{i} &:= \bigcup_{r_1=0}^{m_1}\cdots\bigcup_{r_{k-1}=0}^{m_{k-1}}\bigcup_{r_{k+1}=0}^{m_{k+1}}\cdots\bigcup_{r_n=0}^{m_n}\Omega^\mbf{j}_{\ell+1}(\mbf{i}+\overline{\mbf{r}}+(m_k+1)\delta_k)\;,\\
			C_\mbf{i}^\mbf{j} &:= S_\mbf{i}^\mbf{j}[\mbf{m}+\delta_k]\;,
		\end{split}
	\end{equation*}
	the following hold,
	{
	\begin{align}
		A_\mbf{i}^\mbf{j} \cap B^\mbf{j}_\mbf{i} \supset B^{\mbf{j}+\delta_k}_{\mbf{i}}\;\quad& \mathrm{ for }\; j_k = 0\;,\label{eq:slice_inclusion}\\
		A_\mbf{i}^\mbf{j} \cup B^\mbf{j}_\mbf{i} = C_\mbf{i}^\mbf{j}\;\quad&  \mathrm{ for }\; j_k = 0,1\;.
	\end{align}
	}
\end{lemma}
\begin{proof}
	Given $r_k$, let $\mbf{r} := \overline{\mbf{r}} + r_k\delta_k$.
	Then, since $\cup_{r_k=0}^{m_k+1} = \left(\cup_{r_k=0}^{m_k}\right) \cup \left(\cup_{r_k=m_k+1}\right)$, the second equality follows from definitions of the domains,
	\begin{equation*}
		\begin{split}
			C_\mbf{i}^\mbf{j} &= \bigcup_{r_1=0}^{m_1}\cdots\bigcup_{r_{k-1}=0}^{m_{k-1}}\bigcup_{r_k=0}^{m_k+1}\bigcup_{r_{k+1}=0}^{m_{k+1}}\cdots\bigcup_{r_n=0}^{m_n}\Omega_{\ell+1}^\mbf{j}(\mbf{i}+\mbf{r})\\
			&= \left(\bigcup_{r_1=0}^{m_1}\cdots\bigcup_{r_{k-1}=0}^{m_{k-1}}\bigcup_{r_k=0}^{m_k}\bigcup_{r_{k+1}=0}^{m_{k+1}}\cdots\bigcup_{r_n=0}^{m_n}\Omega^\mbf{j}_{\ell+1}(\mbf{i}+\mbf{r})\right) \\
			&\quad\quad \bigcup~\left(\bigcup_{r_1=0}^{m_1}\cdots\bigcup_{r_{k-1}=0}^{m_{k-1}}\bigcup_{r_{k+1}=0}^{m_{k+1}}\cdots\bigcup_{r_n=0}^{m_n}\Omega^\mbf{j}_{\ell+1}(\mbf{i}+\overline{\mbf{r}} + (m_k+1)\delta_k)\right)\;,\\
			&= A_\mbf{i}^\mbf{j} \cup B^\mbf{j}_\mbf{i}\;.
		\end{split}
	\end{equation*}
	
	The inclusion $B^{\mbf{j}+\delta_k}_{\mbf{i}} \subset A_\mbf{i}^\mbf{j}$ is clear from the definitions so we only need to show the containment $B^{\mbf{j}+\delta_k}_{\mbf{i}} \subset B_\mbf{i}^\mbf{j}$.
	Consider a \Bezier~element contained in $B^{\mbf{j}+\delta_k}_{\mbf{i}}$.
	Then, there exists a $\mbf{1}$-\textextendedknotdomain~$\xgb^\mbf{1}_{\overline{\mbf{i}}}$ contained in $B^{\mbf{j}+\delta_k}_{\mbf{i}}$ and containing this \Bezier~element, i.e., $\overline{\mbf{i}} \in \deriv{\ell,\ell+1}^{\mbf{j}+\delta_k}(\mbf{i}+\overline{\mbf{r}}+(m_k+1)\delta_k)$ for some $\overline{\mbf{r}}$.
	Thus, $(\overline{\mbf{i}} - \mbf{i} - \overline{\mbf{r}} - (m_k+1)\delta_k)$ is a non-negative $n$-tuple that is component-wise less than or equal to $(\mbf{1} - \mbf{j} - \delta_k)$, and thus component-wise less than or equal to $(\mbf{1} - \mbf{j})$.
	Therefore, $\overline{\mbf{i}} \in \deriv{\ell,\ell+1}^\mbf{j}(\mbf{i}+\overline{\mbf{r}} + (m_k+1)\delta_k)$ and thus $\xgb^\mbf{1}_{\overline{\mbf{i}}} \subset B_\mbf{i}^\mbf{j}$.
\end{proof}

\begin{remark}
	{The relationship of equation \eqref{eq:slice_inclusion} in Lemma \ref{lem:unions_intersections_domains} may actually be an equality or may be strict, depending on the refinement pattern of $\Omega_{\ell+1}$.
	Figure \ref{fig:unequal_AB_domains} depicts one scenario in which the subset relationship is a proper subset relationship because $A_\mbf{i}^\mbf{j} \cap B^\mbf{j}_\mbf{i}$ is non-empty, but $B^{\mbf{j}+\delta_k}_\mbf{i}$ is.
	If, however, the two refined $0$-form domains of this picture had a shortest chain connecting each other, the subset relationship would actually be an equality.}
\end{remark}
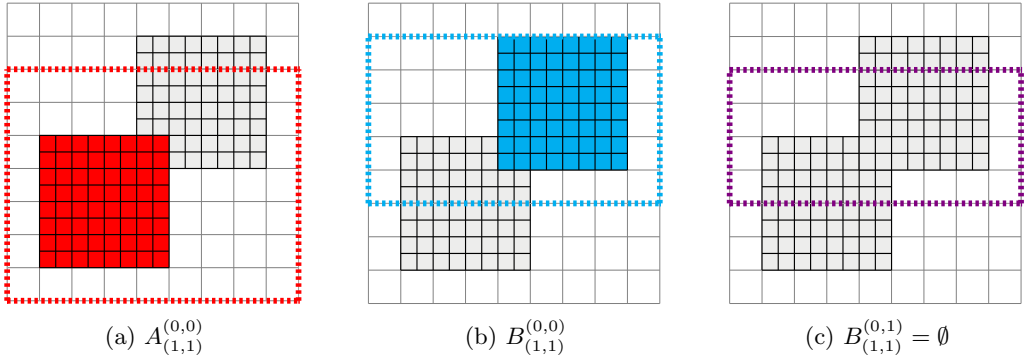
\begin{figure}
    \begin{subfigure}{0.32\textwidth}
    	\centering
    		\resizebox{\xBoxDim}{\yBoxDim}{%	
               	\begin{tikzpicture}
                		\mymanualgrid{0,...,9}{0,...,9}{0.25}{gray}{white};
    			\mymanualgrid{4,4.5,5,5.5,6,6.5,7,7.5,8}{4,4.5,5,5.5,6,6.5,7,7.5,8}{1}{black}{lightgray!15}
    			\mymanualgrid{1,1.5,2,2.5,3,3.5,4,4.5,5}{1,1.5,2,2.5,3,3.5,4,4.5,5}{1}{black}{red}
    			\mymanualgrid{0,9}{0}{5,dashed}{red}{lightgray}
    			\mymanualgrid{0,9}{7}{5,dashed}{red}{lightgray}
    			\mymanualgrid{0}{0,7}{5,dashed}{red}{lightgray}
    			\mymanualgrid{9}{0,7}{5,dashed}{red}{lightgray}
               	 \end{tikzpicture}
               	 }
	 \caption{$A_{(1,1)}^{(0,0)}$}% 
    \end{subfigure}
    \;
    \begin{subfigure}{0.32\textwidth}
    	\centering
    		\resizebox{\xBoxDim}{\yBoxDim}{%	
               	\begin{tikzpicture}
                		\mymanualgrid{0,...,9}{0,...,9}{0.25}{gray}{white};
    			\mymanualgrid{1,1.5,2,2.5,3,3.5,4,4.5,5}{1,1.5,2,2.5,3,3.5,4,4.5,5}{1}{black}{lightgray!15}
    			\mymanualgrid{4,4.5,5,5.5,6,6.5,7,7.5,8}{4,4.5,5,5.5,6,6.5,7,7.5,8}{1}{black}{cyan}
    			\mymanualgrid{0,9}{3}{5,dashed}{cyan}{lightgray}
    			\mymanualgrid{0,9}{8}{5,dashed}{cyan}{lightgray}
    			\mymanualgrid{0}{3,8}{5,dashed}{cyan}{lightgray}
    			\mymanualgrid{9}{3,8}{5,dashed}{cyan}{lightgray}
               	 \end{tikzpicture}
               	 }
	 \caption{$B_{(1,1)}^{(0,0)}$}
    \end{subfigure}
    \;
    \begin{subfigure}{0.32\textwidth}
    	\centering
    		\resizebox{\xBoxDim}{\yBoxDim}{%	
               	\begin{tikzpicture}
                		\mymanualgrid{0,...,9}{0,...,9}{0.25}{gray}{white};
    			\mymanualgrid{1,1.5,2,2.5,3,3.5,4,4.5,5}{1,1.5,2,2.5,3,3.5,4,4.5,5}{1}{black}{lightgray!15}
    			\mymanualgrid{4,4.5,5,5.5,6,6.5,7,7.5,8}{4,4.5,5,5.5,6,6.5,7,7.5,8}{1}{black}{lightgray!15}
    			\mymanualgrid{0,9}{3}{5,dashed}{violet}{lightgray}
    			\mymanualgrid{0,9}{7}{5,dashed}{violet}{lightgray}
    			\mymanualgrid{0}{3,7}{5,dashed}{violet}{lightgray}
    			\mymanualgrid{9}{3,7}{5,dashed}{violet}{lightgray}
               	 \end{tikzpicture}
               	 }
	 \caption{$B_{(1,1)}^{(0,1)} = \emptyset$}
    \end{subfigure}
    \caption{
    	The above figures correspond to a maximally regular hierarchical spline space defined using $\pdeg{\ell}{k} = 3$ for all $\ell$ and $k$.
    	Figures (a), (b) and (c) depict domains $A_\mbf{i}^\mbf{j},$ $B^\mbf{j}_\mbf{i}$, and $B^{\mbf{j}+\delta_k}_\mbf{i}$, respectively, as defined in Lemma \ref{lem:unions_intersections_domains} for $k = 2$.
	    Here, dashed lines indicate the extents of $0$-form B-splines whose support could potentially contribute to the domain, while the shaded cells of the same color correspond to the actual domain.
	    Greyed-out sections are portions of $\Omega_{\ell+1}$ that are not contained in the referenced domain.
	    This refinement pattern obeys all assumptions of this paper, so the inequality of equation \eqref{eq:slice_inclusion} holds and takes the form of $A_\mbf{i}^\mbf{j} \cap B^\mbf{j}_\mbf{i} \neq \emptyset = B^{\mbf{j}+\delta_k}_\mbf{i}$.
    }
    \label{fig:unequal_AB_domains}
\end{figure}

\begin{lemma}\label{lem:intersections}
	Let Assumption \ref{assume:shortest_path} hold.
	For a fixed $1 \leq k \leq n$, let $\delta_k$ be an $n$-tuple with the value $1$ in index $k$ and $0$ elsewhere, and let $\mbf{i}$, $\mbf{j}$ and $\mbf{m}$ be given with $j_k = 0$.
	Then, with $\overline{\mbf{r}} := (r_1,\dots,r_{k-1},0,r_{k+1},\dots,r_n)$ and with
	\begin{equation*}
		\begin{split}
			A_\mbf{i}^\mbf{j} &:= S_\mbf{i}^\mbf{j}[\mbf{m}]\;,\\
			B^\mbf{j}_\mbf{i} &:= \bigcup_{r_1=0}^{m_1}\cdots\bigcup_{r_{k-1}=0}^{m_{k-1}}\bigcup_{r_{k+1}=0}^{m_{k+1}}\cdots\bigcup_{r_n=0}^{m_n}\Omega^\mbf{j}_{\ell+1}(\mbf{i}+\overline{\mbf{r}}+(m_k+1)\delta_k)\;,
		\end{split}
	\end{equation*}
	the following holds for $s = 0, 1$,	
	\begin{align}
		\tpb{\ell+s,\ell+1}{*}(A_\mbf{i}^\mbf{j}) \cap \tpb{\ell+s,\ell+1}{*}(B^\mbf{j}_\mbf{i})
		&= \tpb{\ell+s,\ell+1}{*}(A_\mbf{i}^\mbf{j} \cap B^\mbf{j}_\mbf{i}) = \tpb{\ell+s,\ell+1}{*}(B^{\mbf{j}+\delta_k}_{\mbf{i}})\label{eq:equal_intersections}\;.
	\end{align}
	Here, a superscript of $*$ denotes that the statement is true for all sets of B-splines.
\end{lemma}
\begin{proof}
	If $\gb \in \tpb{\ell+s,\ell+1}{*}(A_\mbf{i}^\mbf{j}) \cap \tpb{\ell+s,\ell+1}{*}(B^\mbf{j}_\mbf{i})$, then by definition there must exist an $n$-form $\threeb \in \tpb{\ell+1,\ell+1}{\mbf{1}}$ and two
	0-form B-splines $\zerob_{\overline{\mbf{i}}},\zerob_{\overline{\mbf{i}}+\Delta\overline{\mbf{i}}} \in \tpb{\ell,\ell+1}{\mbf{0}}$, such that
	\begin{equation*}
		\begin{split}
			&\supp(\threeb) \subset \supp(\gb) \subset \supp(\zerob_{\overline{\mbf{i}}}) \subset A^\mbf{j}_\mbf{i}\;,\\
			&\supp(\threeb) \subset \supp(\gb) \subset \supp(\zerob_{\overline{\mbf{i}}+\Delta\overline{\mbf{i}}}) \subset B^\mbf{j}_\mbf{i}\;.
		\end{split}
	\end{equation*}
	Without loss of generality, assume that $\Delta\overline{\mbf{i}}$ is component-wise non-negative.
	The above implies that $\zerob_{\overline{\mbf{i}}}$ and $\zerob_{\overline{\mbf{i}}+\Delta\overline{\mbf{i}}}$ share an $(n-1,\ell+1)$-intersection and thus, by Assumption \ref{assume:shortest_path}, there exists a shortest chain between them.
	Moreover, from the definitions of $A_\mbf{i}^\mbf{j}$ and $B_\mbf{i}^\mbf{j}$, there exist $\mbf{r}_0$ and $\overline{\mbf{r}}_1$ such that $\overline{\mbf{i}} \in \deriv{\ell,\ell+1}^\mbf{j}({\mbf{i}+\mbf{r}_0})$ and $\overline{\mbf{i}} + \Delta\overline{\mbf{i}} \in \deriv{\ell,\ell+1}^\mbf{j}(\mbf{i}+\overline{\mbf{r}}_1+(m_k+1)\delta_k)$.
	Thus, the following inequalities hold \textit{component-wise},
	\begin{equation*}
		\begin{gathered}
			\mbf{j}-\mbf{1} \leq \mbf{i} + \mbf{r}_0 - \overline{\mbf{i}} \leq \mbf{0}\;,\\
			\mbf{j}-\mbf{1} \leq \mbf{i} + \overline{\mbf{r}}_1 + (m_k+1)\delta_k - \overline{\mbf{i}} - \Delta\overline{\mbf{i}} \leq \mbf{0}\;.
		\end{gathered}
	\end{equation*}
	In particular, the second inequality implies that
	\begin{equation*}
		i_k + (m_k + 1) - \overline{i}_k \leq \Delta \overline{i}_k\;.
	\end{equation*}

	Next, since a shortest chain exists between $\zerob_{\overline{\mbf{i}}}$ and $\zerob_{\overline{\mbf{i}}+\Delta\overline{\mbf{i}}}$, let $\Delta\hat{\mbf{i}}$ be such that it satisfies the following conditions:
	\begin{itemize}
		\item component-wise, $\mbf{0} \leq \Delta\hat{\mbf{i}} \leq \Delta\overline{\mbf{i}}$;
		\item $\Delta\hat{i}_k = i_k + (m_k+1) - \overline{i}_k \leq \Delta\overline{i}_k$;
		\item $\zerob_{\overline{\mbf{i}}+\Delta\hat{\mbf{i}}} \in \tpb{\ell,\ell+1}{\mbf{0}}$ and $\supp(\threeb) \subset \supp(\phi) \subset \supp(\zerob_{\overline{\mbf{i}}+\Delta\hat{\mbf{i}}}) \subset B^\mbf{j}_\mbf{i}$.
	\end{itemize}
	Such a $\Delta\hat{\mbf{i}}$ exists because of Assumption \ref{assume:shortest_path}.
	In particular, the last condition implies that there exists $\overline{\mbf{r}}_2$ such that
	\begin{equation*}
		\mbf{j}-\mbf{1} \leq \mbf{i} + \overline{\mbf{r}}_2 + (m_k+1)\delta_k - \overline{\mbf{i}} - \Delta\hat{\mbf{i}} \leq \mbf{0}\;,
	\end{equation*}
	with equality holding on the right for the $k$-th component because of the definition of $\Delta \hat{i}_k$.
	Therefore, we see that the following inequality holds \textit{component-wise},
	\begin{equation*}
		\begin{gathered}
			\mbf{0} \leq \overline{\mbf{i}} + \Delta\hat{\mbf{i}} - \mbf{i} - \overline{\mbf{r}}_2 - (m_k+1)\delta_k \leq \mbf{1}-\mbf{j}-\delta_k\;.
		\end{gathered}
	\end{equation*}
	Therefore, $\overline{\mbf{i}}+\Delta\hat{\mbf{i}} \in \deriv{\ell,\ell+1}^{\mbf{j}+\delta_k}({\mbf{i}+\overline{\mbf{r}}_2+(m_k+1)\delta_k})$ and consequently $\supp(\zerob_{\overline{\mbf{i}}+\Delta\hat{\mbf{i}}}) \subset B^{\mbf{j}+\delta_k}_{\mbf{i}}$.
	This implies that
	\begin{equation*}\label{eq:partial_containment1_intersection}
		\tpb{\ell+s,\ell+1}{*}(A_\mbf{i}^\mbf{j}) \cap \tpb{\ell+s,\ell+1}{*}(B^\mbf{j}_\mbf{i})
		\subset \tpb{\ell+s,\ell+1}{*}(B^{\mbf{j}+\delta_k}_{\mbf{i}})
		\subset \tpb{\ell+s,\ell+1}{*}(A_\mbf{i}^\mbf{j} \cap B^\mbf{j}_\mbf{i})\;.
	\end{equation*}
	
	Finally, from Lemma \ref{lem:space_containment}, if $\gb \in \tpb{\ell+s,\ell+1}{*}(A_\mbf{i}^\mbf{j} \cap B^\mbf{j}_\mbf{i})$ then $\gb \in \tpb{\ell+s,\ell+1}{*}(A_\mbf{i}^\mbf{j})$ and $\gb \in \tpb{\ell+s,\ell+1}{*}(B^\mbf{j}_\mbf{i})$.
	As a result
	\begin{equation*}\label{eq:partial_containment2_intersection}
		\tpb{\ell+s,\ell+1}{*}(A_\mbf{i}^\mbf{j} \cap B^\mbf{j}_\mbf{i}) \subset \tpb{\ell+s,\ell+1}{*}(A_\mbf{i}^\mbf{j}) \cap \tpb{\ell+s,\ell+1}{*}(B^\mbf{j}_\mbf{i})\;.
	\end{equation*}
	The claim of equation \eqref{eq:equal_intersections} follows.
\end{proof}

\begin{corollary}\label{cor:complex_equality_intersections}
	With $A_\mbf{i}^\mbf{j}$ and $B^\mbf{j}_\mbf{i}$ defined as in Lemma \ref{lem:intersections},
	\begin{equation}\label{eq:intersection_space_equality}
		\tpbs{\ell+s}{}(\square) = \tpbs{\ell+s,\ell+1}{}(\square)
	\end{equation}
	for $\square \in \{ A_\mbf{i}^\mbf{j}, B_\mbf{i}^\mbf{j}, B_\mbf{i}^{\mbf{j}+\delta_k}, A_\mbf{i}^\mbf{j} \cap B^\mbf{j}_\mbf{i}\}$.
\end{corollary}
\begin{proof}
	For $\square \in \{A_\mbf{i}^\mbf{j}, B_\mbf{i}^\mbf{j},  B^{\mbf{j}+\delta_k}_{\mbf{i}}\},$ equation \eqref{eq:intersection_space_equality} follows directly from Lemma \ref{lem:complex_equality}.
	It remains to show that $\tpbs{\ell+s}{}(A_\mbf{i}^\mbf{j} \cap B_\mbf{i}^\mbf{j}) = \tpbs{\ell+s,\ell+1}{}(A_\mbf{i}^\mbf{j} \cap B_\mbf{i}^\mbf{j}).$
	By definition, $\tpbs{\ell+s}{}(A_\mbf{i}^\mbf{j} \cap B_\mbf{i}^\mbf{j}) \supset \tpbs{\ell+s,\ell+1}{}(A_\mbf{i}^\mbf{j} \cap B_\mbf{i}^\mbf{j})$.
	The reverse inclusion holds since
	\begin{align*}
			\gf &\in \tpbs{\ell+s}{}(A_\mbf{i}^\mbf{j} \cap B_\mbf{i}^\mbf{j})\;, &\\
			\Rightarrow \gf &\in \tpbs{\ell+s}{}(A_\mbf{i}^\mbf{j}) \cap \tpbs{\ell+s}{}(B_\mbf{i}^\mbf{j})\;, &\\
			\Rightarrow f &\in \tpbs{\ell+s,\ell+1}{}(A_\mbf{i}^\mbf{j}) \cap \tpbs{\ell+s,\ell+1}{}(B_\mbf{i}^\mbf{j})\;, &\mbox{(Lemma \ref{lem:complex_equality})}\;\;\\
			\Rightarrow f &\in \tpbs{\ell+s,\ell+1}{}(A_\mbf{i}^\mbf{j} \cap B_\mbf{i}^\mbf{j}) & \mbox{(Equation \eqref{eq:equal_intersections})}\;.
	\end{align*}
\end{proof}

\begin{lemma}\label{lem:unions}
	For a fixed $1 \leq k \leq n$, let $\delta_k$ be an $n$-tuple with the value $1$ in index $k$ and $0$ elsewhere, and let $\mbf{i}$, $\mbf{j}$ and $\mbf{m}$ be given.
	Then, with $\overline{\mbf{r}} := (r_1,\dots,r_{k-1},0,r_{k+1},\dots,r_n)$ and with
	\begin{equation*}
		\begin{split}
			A_\mbf{i}^\mbf{j} &:= S_\mbf{i}^\mbf{j}[\mbf{m}]\;,\\
			B^\mbf{j}_\mbf{i} &:= \bigcup_{r_1=0}^{m_1}\cdots\bigcup_{r_{k-1}=0}^{m_{k-1}}\bigcup_{r_{k+1}=0}^{m_{k+1}}\cdots\bigcup_{r_n=0}^{m_n}\Omega^\mbf{j}_{\ell+1}(\mbf{i}+\overline{\mbf{r}}+(m_k+1)\delta_k)\;,\\
			C_\mbf{i}^\mbf{j} &:= S_\mbf{i}^\mbf{j}[\mbf{m}+\delta_k]\;,
		\end{split}
	\end{equation*}
	the following holds for $s = 0, 1$,	
	\begin{align}
		\tpb{\ell+s,\ell+1}{*}(A_\mbf{i}^\mbf{j}) \cup \tpb{\ell+s,\ell+1}{*}(B^\mbf{j}_\mbf{i}) 
		&= \tpb{\ell+s,\ell+1}{*}(A_\mbf{i}^\mbf{j} \cup B^\mbf{j}_\mbf{i})
		= \tpb{\ell+s,\ell+1}{*}(C_\mbf{i}^\mbf{j})\label{eq:equal_unions} \;.
	\end{align}
	In the above a superscript of $*$ denotes that the statement is true for all B-splines.
\end{lemma}
\begin{proof}
	From Lemma \ref{lem:unions_intersections_domains},  $A_\mbf{i}^\mbf{j} \cup B^\mbf{j}_\mbf{i} = C^\mbf{j}_\mbf{i}$, and so we focus only on the first equality of the claim.
	However, the containment
		$\tpb{\ell+s,\ell+1}{*}(A^\mbf{j}_\mbf{i}) \cup \tpb{\ell+s,\ell+1}{*}(B^\mbf{j}_\mbf{i}) 
		\subset \tpb{\ell+s,\ell+1}{*}(A^\mbf{j}_\mbf{i} \cup B^\mbf{j}_\mbf{i})$
	again follows from Lemma \ref{lem:space_containment}.
	
	We prove inclusion in the other direction by contradiction.
	Let there be a B-spline $\phi \in \tpb{\ell+s,\ell+1}{\mbf{j}'}(A^\mbf{j}_\mbf{i} \cup B^\mbf{j}_\mbf{i})$ that is fully supported neither on $A^\mbf{j}_\mbf{i}$ nor on $B^\mbf{j}_\mbf{i}$, i.e., $\supp(\phi) \not\subset A^\mbf{j}_\mbf{i}$ and  $\supp(\phi) \not\subset  B^\mbf{j}_\mbf{i}$.
	By definition of $A^\mbf{j}_\mbf{i}$ and $B^\mbf{j}_\mbf{i}$, this means that $\supp(\phi)$ is split by a hyperplane perpendicular to the $k$-th direction into two disjoint parts --- denote the first part contained in $A^\mbf{j}_\mbf{i}$ as $\supp_A(\phi)$, and denote the second part contained in $B^\mbf{j}_\mbf{i}$ as $\supp_B(\phi)$.
	Moreover, again from the definitions of these domains, there exist $\zerob_{\overline{\mbf{i}}}, \zerob_{\overline{\mbf{i}} + \Delta\overline{\mbf{i}}} \in \tpb{\ell,\ell+1}{\mbf{0}}$ such that
	\begin{equation*}
		\begin{split}
			\supp_A(\phi) \subset \supp(\zerob_{\overline{\mbf{i}}}) \subset A^\mbf{j}_\mbf{i}\;,\quad
			\supp_B(\phi) \subset \supp(\zerob_{\overline{\mbf{i}} + \Delta\overline{\mbf{i}}}) \subset B^\mbf{j}_\mbf{i}\;.
		\end{split}
	\end{equation*}
	The intersection of $\supp(\phi)$ with the aforementioned hyperplane is contained in $\csupp(\zerob_{\overline{\mbf{i}}}) \cap \csupp(\zerob_{\overline{\mbf{i}} + \Delta\overline{\mbf{i}}})$.
	Assumption \ref{assume:shortest_path} then implies that there exists a shortest chain between $\zerob_{\overline{\mbf{i}}}$ and  $\zerob_{\overline{\mbf{i}} + \Delta\overline{\mbf{i}}}$.
	Thus, there must exist a $0$-form B-spline in this shortest chain whose support contains $\supp(\phi)$ and which itself is fully supported either on $A^\mbf{j}_\mbf{i}$ or on $B^\mbf{j}_\mbf{i}$ (since both domains are themselves built as unions of $0$-form supports).
	
\end{proof}

We are now ready to show how the above results and equation \eqref{eq:sliced_partition} can be used to show exactness of the hierarchical spline complex.
We will do this by considering some special slices $S^\mbf{j}_\mbf{i}[\mbf{m}]$.
In the following, we fix $\mbf{i} = (i_1,\dots,i_n)$, $\mbf{j} = (j_1,\dots,j_n)$,
and for $k \geq 1$ define
\begin{subequations}
	\begin{align}
		\mbf{i}_k &:= \sum_{l=1}^k \delta_l + \sum_{l=k+1}^n i_l\delta_l\;,\quad
		\mbf{j}_k := \sum_{l=k+1}^n j_l\delta_l\;,\\
		S[k, m_k] &:= \bigcup_{r_1=0}^{\ndof{\ell}{1}}\cdots\bigcup_{r_{k-1}=0}^{\ndof{\ell}{k-1}}\bigcup_{r_k=0}^{m_k}\Omega^{\mbf{j}_k}_{\ell+1}\Big(\mbf{i}_k+\sum_{l=1}^kr_l\delta_l\Big)\;,\\
		S[k] &:= S[k, \ndof{\ell}{k}]\;.
	\end{align}
\end{subequations}
Observe that, from equation \eqref{eq:sliced_partition}, $S[n] = \Omega_{\ell+1}$.
Moreover, we will say that a domain, say $B$, is an $S[k]$-type domains if it is defined as above but with possibly different choices of $\mbf{i}$ and $\mbf{j}$.

\begin{proposition}\label{prop:1d_result}
	Let Assumption \ref{assume:shortest_path} hold.
	Then, the inclusion operation
	\begin{equation}
		\splinelevelinclusion{\ell,\ell+1}:\tpbs{\ell,\ell+1}{}\Big(S[1]\Big)\rightarrow\tpbs{\ell+1,\ell+1}{}\Big(S[1]\Big)
	\end{equation}
	induces an isomorphism on the cohomology of the spline spaces.
\end{proposition}
\begin{proof}
	We will show the claim for $S[1]$ by considering $S[1,m_1]$ and inducting on $m_1$.
	For the base case of $m_1 = 0$, the claim follows from Corollary \ref{corol:injection_cohomology}.
	Assume that the claim is true for some $m_1 \geq 0$ and, as in Lemmas \ref{lem:intersections} and \ref{lem:unions}, define
	\begin{equation*}
		\begin{split}
			A &:= S[1,m_1]\;,\quad
			B := \Omega^{\mbf{j}_1}_{\ell+1}(\mbf{i}_1+(m_1+1)\delta_1)\;,\quad
			C := S[1,m_1+1]\;.
		\end{split}
	\end{equation*}
	Thus, the induction hypothesis is that the following map induces an isomorphism on the spline cohomology,
	\begin{equation*}
		\splinelevelinclusion{\ell,\ell+1}:\tpbs{\ell,\ell+1}{}(A)\rightarrow\tpbs{\ell+1,\ell+1}{}(A)\;,
	\end{equation*}
	and our objective is to show that the same is true for 
	\begin{equation*}	\splinelevelinclusion{\ell,\ell+1}:\tpbs{\ell,\ell+1}{}(C)\rightarrow\tpbs{\ell+1,\ell+1}{}(C)\;.
	\end{equation*}
	
	First, observe that Lemma \ref{lem:complex_equality} implies that $\tpbs{\ell+s,\ell+1}{}(\square) = \tpbs{\ell+s}{}(\square)$ for any $\square\in\{ A, B, C\},$ and Corollary \ref{cor:complex_equality_intersections} implies  $\tpbs{\ell+s,\ell+1}{}(A \cap B) = \tpbs{\ell+s}{}(A \cap B)$.
	Moreover, from Theorem \ref{thrm:natural_Mayer_vietoris}, the following diagram commutes,
	\begin{equation*}
		\begin{tikzcd}[column sep=small]
			\cdots \arrow{r} & H^{(j)}_{\ell}(A \cap B) \arrow{r}{} \arrow{d}{} &H^{(j)}_{\ell}(A) \oplus H^{(j)}_{\ell}(B) \arrow{r}{} \arrow{d}{} & H^{(j)}_{\ell}(A \boxplus B) \arrow{r}{} \arrow{d}{} & H^{(j+1)}_{\ell}(A \cap B) \arrow{r} \arrow{d}{} & \cdots\\
			\cdots \arrow{r} & H^{(j)}_{\ell+1}(A\cap B) \arrow{r}{} &H^{(j)}_{\ell+1}(A) \oplus H^{(j)}_{\ell+1}(B) \arrow{r}{}  & H^{(j)}_{\ell+1}(A \boxplus B) \arrow{r}{} & H^{(j+1)}_{\ell+1}(A \cap B) \arrow{r} & \cdots
		\end{tikzcd}
	\end{equation*}
	Using Corollary \ref{corol:injection_cohomology} and Lemma \ref{lem:intersections}, the first vertical map is an isomorphism for all $j$.
	By the inductive hypothesis and Corollary \ref{corol:injection_cohomology}, the same is true for the second vertical map.
	Then, from Corollary \ref{cor:natural_Mayer_vietoris}, the third vertical map is also an isomorphism.
	This proves the claim since, from Lemma \ref{lem:unions}, the spline complex $\tpcmplx{\ell+s}(A \boxplus B)$ is the same as $\tpcmplx{\ell+s}(A \cup B) = \tpcmplx{\ell+s}(C)$ for $s = 0,1$.
\end{proof}

\begin{proposition}\label{prop:nd_result}
	Let Assumption \ref{assume:shortest_path} hold.
	Then, the inclusion operation
	\begin{equation}
		\splinelevelinclusion{\ell,\ell+1}:\tpbs{\ell,\ell+1}{}\Big(S[n]\Big)\rightarrow\tpbs{\ell+1,\ell+1}{}\Big(S[n]\Big)
	\end{equation}
	induces an isomorphism on the cohomology of the spline spaces.
\end{proposition}
\begin{proof}
	We proceed by looking at $S[k]$ and inducting on $k$.
	The claim follows for the base case $k = 1$ from Proposition \ref{prop:1d_result}.
	Next, given $k \geq 1$, assume that the following inclusion induces an isomorphism on spline cohomology,
	\begin{equation*}
		\splinelevelinclusion{\ell,\ell+1}:\tpbs{\ell,\ell+1}{}\Big(S[k]\Big)\rightarrow\tpbs{\ell+1,\ell+1}{}\Big(S[k]\Big)\;.\tag{$\star$}\label{eq:ak}
	\end{equation*}
	Our objective is to show that the same is true for $k + 1$,
	\begin{equation*}
		\splinelevelinclusion{\ell,\ell+1}:\tpbs{\ell,\ell+1}{}(S[k+1])\rightarrow\tpbs{\ell+1,\ell+1}{}(S[k+1])\;.\tag{$\ast$}\label{eq:akp1}
	\end{equation*}
	This will complete the proof and we proceed by inducting on subdomains of $S[k+1]$.\\
	
	\noindent
	\textbf{Nested induction for $S[k+1,m_{k+1}]$}:\\
	We will show the claim for $S[k+1]$ by considering $S[k+1,m_{k+1}]$ and inducting on $m_{k+1}$.
	The base case of $m_{k+1} = 0$ follows from the induction hypothesis in \eqref{eq:ak} since $S[k+1,0] = S[k]$.
	Assume that the claim is true for some $m_{k+1} \geq 0$ and, as in Lemma \ref{lem:intersections}, define
	\begin{equation*}
		\begin{split}
			A &:= S[k+1,m_{k+1}]\;,\quad C := S[k+1,m_{k+1}+1]\\
			B_{\mbf{i}_{k+1}}^{\mbf{j}_{k+1}} &:= \bigcup_{r_1=0}^{\ndof{\ell}{1}}\cdots\bigcup_{r_{k}=0}^{\ndof{\ell}{k}}\Omega_{\ell+1}^{\mbf{j}_{k+1}}\Big({\mbf{i}_{k+1}} +\sum_{l=1}^{k}r_l\delta_l + (m_{k+1}+1)\delta_{k+1}\Big)\;.
		\end{split}
	\end{equation*}
	Simple two-dimensional examples of such domains are shown in Figure \ref{fig:slab_induction_2d}, while more involved three-dimensional examples are shown in Figure \ref{fig:slab_induction_3d}.
	Thus, the induction hypothesis is that the following map induces an isomorphism on the spline cohomology,
	\begin{equation*}
		\splinelevelinclusion{\ell,\ell+1}:\tpbs{\ell,\ell+1}{}(A)\rightarrow\tpbs{\ell+1,\ell+1}{}(A)\;,
	\end{equation*}
	and our objective is to show that the same is true for 
	\begin{equation*}	\splinelevelinclusion{\ell,\ell+1}:\tpbs{\ell,\ell+1}{}(C)\rightarrow\tpbs{\ell+1,\ell+1}{}(C)\;.
	\end{equation*}
	
	First, note that $\tpb{\ell+s,\ell+1}{*}(A \cap B_{\mbf{i}_{k+1}}^{\mbf{j}_{k+1}}) =
	\tpb{\ell+s,\ell+1}{*}(B_{\mbf{i}_{k+1}}^{\mbf{j}_{k+1}+\delta_{k+1}})$ by Lemma \ref{lem:intersections}.
	Since both $B_{\mbf{i}_{k+1}}^{\mbf{j}_{k+1}}$ and $B_{\mbf{i}_{k+1}}^{\mbf{j}_{k+1}+\delta_{k+1}}$ are domains of $S[k]$-type domains,
	the following inclusion map induces an isomorphism on the spline cohomology from the induction hypothesis \eqref{eq:ak},
	\begin{equation*}
		\splinelevelinclusion{\ell,\ell+1}:\tpbs{\ell,\ell+1}{}(\square)\rightarrow\tpbs{\ell+1,\ell+1}{}(\square)\;,\quad \square \in \Big\{B_{\mbf{i}_{k+1}}^{\mbf{j}_{k+1}}\;,\; B_{\mbf{i}_{k+1}}^{\mbf{j}_{k+1}+\delta_{k+1}} \Big\}\;.
		\tag{$\dagger$}\label{eq:akp1_0}
	\end{equation*}

	Next, for any $\square \in \{ A, B_{\mbf{i}_{k+1}}^{\mbf{j}_{k+1}}, C, A \cap B_{\mbf{i}_{k+1}}^{\mbf{j}_{k+1}} \}$, Lemma \ref{lem:complex_equality} and Corollary \ref{cor:complex_equality_intersections} imply that $\tpcmplx{\ell+s,\ell+1}(\square) = \tpcmplx{\ell+s}(\square)$, $s = 0,1$.
	Then, following the same line of reasoning as in Proposition \ref{prop:1d_result}, 
	Theorem \ref{thrm:natural_Mayer_vietoris} implies the following commuting diagram,
	\begin{equation*}
		\small
		\begin{tikzcd}[column sep=0.13in]
			\cdots \arrow{r} & H^{(j)}_{\ell}(A \cap B_{\mbf{i}_{k+1}}^{\mbf{j}_{k+1}}) \arrow{r}{} \arrow{d}{} &H^{(j)}_{\ell}(A) \oplus H^{(j)}_{\ell}(B_{\mbf{i}_{k+1}}^{\mbf{j}_{k+1}}) \arrow{r}{} \arrow{d}{} & H^{(j)}_{\ell}(A \boxplus B_{\mbf{i}_{k+1}}^{\mbf{j}_{k+1}}) \arrow{r}{} \arrow{d}{} & H^{(j+1)}_{\ell}(A \cap B_{\mbf{i}_{k+1}}^{\mbf{j}_{k+1}}) \arrow{r} \arrow{d}{} & \cdots\\
			\cdots \arrow{r} & H^{(j)}_{\ell+1}(A\cap B_{\mbf{i}_{k+1}}^{\mbf{j}_{k+1}}) \arrow{r}{} &H^{(j)}_{\ell+1}(A) \oplus H^{(j)}_{\ell+1}(B_{\mbf{i}_{k+1}}^{\mbf{j}_{k+1}}) \arrow{r}{}  & H^{(j)}_{\ell+1}(A \boxplus B_{\mbf{i}_{k+1}}^{\mbf{j}_{k+1}}) \arrow{r}{} & H^{(j+1)}_{\ell+1}(A \cap B_{\mbf{i}_{k+1}}^{\mbf{j}_{k+1}}) \arrow{r} & \cdots
		\end{tikzcd}
	\end{equation*}
	The first vertical map is an isomorphism for all $j$ from \eqref{eq:akp1_0}, and the second vertical map is an isomorphism for all $j$ from the induction hypotheses on $A = S[k+1,m_{k+1}]$ and \eqref{eq:akp1_0}.
	Therefore, so is the third vertical map from Corollary \ref{cor:natural_Mayer_vietoris}.
	Since, from Lemma \ref{lem:unions}, the spline complex $\tpcmplx{\ell+s}(A \boxplus B_{\mbf{i}_{k+1}}^{\mbf{j}_{k+1}}) = \tpcmplx{\ell+s}(A \cup B_{\mbf{i}_{k+1}}^{\mbf{j}_{k+1}}) = \tpcmplx{\ell+s}(C)$ for $s = 0,1$, this completes the induction on $m_{k+1}$ and, consequently, on $k$.
\end{proof}

\renewcommand*{\xBoxDim}{5cm}%
\renewcommand*{\yBoxDim}{5cm}%
\renewcommand*{\thistextwidth}{0.49}
\begin{figure}
	\centering
\begin{subfigure}{\thistextwidth\textwidth}
	\centering
		\resizebox{\xBoxDim}{\yBoxDim}{%	
           	\begin{tikzpicture}
            		\mymanualgrid{0,...,9}{0,...,9}{0.25}{gray}{white};
			\mymanualgrid{1,1.5,2,2.5,3,3.5,4,4.5,5,5.5,6,6.5,7,7.5,8}{1,1.5,2,2.5,3,3.5,4,4.5,5,5.5,6,6.5,7,7.5,8}{1}{black}{white}
			\mymanualgrid{4,5}{4,5}{1}{black}{white}
			\mymanualgrid{1,1.5,2,2.5,3,3.5,4,4.5,5,5.5,6,6.5,7,7.5,8}{1,1.5,2,2.5,3,3.5,4}{1}{black}{myCPlot3}
			\mymanualgrid{1,1.5,2,2.5,3,3.5,4}{4,4.5,5,5.5,6,6.5,7}{1}{black}{myCPlot3}
			\mymanualgrid{5,5.5,6,6.5,7,7.5,8}{4,4.5,5,5.5,6,6.5,7}{1}{black}{myCPlot3}
           	 \end{tikzpicture}
           	 }
	 \caption{Domain $A = S[2,6]$}
\end{subfigure}
	\begin{subfigure}{\thistextwidth\textwidth}
	\centering
		\resizebox{\xBoxDim}{\yBoxDim}{%	
           	\begin{tikzpicture}
            		\mymanualgrid{0,...,9}{0,...,9}{0.25}{gray}{white};
			\mymanualgrid{1,1.5,2,2.5,3,3.5,4,4.5,5,5.5,6,6.5,7,7.5,8}{1,1.5,2,2.5,3,3.5,4,4.5,5,5.5,6,6.5,7,7.5,8}{1}{black}{white}
			\mymanualgrid{4,5}{4,5}{1}{black}{white}
			\mymanualgrid{1,1.5,2,2.5,3,3.5,4,4.5,5,5.5,6,6.5,7,7.5,8}{5,5.5,6,6.5,7,7.5,8}{1}{black}{myCPlot1}
			\mymanualgrid{1,1.5,2,2.5,3,3.5,4}{4,4.5,5}{1}{black}{myCPlot1}
			\mymanualgrid{5,5.5,6,6.5,7,7.5,8}{4,4.5,5}{1}{black}{myCPlot1}
           	 \end{tikzpicture}
           	 }
	 \caption{Domain $B_{\mbf{i}_{k+1}}^{\mbf{j}_{k+1}}$}
\end{subfigure}
\\	 %
\begin{subfigure}{\thistextwidth\textwidth}
	\centering
		\resizebox{\xBoxDim}{\yBoxDim}{%	
           	\begin{tikzpicture}
            		\mymanualgrid{0,...,9}{0,...,9}{0.25}{gray}{white};
			\mymanualgrid{1,1.5,2,2.5,3,3.5,4,4.5,5,5.5,6,6.5,7,7.5,8}{1,1.5,2,2.5,3,3.5,4,4.5,5,5.5,6,6.5,7,7.5,8}{1}{black}{white}
			\mymanualgrid{4,5}{4,5}{1}{black}{white}
			\mymanualgrid{1,1.5,2,2.5,3,3.5,4}{4,4.5,5,5.5,6,6.5,7}{1}{black}{green}
			\mymanualgrid{5,5.5,6,6.5,7,7.5,8}{4,4.5,5,5.5,6,6.5,7}{1}{black}{green}
           	 \end{tikzpicture}
           	 }
	 \caption{Domain $B_{\mbf{i}_{k+1}}^{\mbf{j}_{k+1}+\delta_{k+1}}$}
\end{subfigure}
	\begin{subfigure}{\thistextwidth\textwidth}
	\centering
		\resizebox{\xBoxDim}{\yBoxDim}{%	
           	\begin{tikzpicture}
            		\mymanualgrid{0,...,9}{0,...,9}{0.25}{gray}{white};
			\mymanualgrid{1,1.5,2,2.5,3,3.5,4,4.5,5,5.5,6,6.5,7,7.5,8}{1,1.5,2,2.5,3,3.5,4,4.5,5,5.5,6,6.5,7,7.5,8}{1}{black}{red}
			\mymanualgrid{4,5}{4,5}{1}{black}{white}
           	 \end{tikzpicture}
           	 }
	 \caption{Domain $C=S[2,7]$}
\end{subfigure}
	\caption{
		The above figures correspond to a maximally regular hierarchical spline space defined using $\pdeg{\ell}{k} = 2$ for all $\ell$ and $k$.
		For $k = 1$, and with the proof of Proposition \ref{prop:nd_result} as reference (the part on nested induction), example domains $A, B_{\mbf{i}_{k+1}}^{\mbf{j}_{k+1}}, B_{\mbf{i}_{k+1}}^{\mbf{j}_{k+1}+\delta_{k+1}}$ and $C$ are shown for this hierarchical configuration for $m_{k+1} = 6$.
	 }
	\label{fig:slab_induction_2d}
\end{figure}
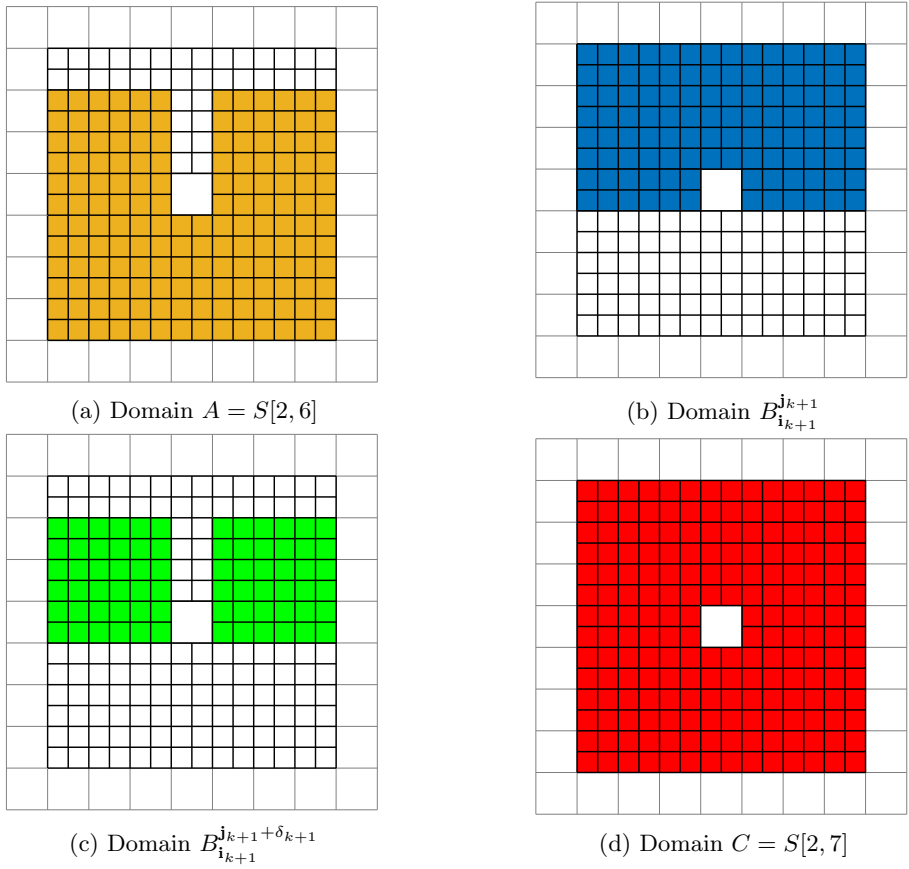

\begin{figure}
	\centering
	\begin{subfigure}{0.485\textwidth}
		\includegraphics[width=1\textwidth,trim=5cm 14.5cm 6.5cm 25cm,clip]{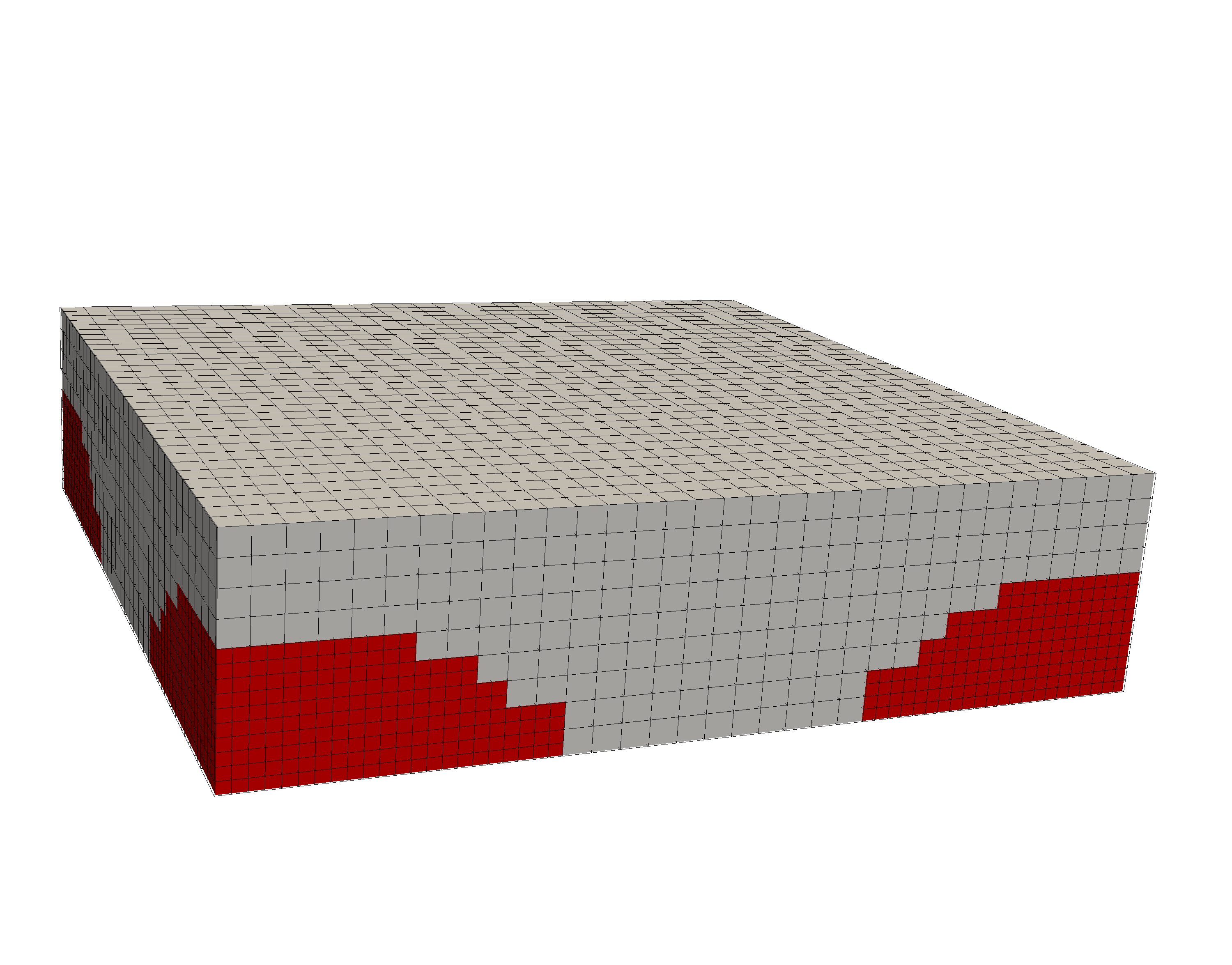}
		\caption{Domain $C = S[3,4]$ with Opaque $\Omega_0$}
	\end{subfigure}
	\;
	\begin{subfigure}{0.485\textwidth}
		\includegraphics[width=1\textwidth,trim=5cm 14.5cm 6.5cm 25cm,clip]{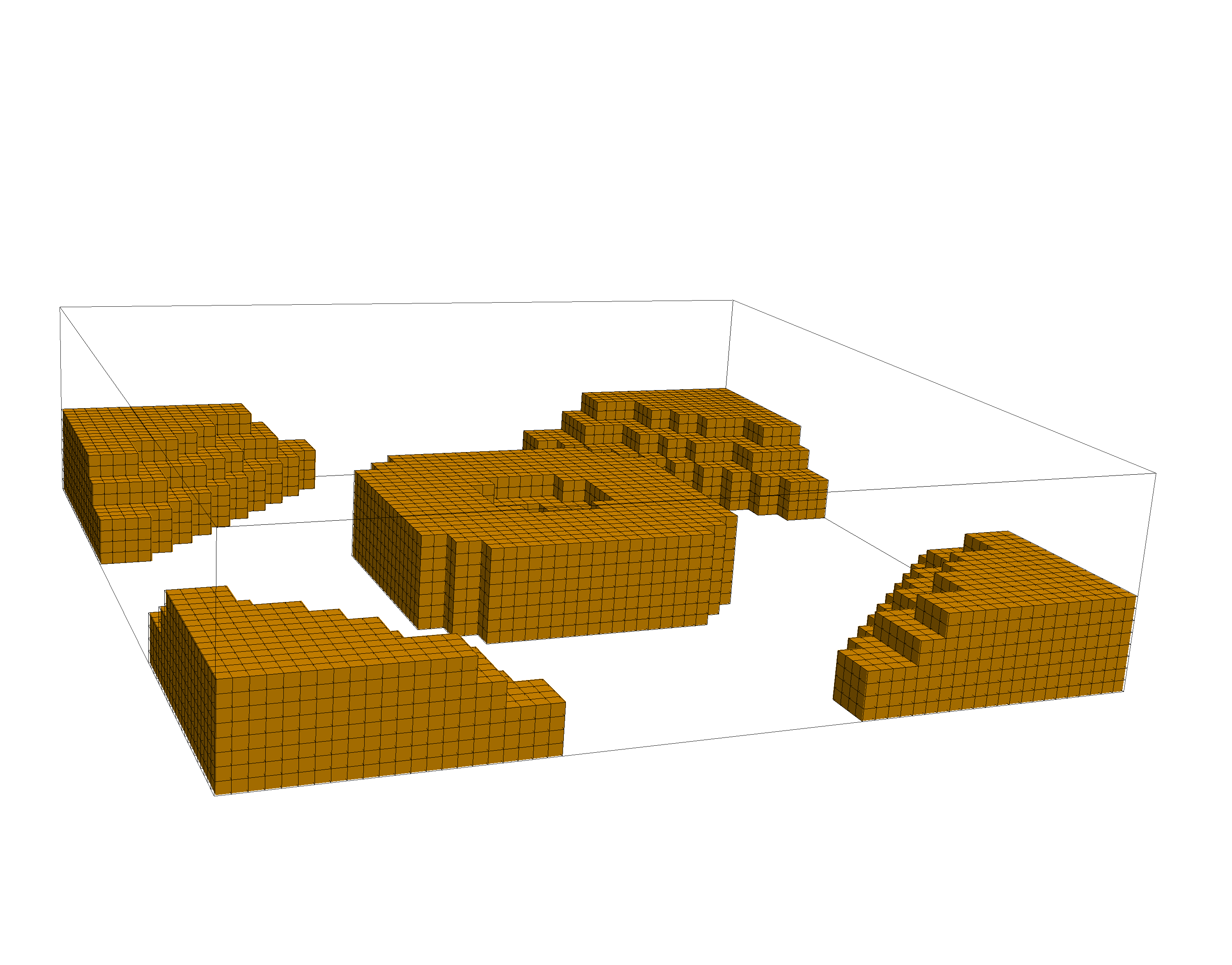}
		\caption{Domain $A = S[3,3]$}
	\end{subfigure}
	\\
	\begin{subfigure}{0.485\textwidth}
		\includegraphics[width=1\textwidth,trim=5cm 14.5cm 6.5cm 25cm,clip]{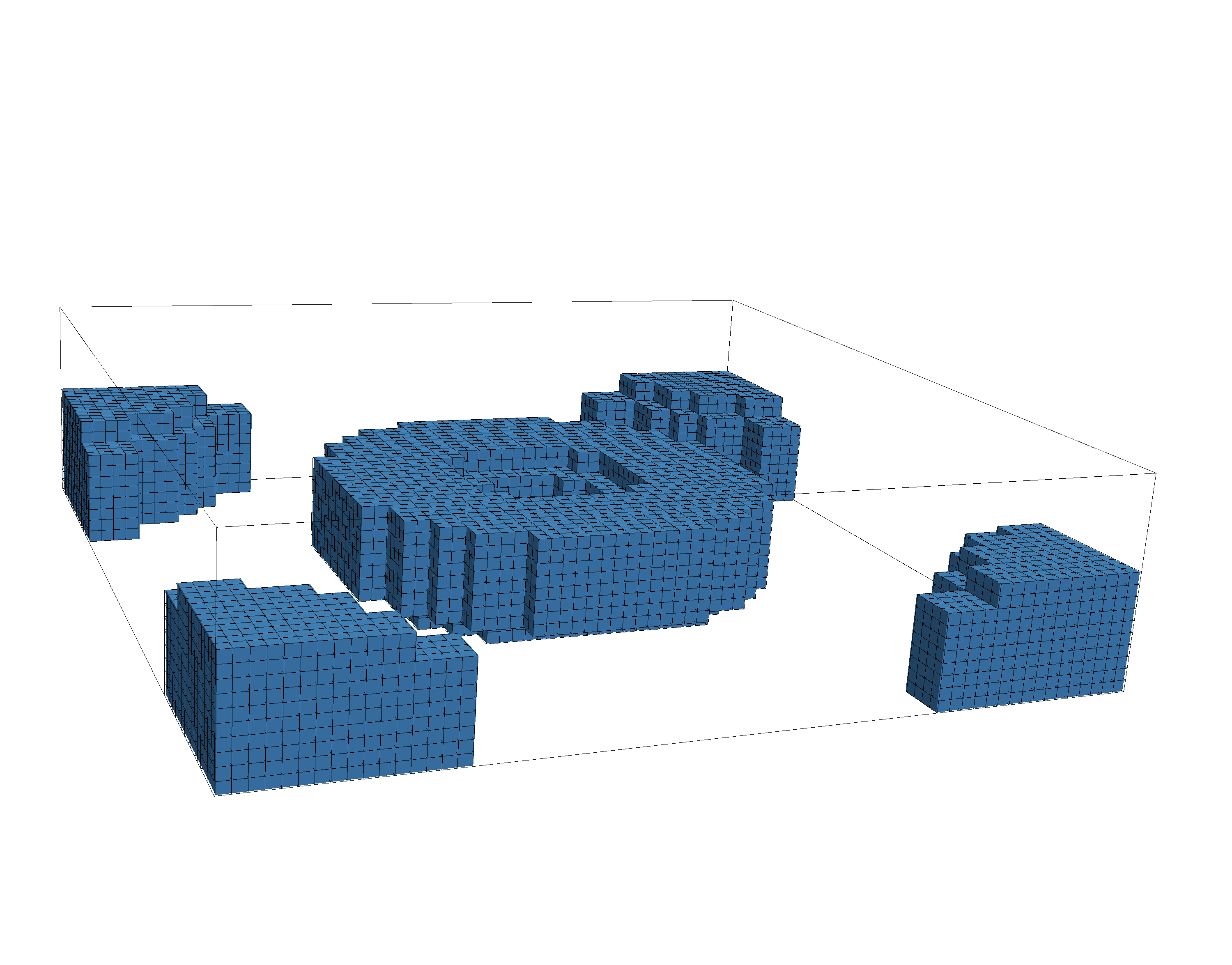}
		\caption{Domain $B_{\mbf{i}_{k+1}}^{\mbf{j}_{k+1}}$}
	\end{subfigure}
	\;
	\begin{subfigure}{0.485\textwidth}
		\includegraphics[width=1\textwidth,trim=5cm 14.5cm 6.5cm 25cm,clip]{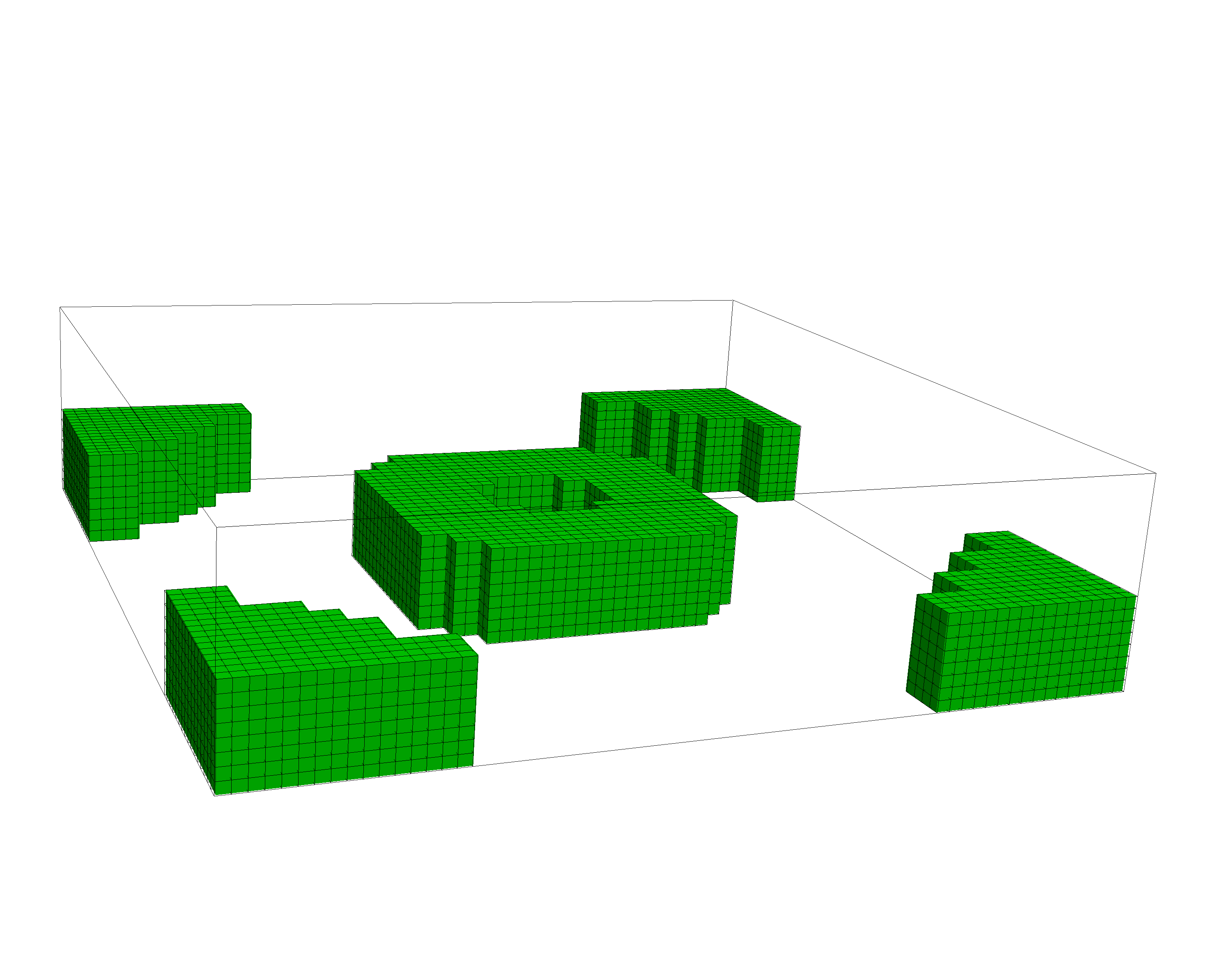}
		\caption{Domain $B_{\mbf{i}_{k+1}}^{\mbf{j}_{k+1}+\delta_{k+1}}$}
	\end{subfigure}
	\\
	\begin{subfigure}{0.75\textwidth}
		\includegraphics[width=1\textwidth,trim=5cm 14.5cm 6.5cm 25cm,clip]{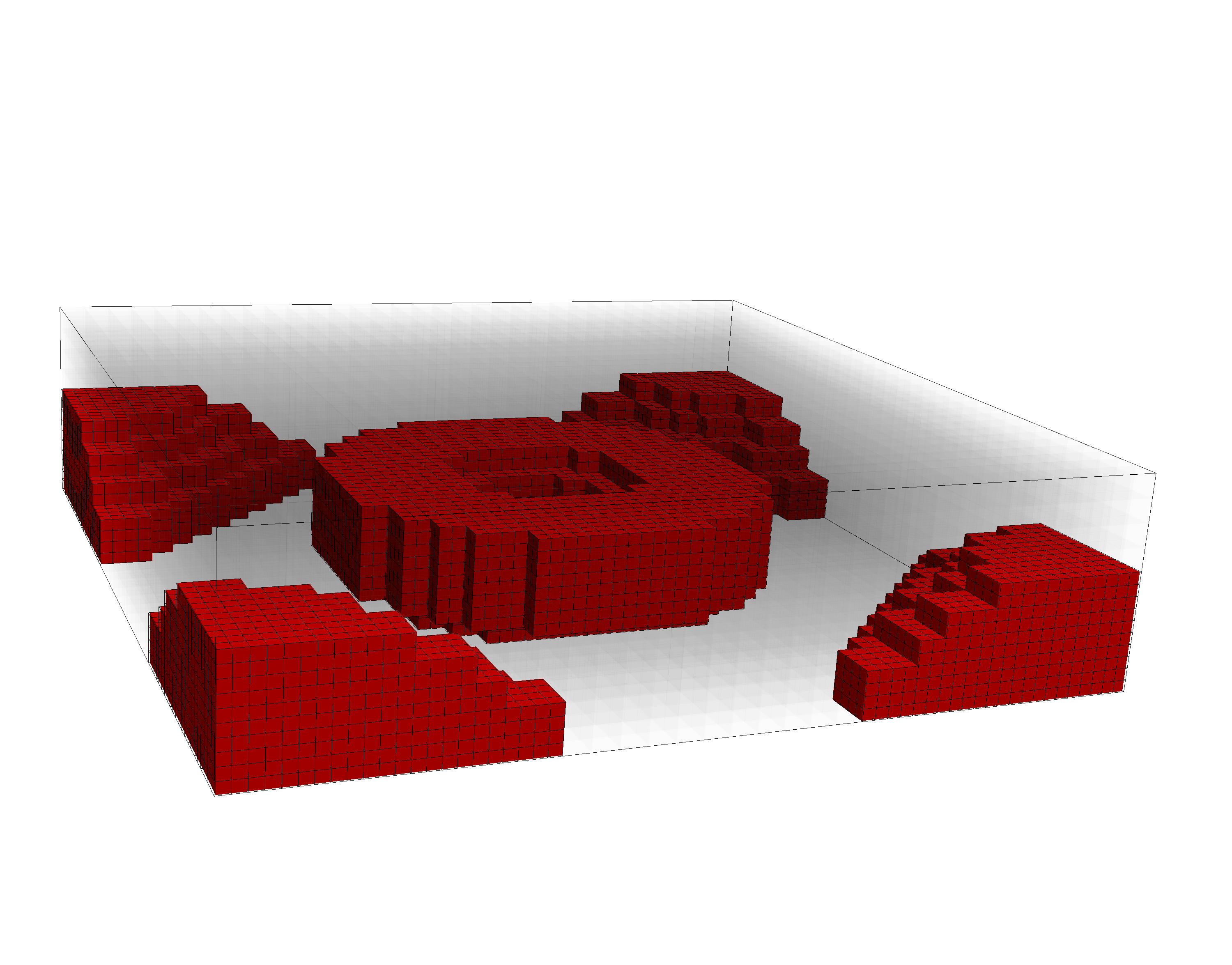}
		\caption{Domain $C = S[3,4] \subset \Omega_1$ with Semitransparent $\Omega_0$}
	\end{subfigure}
	\caption{
		The above figures correspond to a maximally regular hierarchical spline space defined using $\pdeg{\ell}{k} = 3$ for all $\ell$ and $k$.
		For $k = 2$, and with the proof of Proposition \ref{prop:nd_result} as reference (the part on nested induction), example domains $A, B_{\mbf{i}_{k+1}}^{\mbf{j}_{k+1}}, B_{\mbf{i}_{k+1}}^{\mbf{j}_{k+1}+\delta_{k+1}}$ and $C$ are shown for this hierarchical configuration for $m_{k+1} = 3$.
}\label{fig:slab_induction_3d}
\end{figure}

\begin{theorem}\label{thrm:exactness}
	Let Assumption \ref{assume:shortest_path} hold.
	Then, the hierarchical chain complex
	\begin{equation}
		\begin{tikzcd}
			0 \arrow{r}{\subset} & \hbs{\ell}{0} \arrow{r}{d} & \hbs{\ell}{1} \arrow{r}{d} & \cdots \arrow{r}{d} & \hbs{\ell}{n} \arrow{r}{\int} &  \mathbb{R} \arrow{r} &0
		\end{tikzcd}
	\end{equation}
	is exact for any $0 \leq \ell \leq \nref$.
\end{theorem}
\begin{proof}
	Since $S[n] = \Omega_{\ell+1}$ for $\mbf{i} = 1$, Theorem 5.5 of \cite{Evans19} states that the hierarchical B-spline complex is exact as the following inclusion induces isomorphisms on spline cohomology $\forall \ell$,
	\begin{equation*}
		\splinelevelinclusion{\ell,\ell+1}:\tpbs{\ell,\ell+1}{}(\Omega_{\ell+1})\rightarrow\tpbs{\ell+1,\ell+1}{}(\Omega_{\ell+1})\;.
	\end{equation*}
\end{proof}}

\section{Implementation and Validation}\label{sec:numerics}

A thorough investigation of the numerical stability of these spline spaces, as well as a discussion of practical implementation and refinement strategies, is outside of the scope of this paper.
Instead, we discuss some basic ideas for computationally verifying that assumptions are satisfied for a particular refinement configuration and then investigate the (in)exactness of various refinement configurations.
Future work will explore numerical stability, approximation power of the spline spaces, and a-posteriori error analysis.

\subsection{Verifying Assumption \ref{assume:shortest_path} for a given refinement configuration}
Given a hierarchical refinement configuration, it is clear from the formulation of Assumption \ref{assume:shortest_path} that its verification only requires information from pairs of successive refinement levels.
In this subsection, we explicitly state how Assumption \ref{assume:shortest_path} can be verified in a computer implementation.\\

\noindent
\emph{\textbf{Condition 1}}\\
Let $0 \leq \ell \leq L-1$, $\mbf{i}$ be the index of a 0-form B-spline of level $\ell$, and define $I^{k,-}_{\ell,\ell+1}$ and $I^{k,+}_{\ell,\ell+1}$ to be the following maps,
\begin{equation*}
	\begin{split}
		I^{k,-}_{\ell,\ell+1}(\mbf{i}) &:= \min \left\{ l~:~\knt{l,\ell+1,k} = \knt{i_k,\ell,k}\right\}\;,\\
		I^{k,+}_{\ell,\ell+1}(\mbf{i}) &:= \max \left\{ l~:~\knt{l,\ell+1,k} = \knt{i_k+\pdeg{\ell}{k}+1,\ell,k}\right\}\;.
	\end{split}
\end{equation*}
Consider $\zerob_\mbf{i}, \zerob_{\mbf{i}+\Delta\mbf{i}} \in \tpb{\ell,\ell+1}{\mbf{0}}$ and, without loss of generality, let $\Delta\mbf{i}$ be component-wise non-negative.
Then,
\begin{equation*}
	\begin{gathered}
		\supp(\zerob_\mbf{i}), \supp(\zerob_{\mbf{i}+\Delta\mbf{i}}) \text{ share an $(\ndim-1,\ell+1)$-intersection}\\
		\Updownarrow\\
		\left(I^{k,+}_{\ell,\ell+1}(\mbf{i}) \geq I^{k,-}_{\ell,\ell+1}(\mbf{i}+\Delta\mbf{i})\quad \forall k \in \{1,\dots,\ndim\}\right) 
		\text{ and}\\
		\left(I^{k,+}_{\ell,\ell+1}(\mbf{i}) - I^{k,-}_{\ell,\ell+1}(\mbf{i}+\Delta\mbf{i}) \geq \pdeg{\ell+1}{k} \quad \forall k \in \mathcal{I} \subset \{1,\dots,\ndim\},  |\mathcal{I}| \geq (\ndim - 1) \right)
	\end{gathered}
\end{equation*}
This is how we can check whether the these splines share an $(\ndim-1,\ell+1)$ intersection, as in Assumption \ref{assume:shortest_path}.
If they do not, then there is no need to check the second condition and the assumption is satisfied.
On the other hand, such an intersection exists, then we need to check that the shortest path condition holds.
This can be done as follows.\\

\noindent
\emph{\textbf{Condition 2}}\\
{Assume that $\supp(\zerob_\mbf{i})$ and $\supp(\zerob_{\mbf{i}+\Delta\mbf{i}})$ share an $(n-1,\ell+1)$-intersection.}
Then, a simple computation reveals that there are at most $\Pi_{k=1}^n(\Delta i_k+1)$ shortest paths between $\zerob_{\mbf{i}}$ and $\zerob_{\mbf{i}+\Delta\mbf{i}}$.
Each such shortest path will consist of B-splines $\alpha_{\mbf{i}+\Delta\hat{\mbf{i}}}$ such that $0 \leq \Delta\hat{{i}}_k \leq \Delta{{i}}_k$ holds for all $k$.
Moreover, we also have $\Delta i_k \leq \pdeg{\ell}{k}$ since $\supp(\zerob_\mbf{i}) \cap \supp(\zerob_{\mbf{i}+\Delta\mbf{i}})$ contains at least one \Bezier{} cell, a conservative upper bound that ignores the effect of repeated knots.
Therefore, a brute force check on these limited number of shortest paths can be performed to verify that the second condition from Assumption \ref{assume:shortest_path} holds.
For instance, for a potential shortest path, each B-spline $\alpha_{\mbf{i}+\Delta\hat{\mbf{i}}}$ in it must be such that all \Bezier{} cells in its support have been refined.

Finally, let $B(\mbf{i},\Delta\mbf{i}) \subset \tpb{\ell,\ell+1}{\mbf{0}}$ be the set of B-splines that form a shortest path between $\zerob_\mbf{i}$ and $\zerob_{\mbf{i}+\Delta\mbf{i}}$; by convention we assume that $\zerob_\mbf{i}, \zerob_{\mbf{i}+\Delta\mbf{i}} \in B(\mbf{i},\Delta\mbf{i})$.
Then, the shortest-path check does not need to be performed again for any $\zerob_a,\zerob_b \in B(\mbf{i},\Delta\mbf{i})$ since the shortest path relation is commutative (the shortest path from $\zerob_a$ to $\zerob_b$ is also the shortest path from $\zerob_b$ to $\zerob_a$) and obeys the following recursion condition: if there is a shortest path from $\zerob_a$ to $\zerob_c$ containing $\zerob_b$, a subset of this path is a shortest path between $\zerob_a$ and $\zerob_b$ and another subset is a shortest path between $\zerob_b$ and $\zerob_c$.
This further reduces the computational workload behind this verification task.

\begin{remark}
	While the above explicit and conservative approaches may not be the most efficient ones, they are easy to understand and illustrate how the conditions in Assumption \ref{assume:shortest_path} can be verified in a local manner for any given hierarchical refinement configuration, with the basic process amounting to a sequence of counting operations.
\end{remark}

\subsection{Example applications and limitations of Theorem \ref{thrm:exactness}}

To verify that the current theory holds computationally, we construct the hierarchical B-spline complex of discrete differential forms for a variety of refinement configurations using extensions of GeoPDEs \cite{Vazquez:2016}.
The computations involve finding the nullity of the matrix corresponding to the mixed discretization of the $j$-form Hodge Laplacian for the de Rham complex \cite[Chapter 4]{Arnold:2018} (e.g., using a QR or singular value decomposition) for all $j$.
Specifically, this matrix corresponds to the following weak problem: find $(g,f) \in \hbs{\nref}{j-1} \times \hbs{\nref}{j}$ such that
\begin{equation*}
	\begin{split}
		\inpr[]{v}{g} - \inpr[]{dv}{f} &= 0\;,\quad v \in \hbs{\nref}{j-1}\;,\\
		\inpr[]{w}{dg} + \inpr[]{dw}{df} &= 0\;,\quad w \in \hbs{\nref}{j}\;.
	\end{split}
\end{equation*}
The hierarchical complex is exact in position $j$ if the matrix corresponding to the above problem is of full rank for $j < \ndim$ and has a nullity of one for $j = \ndim$.
Any increments in the nullity are in one-to-one correspondence with so-called discrete harmonic forms, i.e., spline functions in $\hbs{\nref}{j}$ that are cocycles and are orthogonal to coboundaries.
Existence of harmonic forms means the hierarchical spline complex is inexact, and the Hodge Laplacian is well-posed only up to these harmonic forms.

\subsubsection{Inexact Refinement Configurations and Resulting Harmonics}
In this section, we build various hierarchical configurations that violate one or more assumptions leading up to Theorem \ref{thrm:exactness}.
These are minimal examples for which sprurious harmonic $j$-forms exist for $j > 0$.
Spurious harmonic 0-forms cannot be created here because the imposition of homogeneous boundary conditions makes the 0-form Hodge Laplacian well-posed.\\

\noindent
\emph{\textbf{Two-dimensional examples}}\\
In two dimensions, representative configurations that are inexact and their accompanying Greville meshes $\gmesh{0,1}$ and $\gmesh{1,1}$ (which manifest the change in the spline space topology) are shown in Figure \ref{fig:2d_refinements_inexact_configs}.
The harmonic forms that are introduced as a result of this change in topology are shown in Figure \ref{fig:2d_refinements_inexact_harmonics}.
In all cases, it can be seen that the homology of the Greville mesh changes between $\gmesh{0,1}$ and $\gmesh{1,1}$ to produce a spurious harmonic form.
Introducing a new component (e.g. by violating Assumption \ref{assume:zero_union}) introduces additional harmonic 2-forms.
Harmonic 1-forms are created by either merging two connected components of $\gmesh{0,1}$ into a single connected component in $\gmesh{1,1}$ or by modifying the topology of $\gmesh{0,1}$ from being a topological ball into a topological annulus on $\gmesh{1,1}$.
In the case of Figure \ref{fig:2d_refinements_inexact_configs} configuration (d), the refinement pattern introduces two harmonics because it transitions from refinement of two topological balls to a single connected component (introducing a harmonic 1-form) with non-trivial first homology group (introducing another harmonic 1-form).

\renewcommand*{\xBoxDim}{4cm}%
\renewcommand*{\yBoxDim}{4cm}%
\renewcommand*{\thistextwidth}{0.24}
\newcommand*{\nodesize}{1.2cm}
\newcommand*{\trinodesize}{1.5cm}
\newcommand*{\figureseparation}{1.0cm}
\begin{figure}[ht!]
	\centering
	%
	% New connected component
	%
	\begin{subfigure}{1\textwidth}
	\centering
		\resizebox{\xBoxDim}{\yBoxDim}{%
		\begin{tikzpicture}
	
                		\mymanualgrid{0,...,4}{0,...,4}{0.25}{gray}{white}
                		\mymanualgrid{1,1.5,2,2.5,3}{1,1.5,2,2.5,3}{1}{black}{cyan}

		\end{tikzpicture}
		}
%
	% Coarse Greville Grid
	\nolinebreak
	\hfill
		\resizebox{\xBoxDim}{\yBoxDim}{%
                \begin{tikzpicture}
			% Background grid
                		\mymanualgrid{0,2,6,12,18,22,24}{0,2,6,12,18,22,24}{0.25}{gray}{white}
                \end{tikzpicture}
			}
	\nolinebreak
	\hfill
	% Fine Greville Grid
		\resizebox{\xBoxDim}{\yBoxDim}{%
                \begin{tikzpicture}
			% Background grid
			 \mymanualgrid{0,1,3,6,9,12,15,18,21,23,24}{0,1,3,6,9,12,15,18,21,23,24} {0.25}{gray}{white}
			 \mymanualgrid{9,12,15}{9,12,15} {0.25}{gray}{red}
			% Greville points
			\plotGrevillePoints{0,1,3,6,9,12,15,18,21,23,24}{0,1,3,6,9,12,15,18,21,23,24}{5}{5}{black}{\nodesize}{\trinodesize}
                \end{tikzpicture}
			}
			\caption{Bicubic splines with maximal smoothness and Greville grid refinement from an empty domain to a topological ball.}
	\end{subfigure}
	\\
	%
	% S0 to ball
	%
	\begin{subfigure}{1\textwidth}
	\centering
		\resizebox{\xBoxDim}{\yBoxDim}{%
		\begin{tikzpicture}
			% Background grid
			\mymanualgrid{0,...,7}{0,...,7}{0.25}{gray}{white}
                		\mymanualgrid{1,1.5,2,2.5,3,3.5,4}{1,1.5,2,2.5,3,3.5,4}{1}{black}{cyan}
                		\mymanualgrid{3,3.5,4,4.5,5,5.5,6}{3,3.5,4,4.5,5,5.5,6}{1}{black}{cyan}

			\end{tikzpicture}
			}
%
	% Coarse Greville Grid
	\nolinebreak
	\hfill
	\resizebox{\xBoxDim}{\yBoxDim}{%
                \begin{tikzpicture}
			% Background grid
			\mymanualgrid{0,2,6,10,14,18,22,26,28}{0,2,6,10,14,18,22,26,28}{0.25}{gray}{white}
			% Greville grid
			\mymanualgrid{6,10,14}{6,10,14}{0.25}{gray}{yellow}
			\mymanualgrid{14,18,22}{14,18,22}{0.25}{gray}{yellow}

			\plotGrevillePoints{0,2,6,10,14,18,22,26,28}{0,2,6,10,14,18,22,26,28}{3}{3}{black}{\nodesize}{\trinodesize}
			\plotGrevillePoints{0,2,6,10,14,18,22,26,28}{0,2,6,10,14,18,22,26,28}{5}{5}{black}{\nodesize}{\trinodesize}
			
                \end{tikzpicture}
			}
%
	% Fine Greville Grid
	\nolinebreak
	\renewcommand*{\nodesize}{0.8cm}
	\renewcommand*{\trinodesize}{1.0cm}
	\hfill
		\resizebox{\xBoxDim}{\yBoxDim}{%
                \begin{tikzpicture}
			% Background grid
			\mymanualgrid{0,1,3,5,7,9,11,13,15,17,19,21,23,25,27,28}{0,1,3,5,7,9,11,13,15,17,19,21,23,25,27,28}{0.25}{gray}{white}			
			% Greville grid
			\mymanualgrid{5,7,9,11,13,15}{5,7,9,11,13,15}{0.25}{gray}{red}			
			\mymanualgrid{13,15,17,19,21,23}{13,15,17,19,21,23}{0.25}{gray}{red}			
			
			\plotGrevillePoints{0,1,3,5,7,9,11,13,15,17,19,21,23,25,27,28}{0,1,3,5,7,9,11,13,15,17,19,21,23,25,27,28}{4,...,7}{4,...,7}{black}{\nodesize}{\trinodesize}
			\plotGrevillePoints{0,1,3,5,7,9,11,13,15,17,19,21,23,25,27,28}{0,1,3,5,7,9,11,13,15,17,19,21,23,25,27,28}{8,...,11}{8,...,11}{black}{\nodesize}{\trinodesize}
                \end{tikzpicture}
			}
			\caption{Biquadratic splines with maximal smoothness and refinement from two topological balls to a single topological ball}
	\end{subfigure}
	\\
	%
	% Ball to S1
	%
	\begin{subfigure}{1\textwidth}
	\centering
		\resizebox{\xBoxDim}{\yBoxDim}{%
		\begin{tikzpicture}
			% Background grid
			\mymanualgrid{0,...,9}{0,...,9}{0.25}{gray}{white}
                		\mymanualgrid{1,1.5,2,2.5,3,3.5,4}{1,1.5,2,2.5,3,3.5,4,4.5,5}{1}{black}{cyan}
                		\mymanualgrid{4,4.5,5,5.5,6,6.5,7,7.5,8}{1,1.5,2,2.5,3,3.5,4}{1}{black}{cyan}
                		\mymanualgrid{5,5.5,6,6.5,7,7.5,8}{4,4.5,5,5.5,6,6.5,7,7.5,8}{1}{black}{cyan}
                		\mymanualgrid{3,3.5,4,4.5,5}{5,5.5,6,6.5,7,7.5,8}{1}{black}{cyan}

		\end{tikzpicture}
		}
%
	% Coarse Greville Grid
	\nolinebreak
	\hfill
		\resizebox{\xBoxDim}{\yBoxDim}{%
                \begin{tikzpicture}
			% Background grid
			\mymanualgrid{0,2,6,10,14,18,22,26,30,34,36}{0,2,6,10,14,18,22,26,30,34,36}{0.25}{gray}{white}

			% Greville mesh
			\mymanualgrid{6,10,14}{6,10,14,18}{0.25}{gray}{yellow}
			\mymanualgrid{14,18,22,26,30}{6,10,14}{0.25}{gray}{yellow}
			\mymanualgrid{22,26,30}{14,18,22,26,30}{0.25}{gray}{yellow}
			\mymanualgrid{14,18,22}{22,26,30}{0.25}{gray}{yellow}
			
			% Greville Points
			\plotGrevillePoints{0,2,6,10,14,18,22,26,30,34,36}{0,2,6,10,14,18,22,26,30,34,36}{3}{3,4}{black}{\nodesize}{\trinodesize}
			\plotGrevillePoints{0,2,6,10,14,18,22,26,30,34,36}{0,2,6,10,14,18,22,26,30,34,36}{4,...,7}{3}{black}{\nodesize}{\trinodesize}
			\plotGrevillePoints{0,2,6,10,14,18,22,26,30,34,36}{0,2,6,10,14,18,22,26,30,34,36}{7}{4,...,7}{black}{\nodesize}{\trinodesize}
			\plotGrevillePoints{0,2,6,10,14,18,22,26,30,34,36}{0,2,6,10,14,18,22,26,30,34,36}{5,6}{7}{black}{\nodesize}{\trinodesize}
                \end{tikzpicture}
			}
%
	% Fine Greville Grid
	\nolinebreak
	\renewcommand*{\nodesize}{0.6cm}
	\renewcommand*{\trinodesize}{0.8cm}
	\hfill
		\resizebox{\xBoxDim}{\yBoxDim}{%
                \begin{tikzpicture}
			% Background grid
			\mymanualgrid{0,1,3,5,7,9,11,13,15,17,19,21,23,25,27,29,31,33,35,36}{0,1,3,5,7,9,11,13,15,17,19,21,23,25,27,29,31,33,35,36}{0.25}{gray}{white}

			% Greville mesh
			\mymanualgrid{5,7,9,11,13,15}{5,7,9,11,13,15,17,19}{0.25}{gray}{red}
			\mymanualgrid{15,17,19,21,23,25,27,29,31}{5,7,9,11,13,15}{0.25}{gray}{red}
			\mymanualgrid{21,23,25,27,29,31}{15,17,19,21,23,25,27,29,31}{0.25}{gray}{red}
			\mymanualgrid{13,15,17,19,21}{21,23,25,27,29,31}{0.25}{gray}{red}
			\mymanualgrid{13,15}{19,21}{0.25}{gray}{red}

			% Greville points
			\plotGrevillePoints{0,1,3,5,7,9,11,13,15,17,19,21,23,25,27,29,31,33,35,36}{0,1,3,5,7,9,11,13,15,17,19,21,23,25,27,29,31,33,35,36}{4,...,7}{4,...,9}{black}{\nodesize}{\trinodesize}
			\plotGrevillePoints{0,1,3,5,7,9,11,13,15,17,19,21,23,25,27,29,31,33,35,36}{0,1,3,5,7,9,11,13,15,17,19,21,23,25,27,29,31,33,35,36}{8,...,15}{4,...,7}{black}{\nodesize}{\trinodesize}
			\plotGrevillePoints{0,1,3,5,7,9,11,13,15,17,19,21,23,25,27,29,31,33,35,36}{0,1,3,5,7,9,11,13,15,17,19,21,23,25,27,29,31,33,35,36}{12,...,15}{8,...,15}{black}{\nodesize}{\trinodesize}
			\plotGrevillePoints{0,1,3,5,7,9,11,13,15,17,19,21,23,25,27,29,31,33,35,36}{0,1,3,5,7,9,11,13,15,17,19,21,23,25,27,29,31,33,35,36}{8,...,11}{12,...,15}{black}{\nodesize}{\trinodesize}

			\node at (14,19) [regular polygon, regular polygon sides=3,fill=black,minimum size=\trinodesize,rotate=-90]{};
			\node at (14,21) [regular polygon, regular polygon sides=3,fill=black,minimum size=\trinodesize,rotate=-90]{};
			\node at (14,20) [rectangle,fill=black,minimum size=\nodesize]{};

                \end{tikzpicture}
			}
			\caption{Biquadratic splines with maximal smoothness and refinement from a topological ball to a topological annulus}
	\end{subfigure}
	\\
	%
	% S0 to S1
	%
	\begin{subfigure}{1\textwidth}
	\centering
		\resizebox{\xBoxDim}{\yBoxDim}{%
		\begin{tikzpicture}
			% Background grid
			\mymanualgrid{0,...,9}{0,...,9}{0.25}{gray}{white}
			% refinement
			\mymanualgrid{1,1.5,2,2.5,3,3.5,4}{1,1.5,2,2.5,3,3.5,4,4.5,5}{0.25}{black}{cyan}
			\mymanualgrid{4,4.5,5,5.5,6}{1,1.5,2,2.5,3,3.5,4}{0.25}{black}{cyan}
			\mymanualgrid{5,5.5,6,6.5,7,7.5,8}{4,4.5,5,5.5,6,6.5,7,7.5,8}{0.25}{black}{cyan}
			\mymanualgrid{3,3.5,4,4.5,5}{5,5.5,6,6.5,7,7.5,8}{0.25}{black}{cyan}
			\end{tikzpicture}
			}
%
	% Coarse Greville Grid
	\renewcommand*{\nodesize}{0.8cm}
	\renewcommand*{\trinodesize}{1.0cm}
	\nolinebreak
	\hfill
		\resizebox{\xBoxDim}{\yBoxDim}{%
                \begin{tikzpicture}
			% Background grid
			\mymanualgrid{0,2,6,10,14,18,22,26,30,34,36}{0,2,6,10,14,18,22,26,30,34,36}{0.25}{gray}{white}

			% Greville mesh
			\mymanualgrid{6,10,14}{6,10,14,18}{0.25}{gray}{yellow}
			\mymanualgrid{14,18,22}{6,10,14}{0.25}{gray}{yellow}
			\mymanualgrid{22,26,30}{18,22,26,30}{0.25}{gray}{yellow}
			\mymanualgrid{14,18,22}{22,26,30}{0.25}{gray}{yellow}
			
			% Greville Points
			\plotGrevillePoints{0,2,6,10,14,18,22,26,30,34,36}{0,2,6,10,14,18,22,26,30,34,36}{3}{3,4}{black}{\nodesize}{\trinodesize}
			\plotGrevillePoints{0,2,6,10,14,18,22,26,30,34,36}{0,2,6,10,14,18,22,26,30,34,36}{4,5}{3}{black}{\nodesize}{\trinodesize}
			\plotGrevillePoints{0,2,6,10,14,18,22,26,30,34,36}{0,2,6,10,14,18,22,26,30,34,36}{7}{6,7}{black}{\nodesize}{\trinodesize}
			\plotGrevillePoints{0,2,6,10,14,18,22,26,30,34,36}{0,2,6,10,14,18,22,26,30,34,36}{5,6}{7}{black}{\nodesize}{\trinodesize}
                \end{tikzpicture}
			}
%
	% Fine Greville Grid
	\renewcommand*{\nodesize}{0.6cm}
	\renewcommand*{\trinodesize}{0.8cm}
	\nolinebreak
	\hfill
		\resizebox{\xBoxDim}{\yBoxDim}{%
                \begin{tikzpicture}
			% Background grid
			\mymanualgrid{0,1,3,5,7,9,11,13,15,17,19,21,23,25,27,29,31,33,35,36}{0,1,3,5,7,9,11,13,15,17,19,21,23,25,27,29,31,33,35,36}{0.25}{gray}{white}

			% Greville mesh
			\mymanualgrid{5,7,9,11,13,15}{5,7,9,11,13,15,17,19}{0.25}{gray}{red}
			\mymanualgrid{15,17,19,21,23}{5,7,9,11,13,15}{0.25}{gray}{red}
			\mymanualgrid{21,23,25,27,29,31}{17,19,21,23,25,27,29,31}{0.25}{gray}{red}
			\mymanualgrid{13,15,17,19,21}{21,23,25,27,29,31}{0.25}{gray}{red}
			\mymanualgrid{13,15}{19,21}{0.25}{gray}{red}
			\mymanualgrid{21,23}{15,17}{0.25}{gray}{red}
			
			% Greville points
			\plotGrevillePoints{0,1,3,5,7,9,11,13,15,17,19,21,23,25,27,29,31,33,35,36}{0,1,3,5,7,9,11,13,15,17,19,21,23,25,27,29,31,33,35,36}{4,...,7}{4,...,9}{black}{\nodesize}{\trinodesize}
			\plotGrevillePoints{0,1,3,5,7,9,11,13,15,17,19,21,23,25,27,29,31,33,35,36}{0,1,3,5,7,9,11,13,15,17,19,21,23,25,27,29,31,33,35,36}{8,...,11}{4,...,7}{black}{\nodesize}{\trinodesize}
			\plotGrevillePoints{0,1,3,5,7,9,11,13,15,17,19,21,23,25,27,29,31,33,35,36}{0,1,3,5,7,9,11,13,15,17,19,21,23,25,27,29,31,33,35,36}{12,...,15}{10,...,15}{black}{\nodesize}{\trinodesize}
			\plotGrevillePoints{0,1,3,5,7,9,11,13,15,17,19,21,23,25,27,29,31,33,35,36}{0,1,3,5,7,9,11,13,15,17,19,21,23,25,27,29,31,33,35,36}{8,...,11}{12,...,15}{black}{\nodesize}{\trinodesize}
			
			% extra nodes top left
			\node at (14,19) [regular polygon, regular polygon sides=3,fill=black,minimum size=\trinodesize,rotate=-90]{};
			\node at (14,21) [regular polygon, regular polygon sides=3,fill=black,minimum size=\trinodesize,rotate=-90]{};
			\node at (14,20) [rectangle,fill=black,minimum size=\nodesize]{};

			% extra nodes bottom right
			\node at (22,15) [regular polygon, regular polygon sides=3,fill=black,minimum size=\trinodesize,rotate=-90]{};
			\node at (22,17) [regular polygon, regular polygon sides=3,fill=black,minimum size=\trinodesize,rotate=-90]{};
			\node at (22,16) [rectangle,fill=black,minimum size=\nodesize]{};

                \end{tikzpicture}
			}
			\caption{Biquadratic splines with maximal smoothness and refinement from two topological balls to a topological annulus}
	\end{subfigure}
	\caption{Bezi\'er meshes (left), coarse Greville subgrids $\gmesh{0,1}$ (center), and refined Greville subgrid $\gmesh{1,1}$ (right) for various inexact refinement schemes are shown.
	Inexactness can be visualized by the changes in the topology between $\gmesh{0,1}$ and $\gmesh{1,1}$ for the various configurations.
	Here, the filled disks correspond to Greville 0-cells, filled triangles correspond to Greville 1-cells, and filled squares correspond to Greville 2-cells.
	}
	\label{fig:2d_refinements_inexact_configs}
\end{figure}
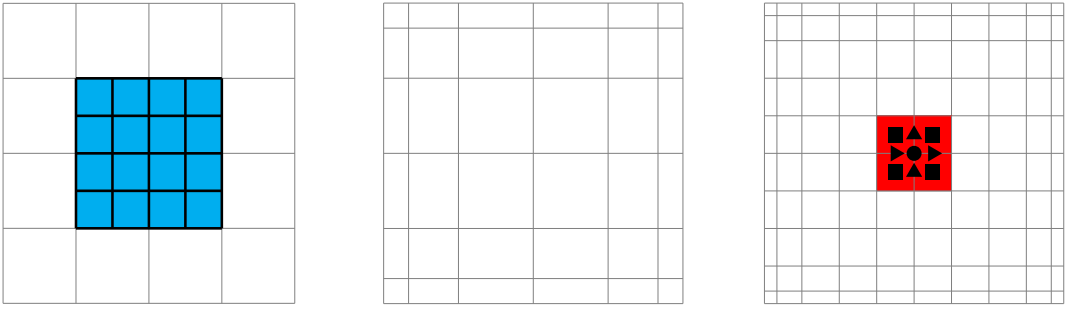
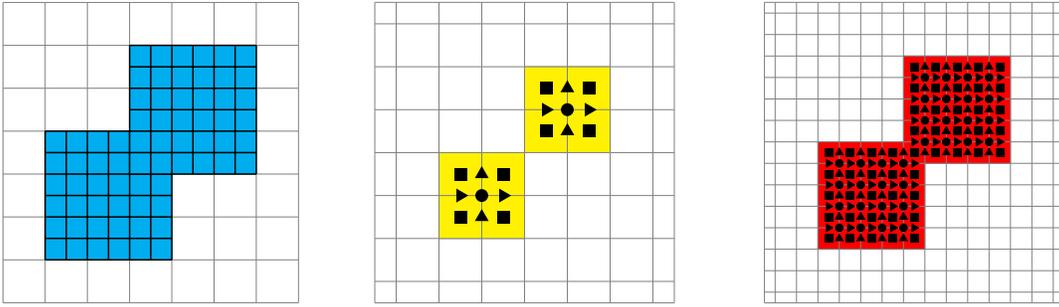
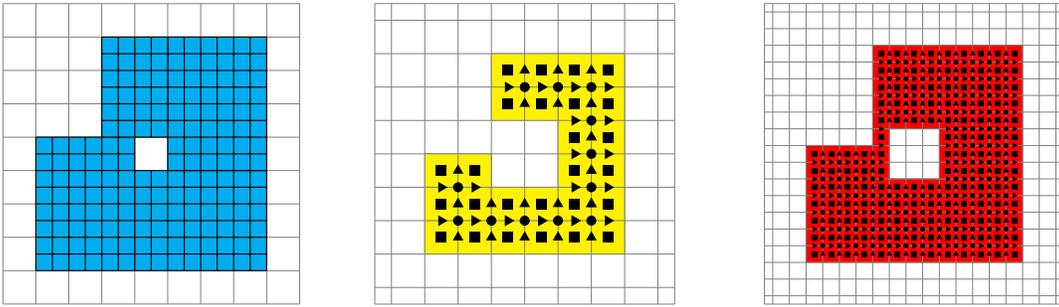
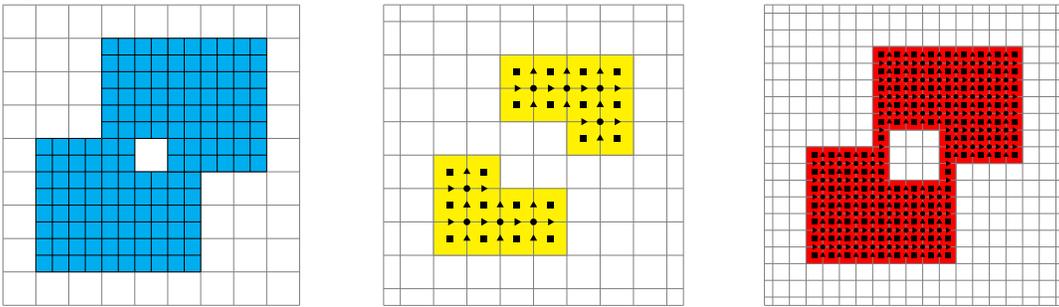
\begin{figure}
	\centering
	\begin{subfigure}{0.31\textwidth}
		\includegraphics[width=1\textwidth,trim=22.4cm 7.9cm 21.8cm 12.55cm,clip]{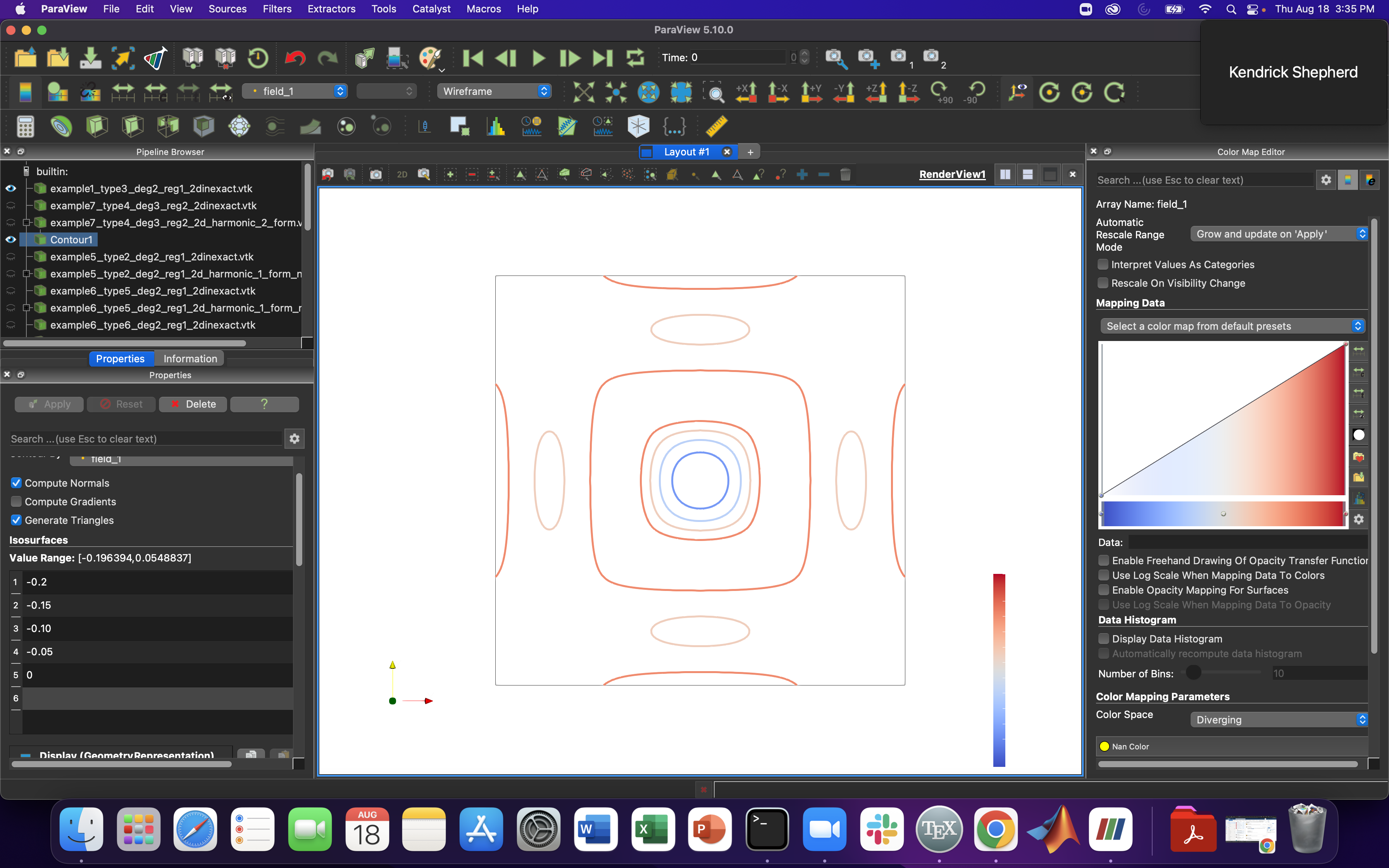}
		\caption{$2$-form, Figure \ref{fig:2d_refinements_inexact_configs}(a)}
	\end{subfigure}
	\;
	\begin{subfigure}{0.313\textwidth}
		\includegraphics[width=1\textwidth,trim=22.4cm 7.9cm 21.8cm 12.55cm,clip]{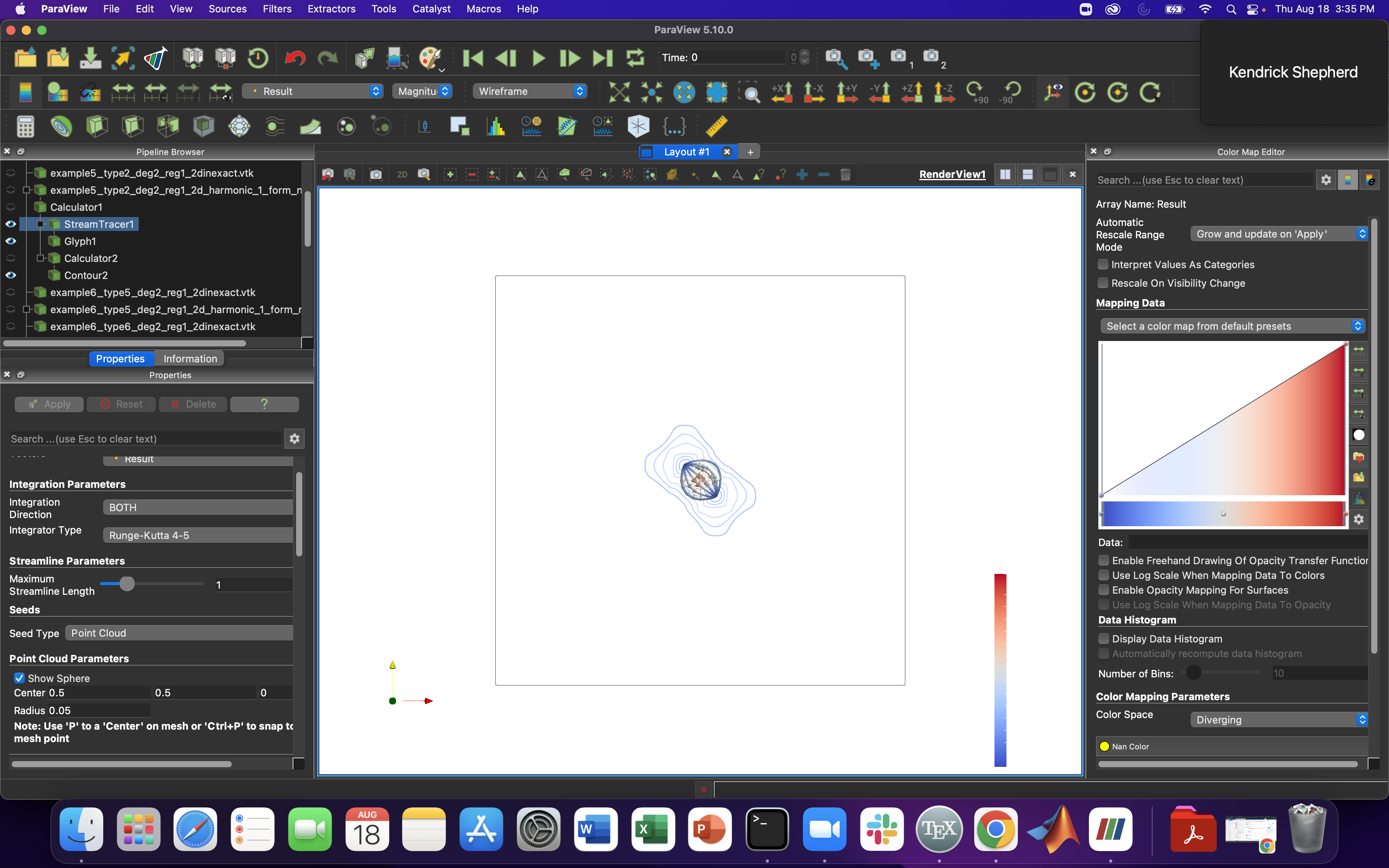}
		\caption{$1$-form, Figure \ref{fig:2d_refinements_inexact_configs}(b)}
	\end{subfigure}
	\;
	\begin{subfigure}{0.31\textwidth}
		\includegraphics[width=1\textwidth,trim=22.4cm 7.9cm 21.8cm 12.55cm,clip]{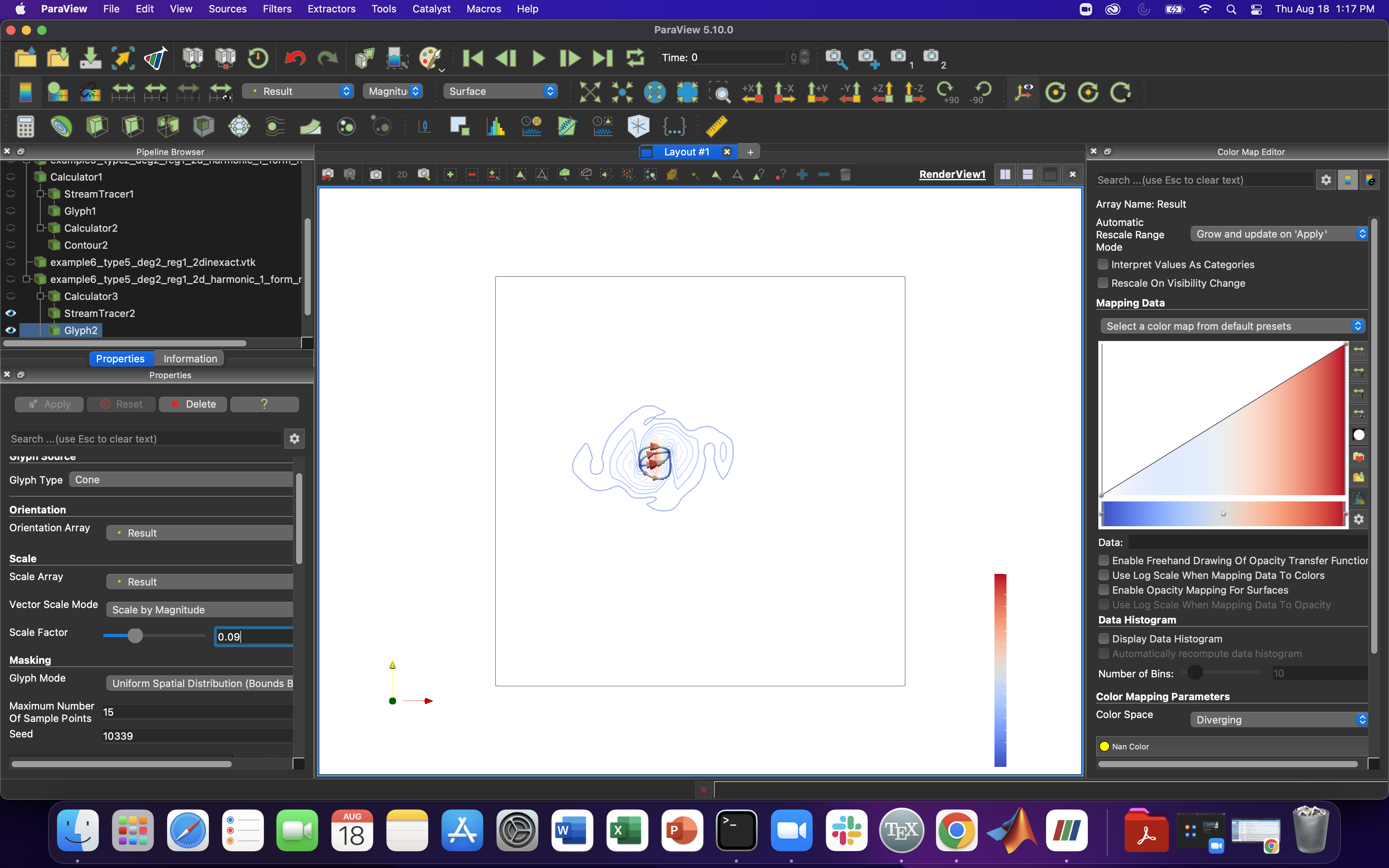}
		\caption{$1$-form, Figure \ref{fig:2d_refinements_inexact_configs}(c)}
	\end{subfigure}
	\\
		\begin{subfigure}{0.65\textwidth}
		\includegraphics[width=0.49\textwidth,trim=22.4cm 7.9cm 21.8cm 12.55cm,clip]{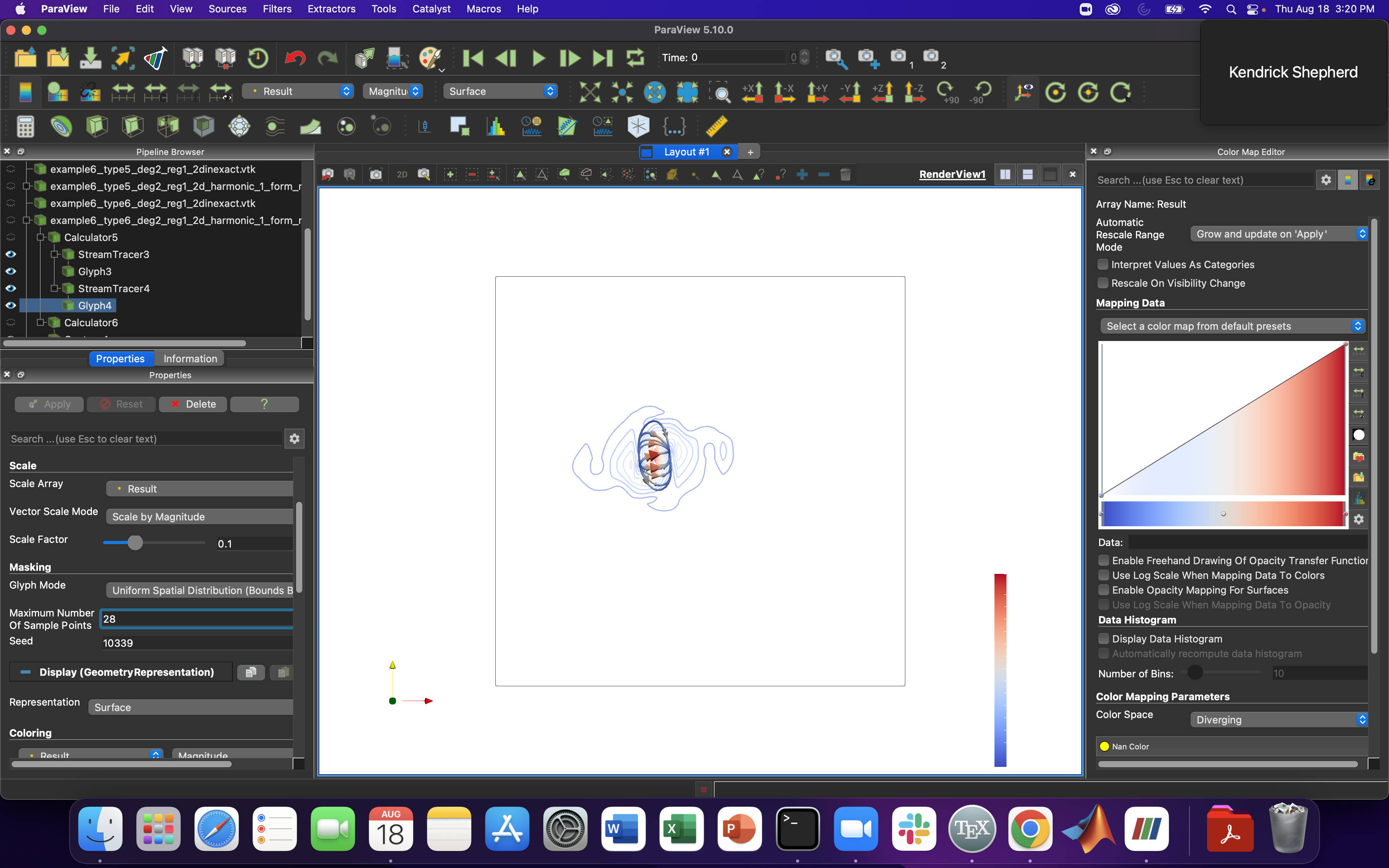}
		\hspace{1em}
		\includegraphics[width=0.49\textwidth,trim=22.4cm 7.9cm 21.8cm 12.55cm,clip]{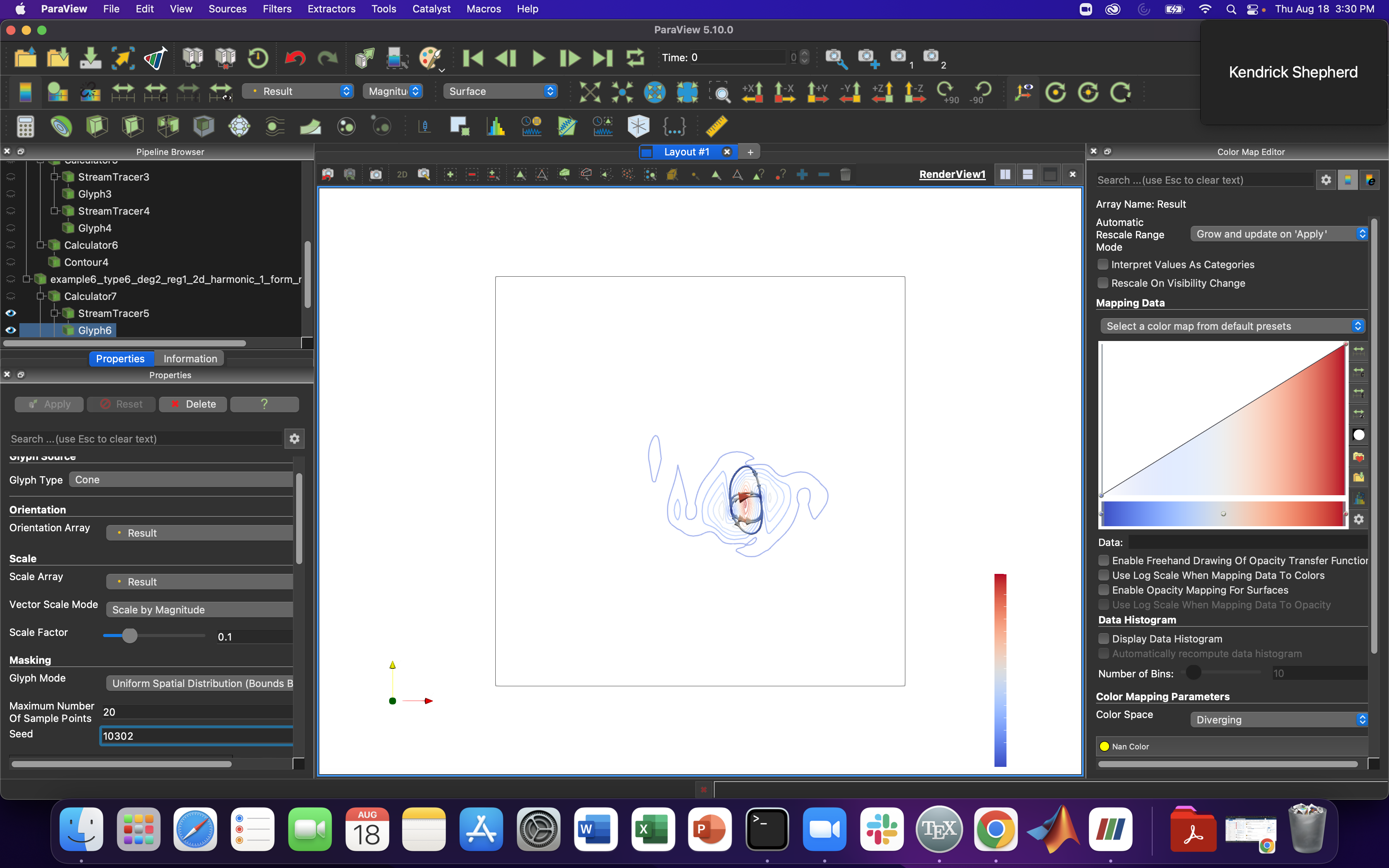}
		\caption{$1$-forms, Figure \ref{fig:2d_refinements_inexact_configs}(d)}
	\end{subfigure}
	\caption{Spurious harmonic forms introduced by refinement patterns of Figure \ref{fig:2d_refinements_inexact_configs} are shown.}
	\label{fig:2d_refinements_inexact_harmonics}
\end{figure}

\vspace{\baselineskip}
\noindent
\emph{\textbf{Three-dimensional examples}}\\
Examples of refinement configurations resulting in spurious harmonics in three dimensions are depicted in Figure \ref{fig:3d_harmonic_configurations}.
Similarly to two dimensions, each of these harmonics corresponds to a change in topology between $\gmesh{0,1}$ and $\gmesh{1,1}$.
Specifically,
harmonic 3-forms are introduced by violating Assumption \ref{assume:zero_union} and including splines in the fine space that do not correspond to removal of splines in the coarse space.
Harmonic 2-forms are introduced both by modifying $\gmesh{0,1}$ from a topological ball to a topological solid torus on $\gmesh{1,1}$ or by modifying $\gmesh{0,1}$ from being two topological balls into a single topological ball in $\gmesh{1,1}$.
Similarly, harmonic 1-forms are produced by refinement patterns converting $\gmesh{0,1}$ from a topological ball into a configuration in $\gmesh{1,1}$ with an internal cavity or by refinement that converts $\gmesh{0,1}$ from being a solid torus into a topological ball in $\gmesh{1,1}$.

\begin{figure}
	\centering
	\begin{subfigure}{0.485\textwidth}
		\includegraphics[width=0.48\textwidth,trim=18cm 2.5cm 17cm 11cm,clip]{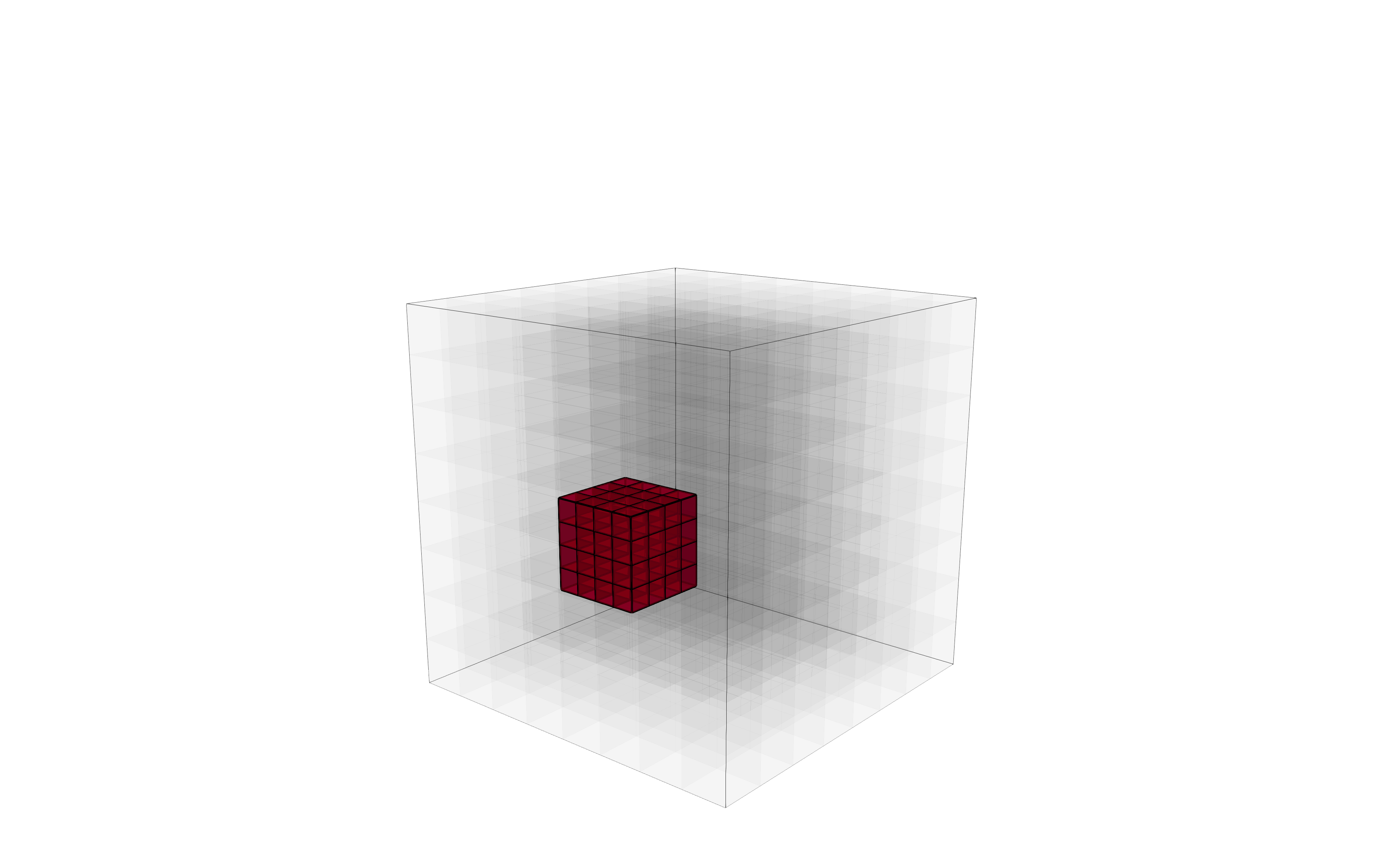}
		\;
		\includegraphics[width=0.48\textwidth,trim=19.5cm 5cm 20cm 12.5cm,clip]{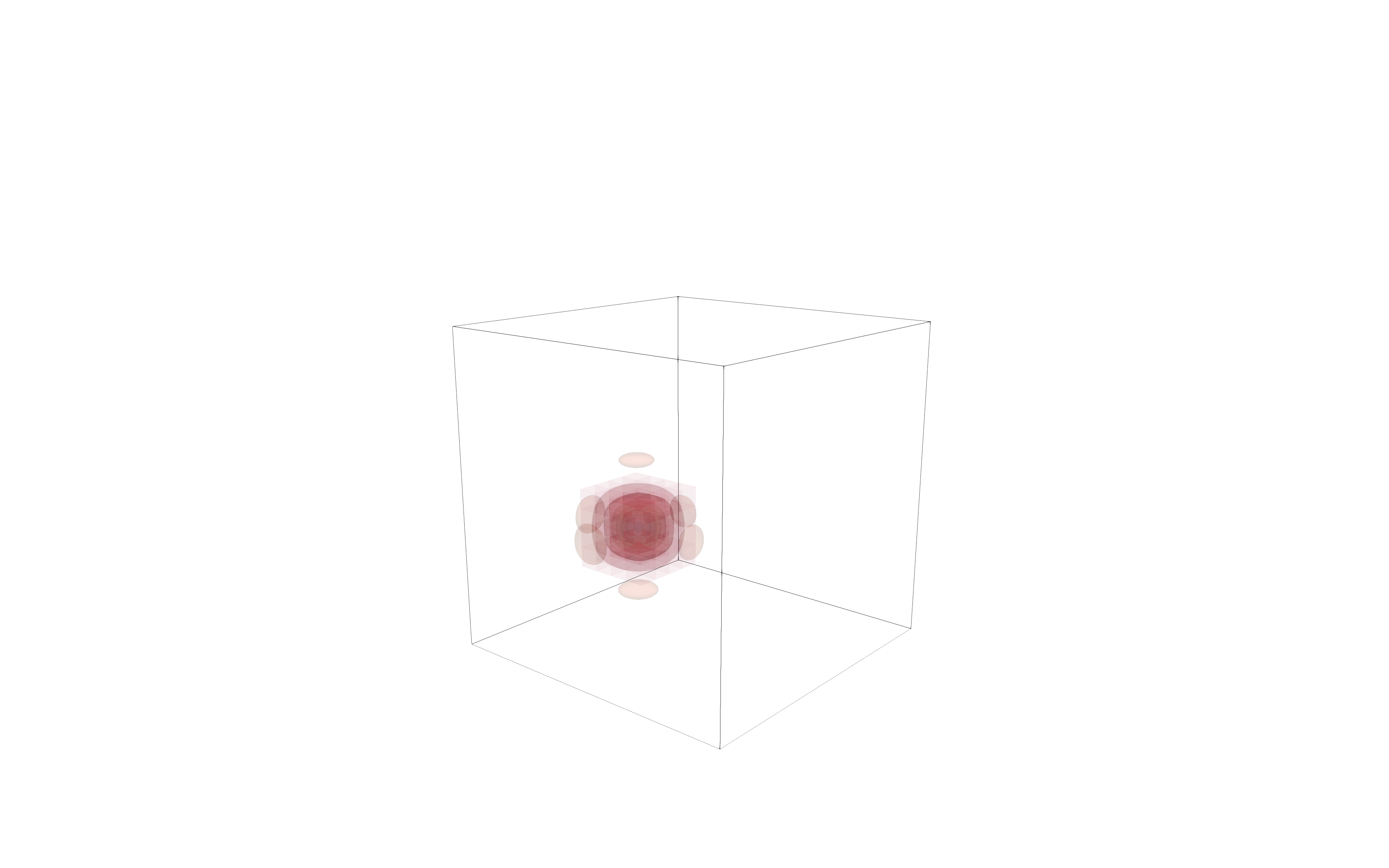}
		\caption{\Og Harmonic $3$-form refinement: $p=3$}
	\end{subfigure}
	\;
	\begin{subfigure}{0.485\textwidth}
		\includegraphics[width=0.48\textwidth,trim=18cm 2.5cm 17cm 11cm,clip]{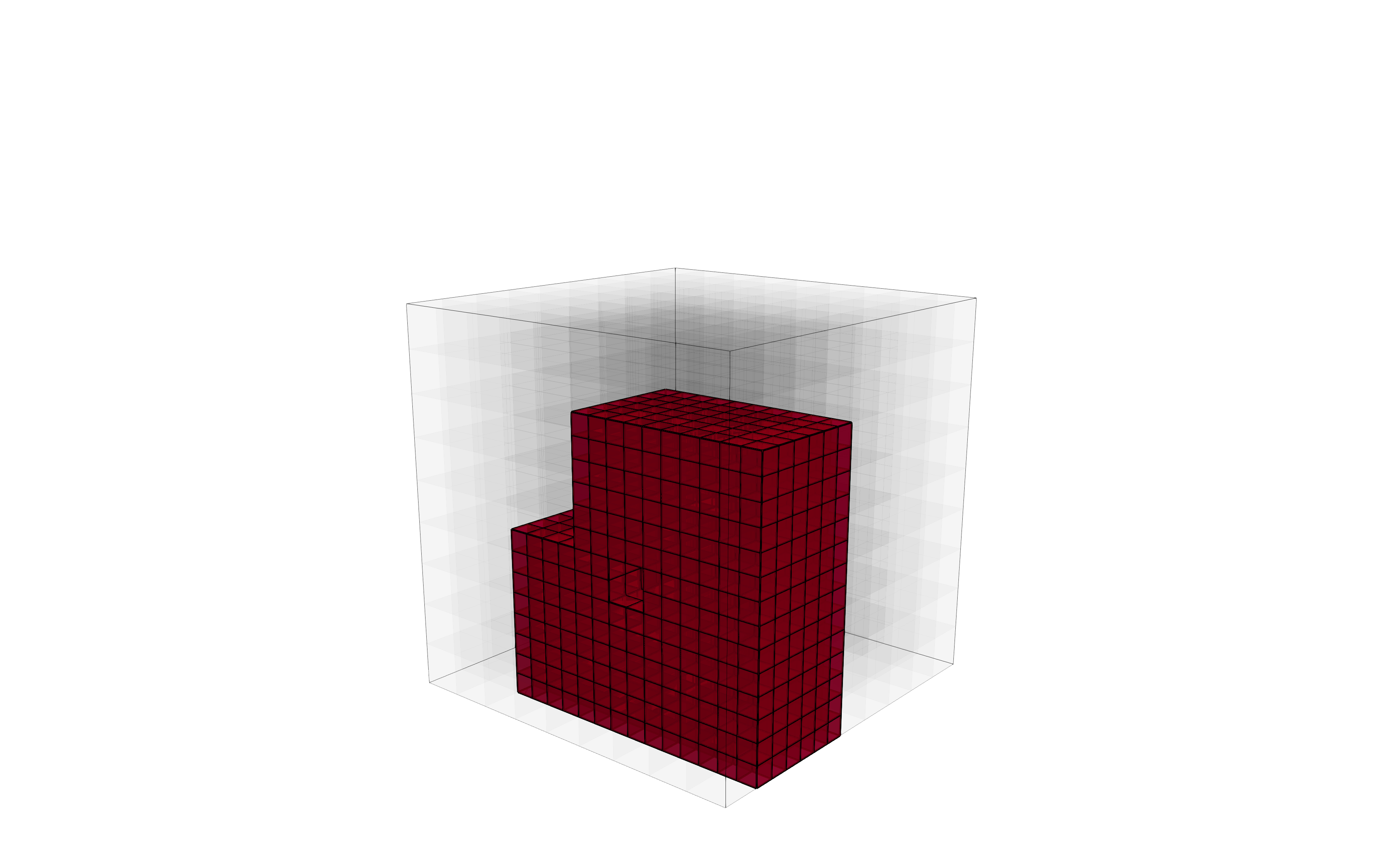}
		\;
		\includegraphics[width=0.48\textwidth,trim=19.5cm 5cm 20cm 12.5cm,clip]{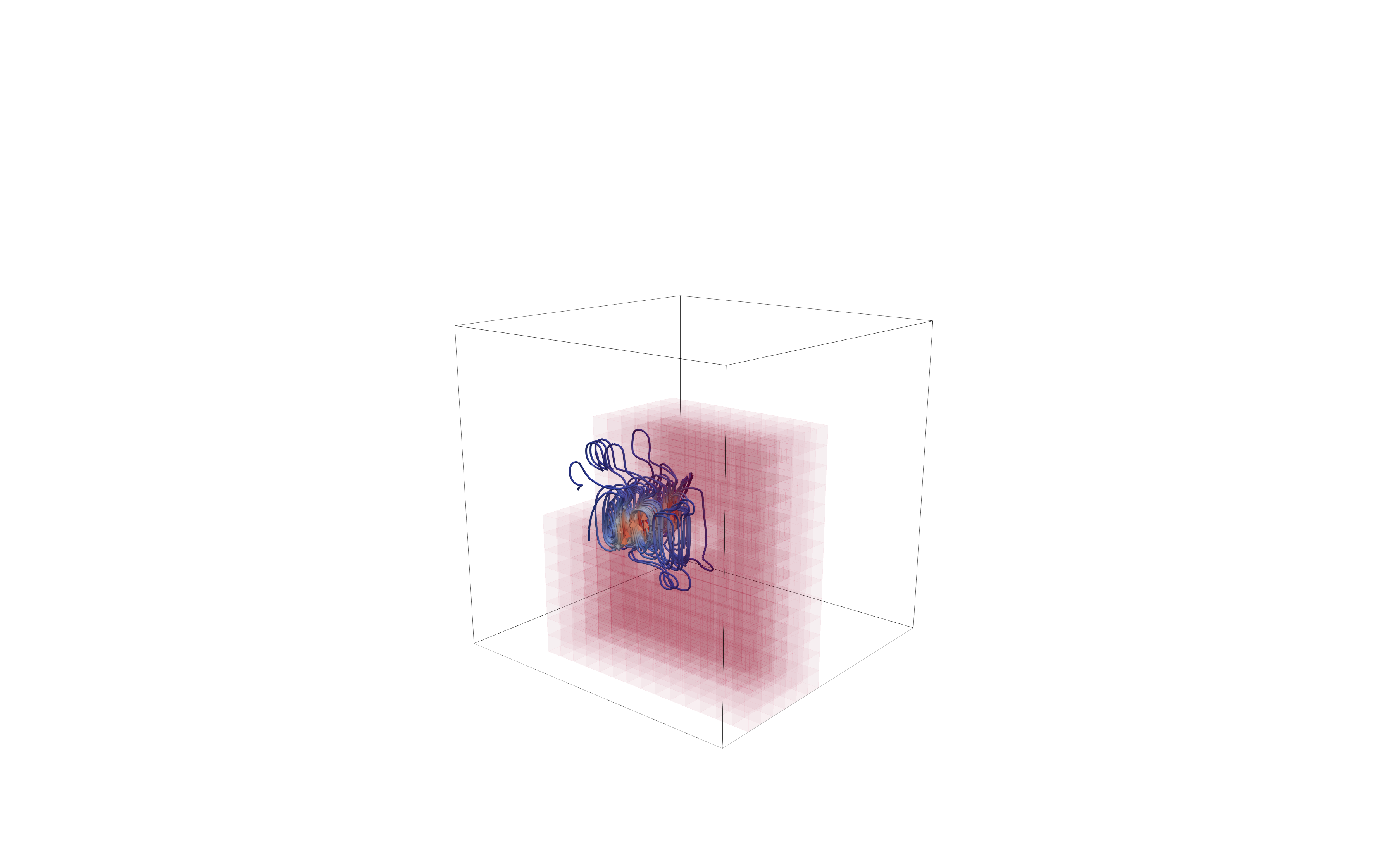}
		\caption{\Og Harmonic $2$-form refinement: $p=2$}
	\end{subfigure}
	\\
	\begin{subfigure}{0.485\textwidth}
		\includegraphics[width=0.48\textwidth,trim=18cm 2.5cm 17cm 11cm,clip]{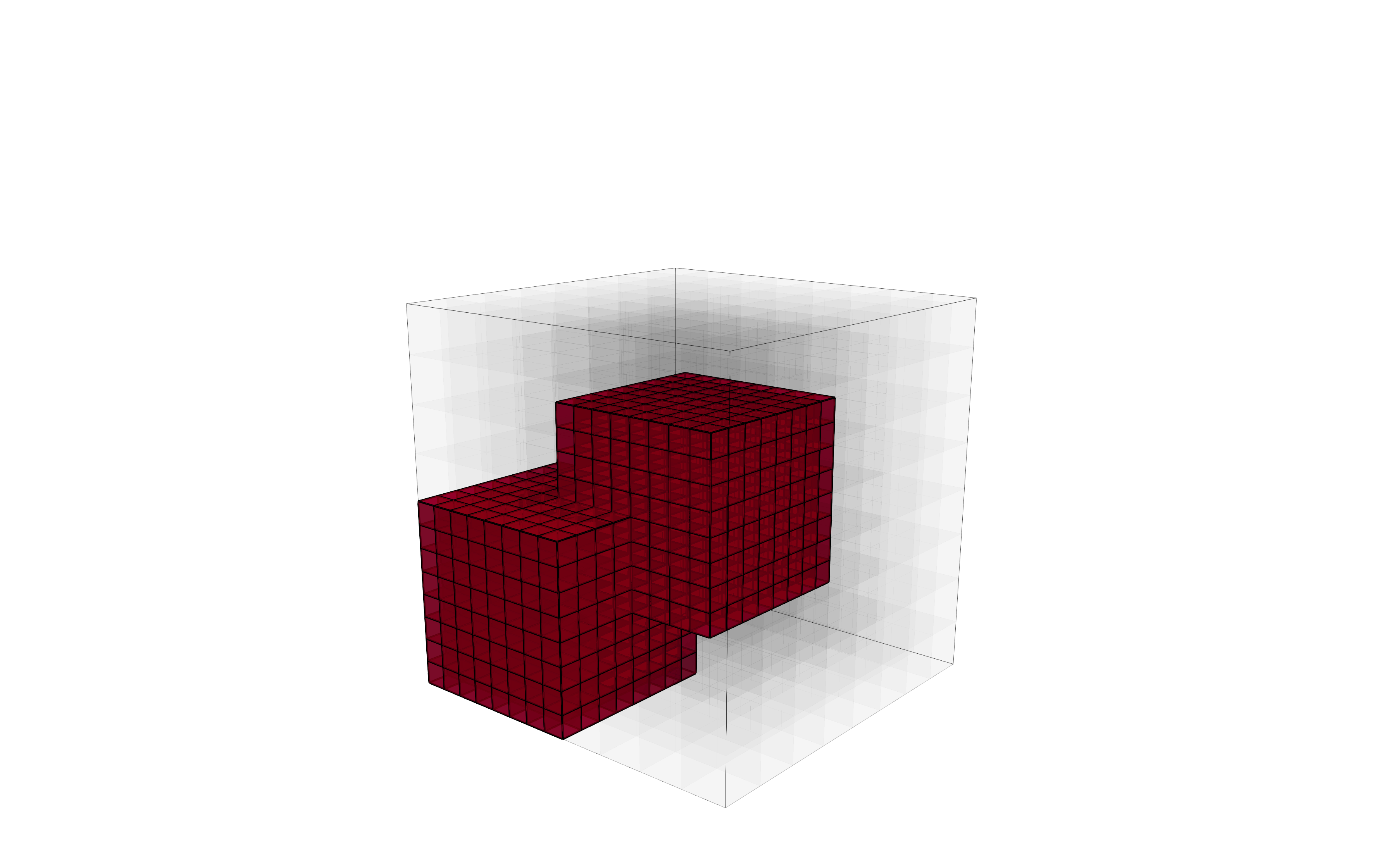}
		\;
		\includegraphics[width=0.48\textwidth,trim=19.5cm 5cm 20cm 12.5cm,clip]{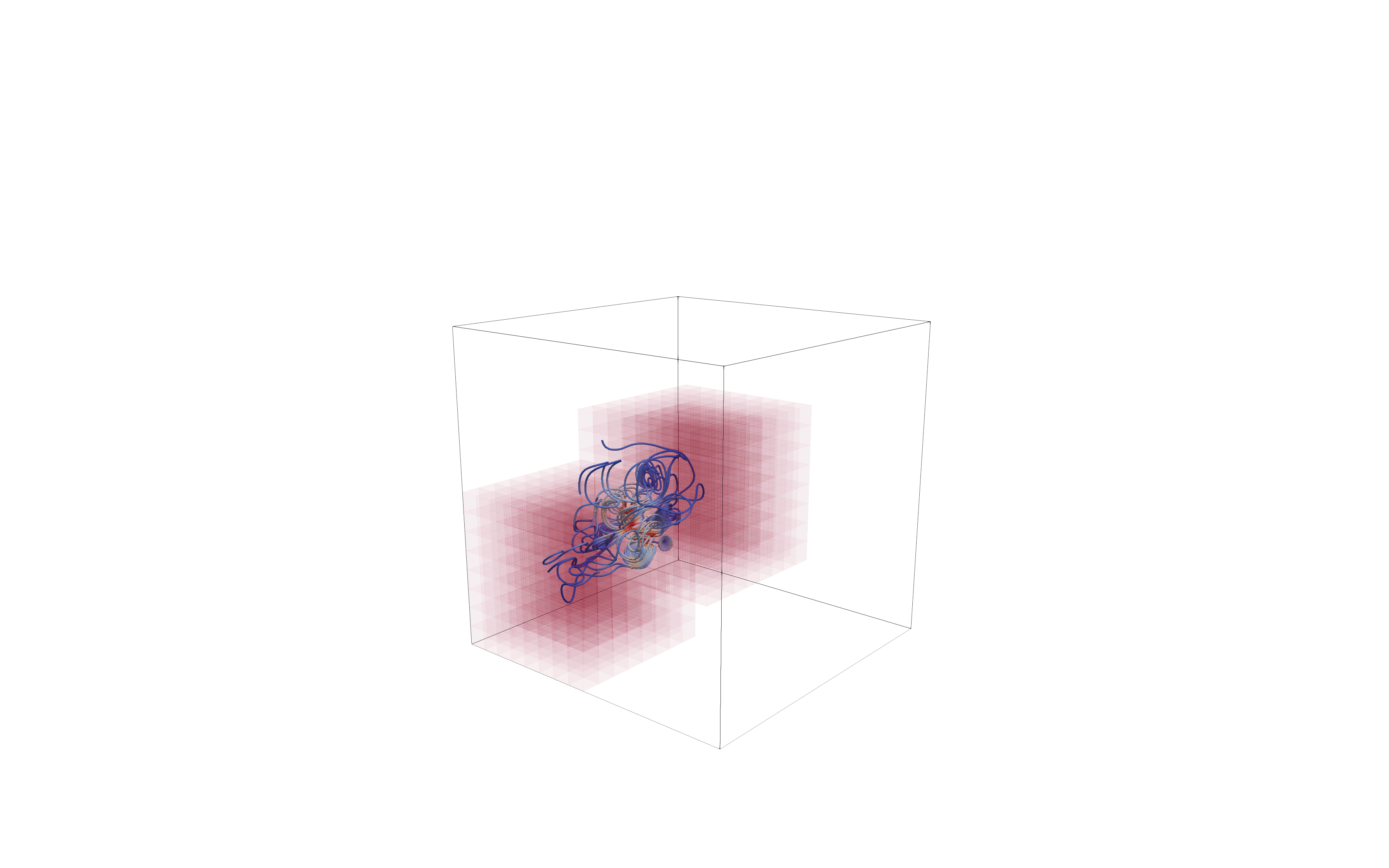}
		\caption{\Og Harmonic $2$-form refinement: $p=3$}
	\end{subfigure}
	\;
	\begin{subfigure}{0.485\textwidth}
		\includegraphics[width=0.48\textwidth,trim=18cm 2.5cm 17cm 11cm,clip]{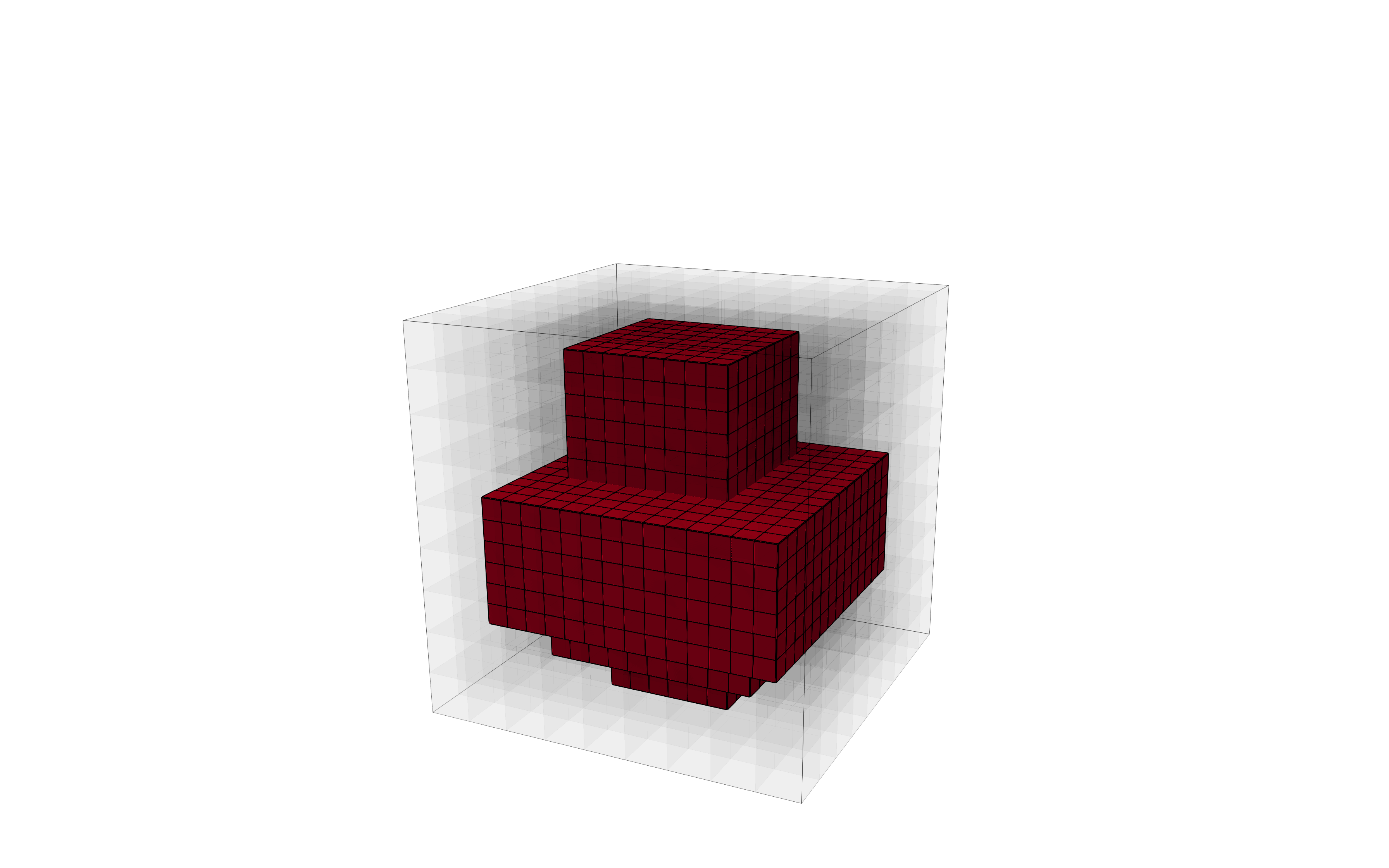}
		\;
		\includegraphics[width=0.48\textwidth,trim=19.5cm 5cm 20cm 12.5cm,clip]{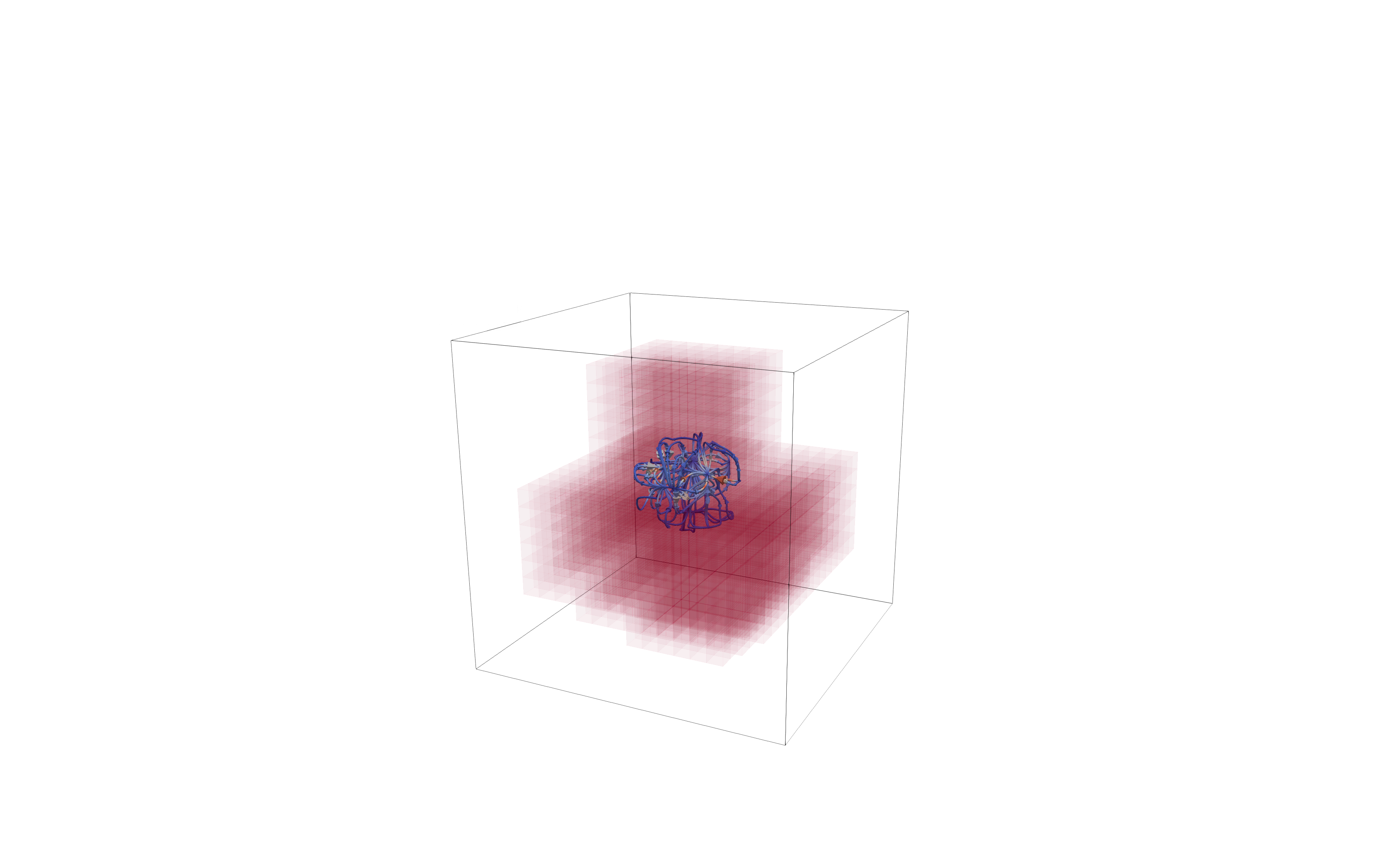}
		\caption{\Og Harmonic $1$-form refinement: $p=2$}
	\end{subfigure}
	\\
	\begin{subfigure}{0.76\textwidth}
		\includegraphics[width=0.315\textwidth,trim=18cm 2.5cm 17cm 11cm,clip]{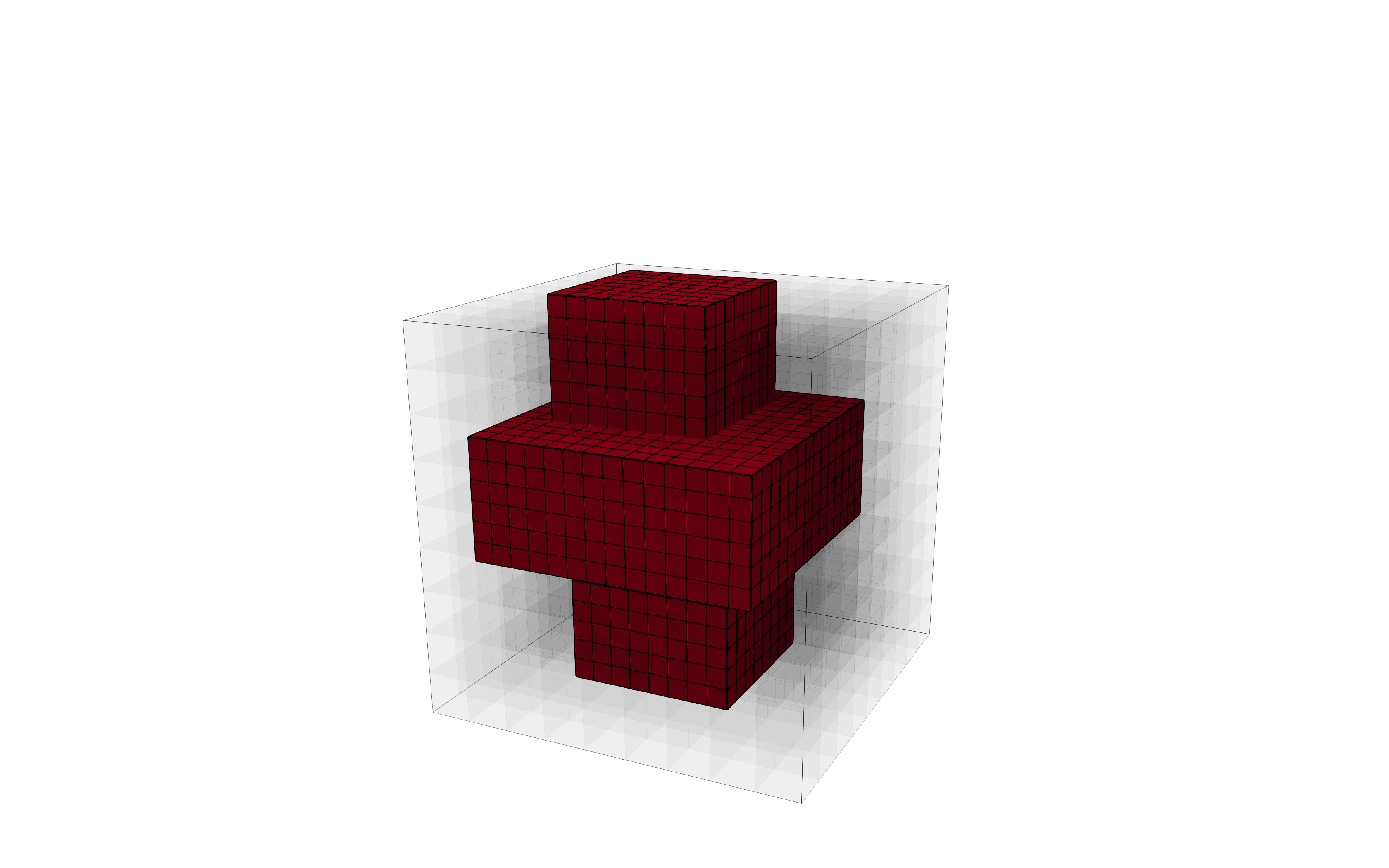}
		\;
		\includegraphics[width=0.315\textwidth,trim=19.5cm 5cm 20cm 12.5cm,clip]{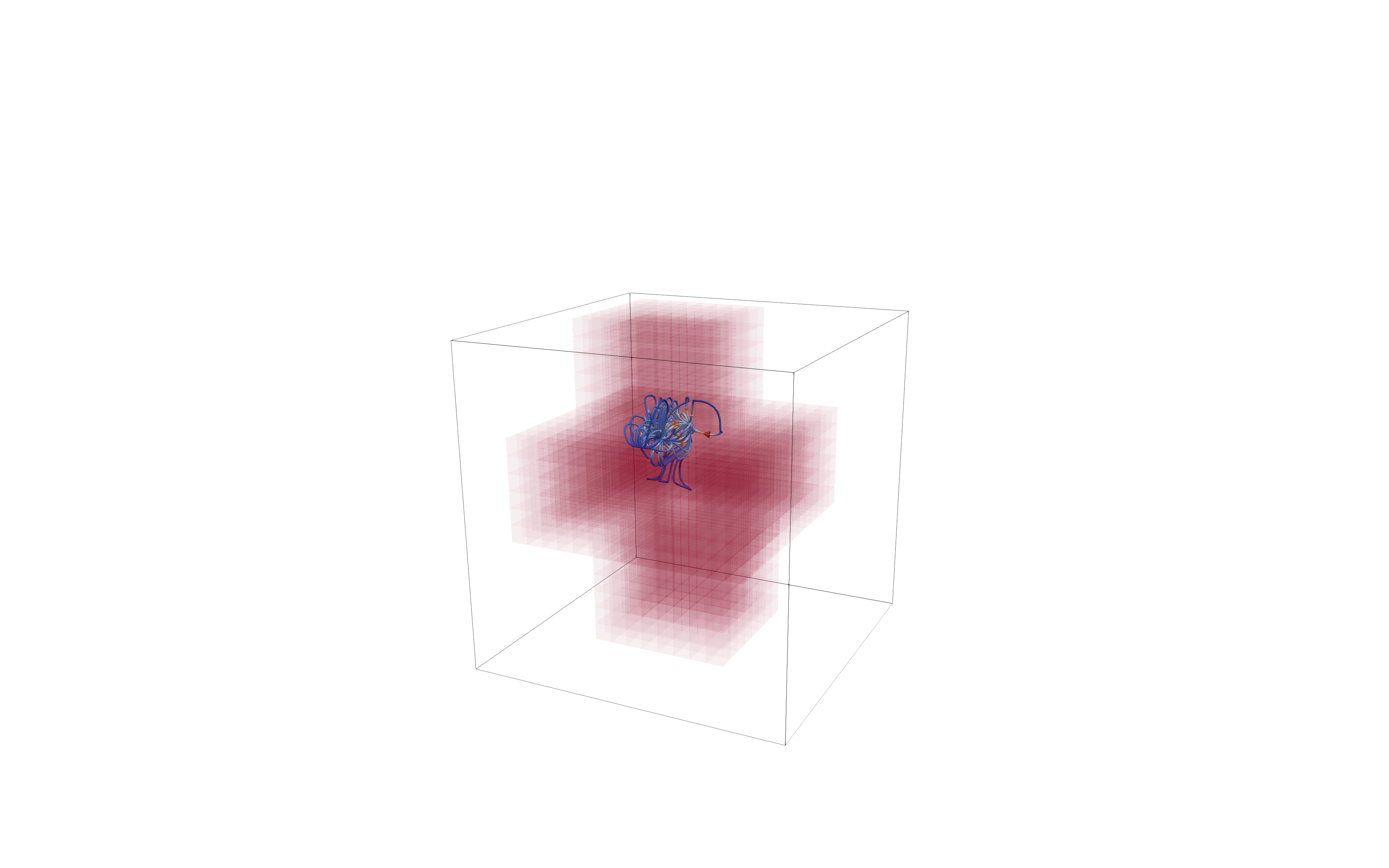}
		\;
		\includegraphics[width=0.315\textwidth,trim=19.5cm 5cm 20cm 12.5cm,clip]{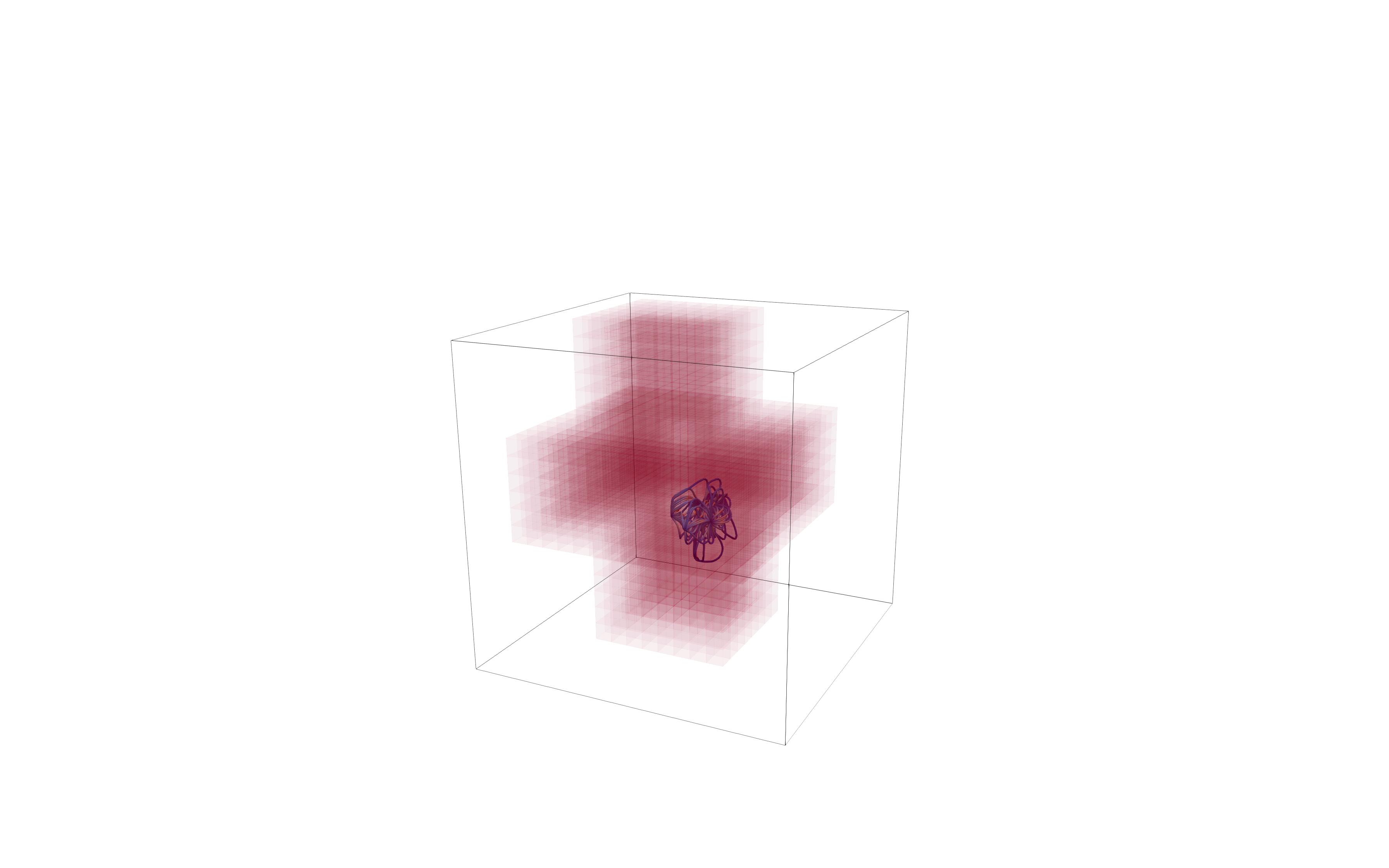}
		\caption{\Og Refinement with two harmonic $1$-forms: $p=2$}
	\end{subfigure}
	\\
	\begin{subfigure}{0.485\textwidth}
		\includegraphics[width=0.48\textwidth,trim=18cm 2.5cm 17cm 11cm,clip]{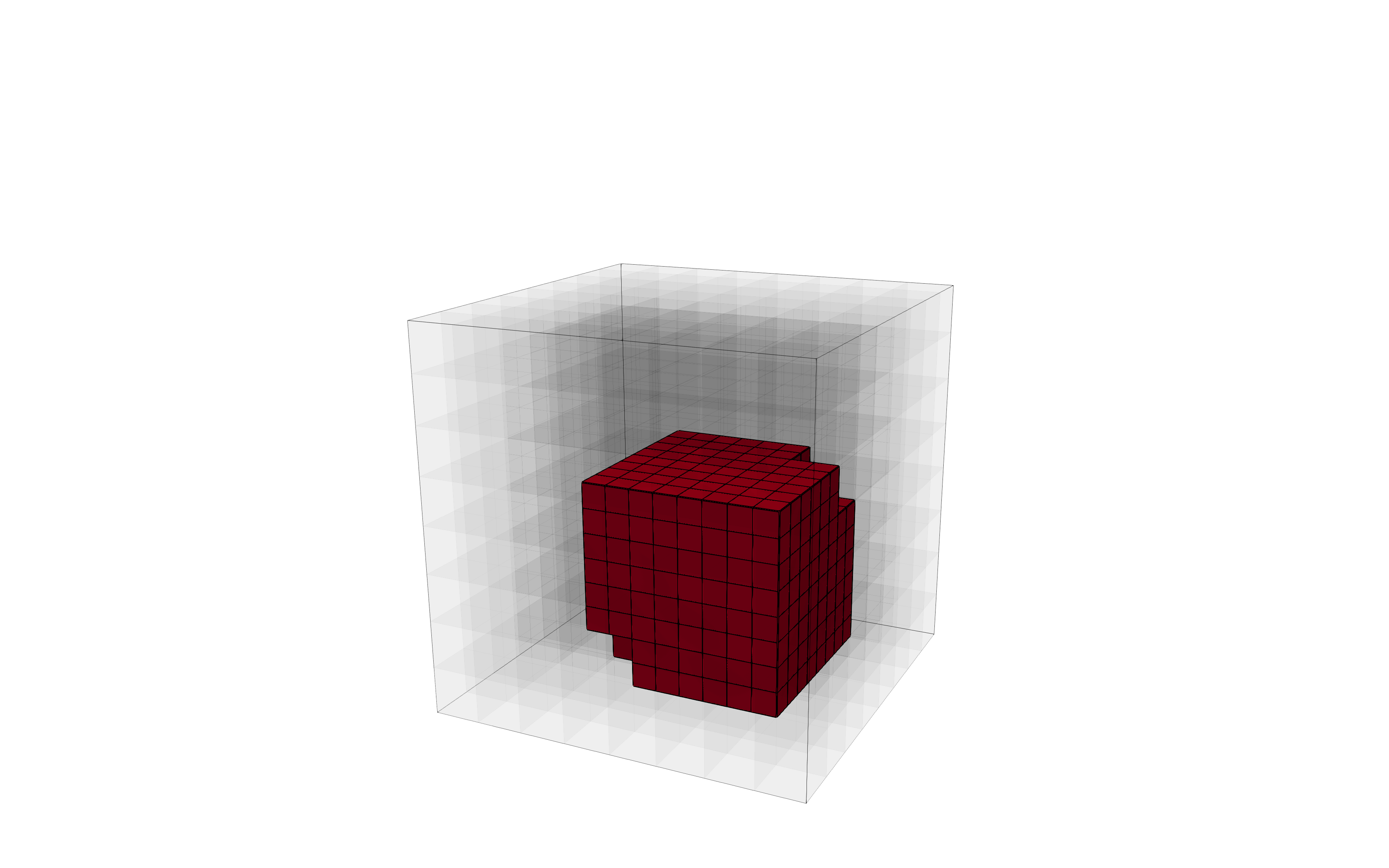}
		\;
		\includegraphics[width=0.48\textwidth,trim=19.5cm 5cm 20cm 12.5cm,clip]{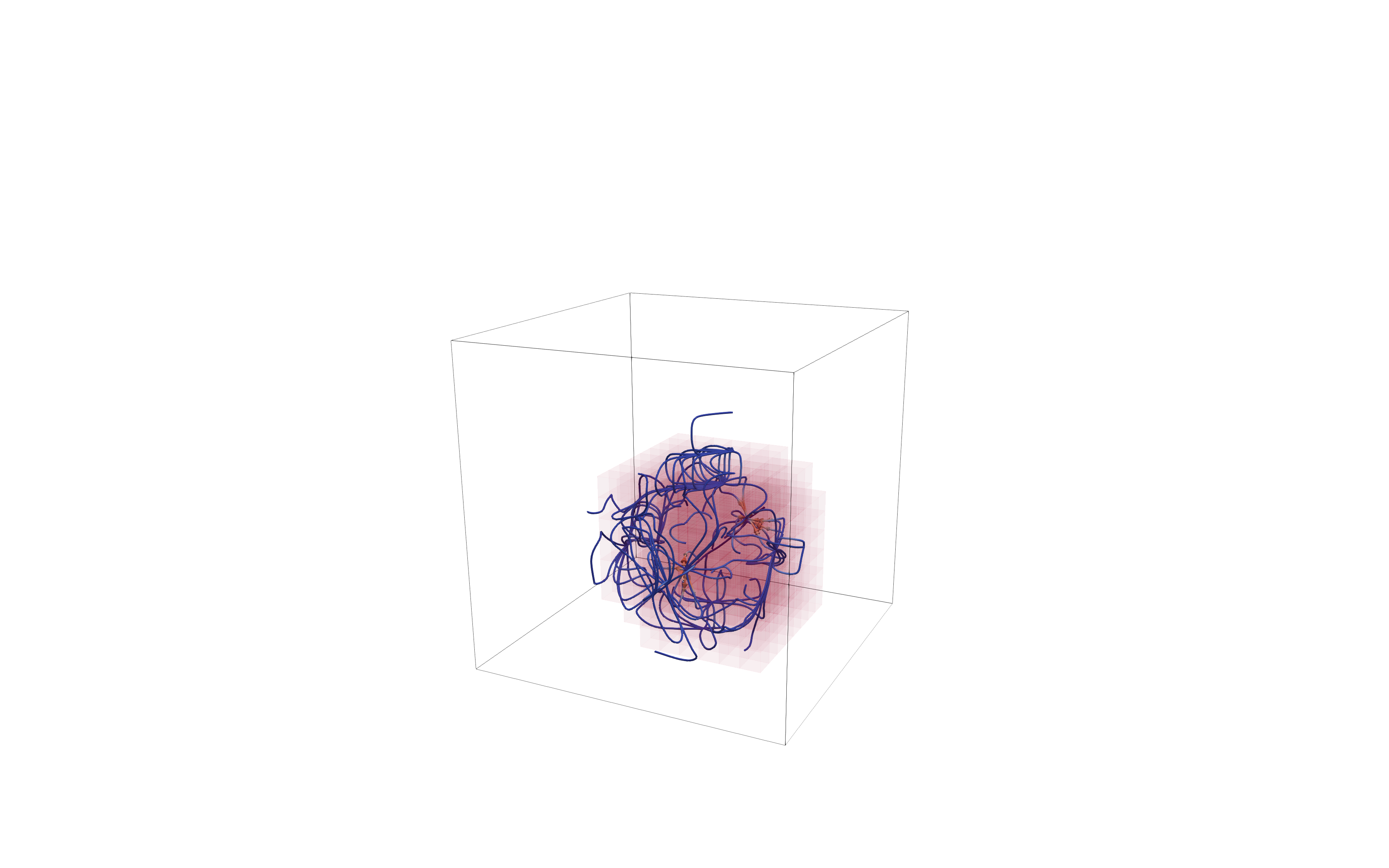}
		\caption{\Og Harmonic $1$-form refinement: $p = 3$}
	\end{subfigure}
	\caption{Inexact refinement patterns in three dimensions are shown for trivariate splines of degree $\pdeg{\ell}{k} = p$ and maximal smoothness for all $\ell$ and $k$.
	Each refinement pattern leads to $\gmesh{0,1}$ to $\gmesh{1,1}$ with different topologies.
	Pattern (a) introduces a new component in $\gmesh{1,1}$, pattern (b) transitions from a ball topology in $\gmesh{0,1}$ to a solid torus in $\gmesh{1,1}$, pattern (c) transitions from two disconnected balls in $\gmesh{0,1}$ to one in $\gmesh{1,1}$, pattern (d) transitions from a ball in $\gmesh{0,1}$ to a simply connected volume with a void in $\gmesh{1,1}$, pattern (e) transitions from a solid torus in $\gmesh{0,1}$ to a simply connected volume with a void in $\gmesh{1,1}$, and pattern (f) transitions from a solid torus in $\gmesh{0,1}$ to a ball in $\gmesh{1,1}$}
	\label{fig:3d_harmonic_configurations}
\end{figure}

\subsubsection{Exact Refinement Patterns}
Having characterized various harmonic forms and described minimal examples showing how changes in topology of the Greville meshes affect exactness of the hierarchical spline complex, we now turn our attention to certain configurations that lead to exact hierarchical spline complexes.\\

\noindent
\emph{\textbf{Two-dimensional examples}}\\
In two dimensions, Figure \ref{fig:2d_refinements_diag} shows a maximal regularity bi-degree $(6,6)$ hierarchical configuration where the refinement domain is built as the union of two 0-form B-splines; different subfigures correspond to different choices of the two B-splines.
Table \ref{tab:2d_refinements_diag} compares these different configurations based on the local exactness characterization of \cite{Evans19} and our local exactness characterization presented herein.
It can be seen that in some cases, the theory of \cite{Evans19} can capture refinements that violate at least one of our assumptions: this is because \cite{Evans19} permits refinement of $2$-form B-splines rather than just $0$-form B-splines (see Assumption \ref{assume:zero_union}).
That said, Assumption \ref{assume:shortest_path} of this work can be much less restrictive than Assumption 5.7 of \cite{Evans19} -- particularly for high-degree splines with high smoothness. 
In the example shown, our proposed assumptions admits five of the six hierarchical configurations that are exact, while \cite{Evans19} only allows for two of them.
This is largely because our result only requires identical topology between the coarse and fine Greville grids, while that of \cite{Evans19} is based on the more restrictive assumption requiring identical topologies for both the Bezi\'er meshes and Greville grids.
On the other hand, the assumptions proposed in this work can also be restrictive in other situations as they require refinement along paths of 0-forms (see Definition \ref{def:shortest_path}).

\renewcommand*{\xBoxDim}{3.3cm}%
\renewcommand*{\yBoxDim}{3.3cm}%
\renewcommand*{\thistextwidth}{0.24}
\begin{figure}[ht]
	\begin{subfigure}{\thistextwidth\textwidth}
	\centering
		\resizebox{\xBoxDim}{\yBoxDim}{%
		\begin{tikzpicture}
			% Background grid
                		\mymanualgrid{0,...,17}{0,...,17}{0.25}{gray}{white}
			% bottom left refinement
                		\mymanualgrid{1,1.5,2,2.5,3,3.5,4,4.5,5,5.5,6,6.5,7,7.5,8}{1,1.5,2,2.5,3,3.5,4,4.5,5,5.5,6,6.5,7,7.5,8}{1}{black}{cyan}
			% top right refinement
                		\mymanualgrid{9,9.5,10,10.5,11,11.5,12,12.5,13,13.5,14,14.5,15,15.5,16}{9,9.5,10,10.5,11,11.5,12,12.5,13,13.5,14,14.5,15,15.5,16}{1}{black}{cyan}
			\end{tikzpicture}
			}
			\caption{}
	\end{subfigure}
	\begin{subfigure}{\thistextwidth\textwidth}
	\centering
		\resizebox{\xBoxDim}{\yBoxDim}{%
		\begin{tikzpicture}
			% Background grid
                		\mymanualgrid{0,...,17}{0,...,17}{0.25}{gray}{white}
			% bottom left refinement
                		\mymanualgrid{1,1.5,2,2.5,3,3.5,4,4.5,5,5.5,6,6.5,7,7.5,8}{1,1.5,2,2.5,3,3.5,4,4.5,5,5.5,6,6.5,7,7.5,8}{1}{black}{cyan}
			% top right refinement
                		\mymanualgrid{8,8.5,9,9.5,10,10.5,11,11.5,12,12.5,13,13.5,14,14.5,15}{8,8.5,9,9.5,10,10.5,11,11.5,12,12.5,13,13.5,14,14.5,15}{1}{black}{cyan}
			\end{tikzpicture}
			}
			\caption{}
	\end{subfigure}
	\begin{subfigure}{\thistextwidth\textwidth}
	\centering
		\resizebox{\xBoxDim}{\yBoxDim}{%
		\begin{tikzpicture}
			% Background grid
                		\mymanualgrid{0,...,17}{0,...,17}{0.25}{gray}{white}
			% bottom left refinement
                		\mymanualgrid{2,2.5,3,3.5,4,4.5,5,5.5,6,6.5,7,7.5,8,8.5,9}{2,2.5,3,3.5,4,4.5,5,5.5,6,6.5,7,7.5,8,8.5,9}{1}{black}{cyan}
			% top right refinement
                		\mymanualgrid{8,8.5,9,9.5,10,10.5,11,11.5,12,12.5,13,13.5,14,14.5,15}{8,8.5,9,9.5,10,10.5,11,11.5,12,12.5,13,13.5,14,14.5,15}{1}{black}{cyan}
			\end{tikzpicture}
			}
			\caption{}
	\end{subfigure}
	\begin{subfigure}{\thistextwidth\textwidth}
	\centering
		\resizebox{\xBoxDim}{\yBoxDim}{%
		\begin{tikzpicture}
			% Background grid
                		\mymanualgrid{0,...,17}{0,...,17}{0.25}{gray}{white}
			% bottom left refinement
                		\mymanualgrid{2,2.5,3,3.5,4,4.5,5,5.5,6,6.5,7,7.5,8,8.5,9}{2,2.5,3,3.5,4,4.5,5,5.5,6,6.5,7,7.5,8,8.5,9}{1}{black}{cyan}
			% top right refinement
                		\mymanualgrid{7,7.5,8,8.5,9,9.5,10,10.5,11,11.5,12,12.5,13,13.5,14}{7,7.5,8,8.5,9,9.5,10,10.5,11,11.5,12,12.5,13,13.5,14}{1}{black}{cyan}
			\end{tikzpicture}
			}
			\caption{}
	\end{subfigure}
	\begin{subfigure}{\thistextwidth\textwidth}
	\centering
		\resizebox{\xBoxDim}{\yBoxDim}{%
		\begin{tikzpicture}
			% Background grid
                		\mymanualgrid{0,...,17}{0,...,17}{0.25}{gray}{white}
			% bottom left refinement
                		\mymanualgrid{3,3.5,4,4.5,5,5.5,6,6.5,7,7.5,8,8.5,9,9.5,10}{3,3.5,4,4.5,5,5.5,6,6.5,7,7.5,8,8.5,9,9.5,10}{1}{black}{cyan}
			% top right refinement
                		\mymanualgrid{7,7.5,8,8.5,9,9.5,10,10.5,11,11.5,12,12.5,13,13.5,14}{7,7.5,8,8.5,9,9.5,10,10.5,11,11.5,12,12.5,13,13.5,14}{1}{black}{cyan}
			\end{tikzpicture}
			}
			\caption{}
	\end{subfigure}
	\begin{subfigure}{\thistextwidth\textwidth}
	\centering
		\resizebox{\xBoxDim}{\yBoxDim}{%
		\begin{tikzpicture}
			% Background grid
                		\mymanualgrid{0,...,17}{0,...,17}{0.25}{gray}{white}
			% bottom left refinement
                		\mymanualgrid{3,3.5,4,4.5,5,5.5,6,6.5,7,7.5,8,8.5,9,9.5,10}{3,3.5,4,4.5,5,5.5,6,6.5,7,7.5,8,8.5,9,9.5,10}{1}{black}{cyan}
			% top right refinement
                		\mymanualgrid{6,6.5,7,7.5,8,8.5,9,9.5,10,10.5,11,11.5,12,12.5,13}{6,6.5,7,7.5,8,8.5,9,9.5,10,10.5,11,11.5,12,12.5,13}{1}{black}{cyan}
			\end{tikzpicture}
			}
			\caption{}
	\end{subfigure}
	\begin{subfigure}{\thistextwidth\textwidth}
	\centering
		\resizebox{\xBoxDim}{\yBoxDim}{%
		\begin{tikzpicture}
			% Background grid
                		\mymanualgrid{0,...,17}{0,...,17}{0.25}{gray}{white}
			% bottom left refinement
                		\mymanualgrid{4,4.5,5,5.5,6,6.5,7,7.5,8,8.5,9,9.5,10,10.5,11}{4,4.5,5,5.5,6,6.5,7,7.5,8,8.5,9,9.5,10,10.5,11}{1}{black}{cyan}
			% top right refinement
                		\mymanualgrid{6,6.5,7,7.5,8,8.5,9,9.5,10,10.5,11,11.5,12,12.5,13}{6,6.5,7,7.5,8,8.5,9,9.5,10,10.5,11,11.5,12,12.5,13}{1}{black}{cyan}
			\end{tikzpicture}
			}
			\caption{}
	\end{subfigure}
	\begin{subfigure}{\thistextwidth\textwidth}
	\centering
		\resizebox{\xBoxDim}{\yBoxDim}{%
		\begin{tikzpicture}
			% Background grid
                		\mymanualgrid{0,...,17}{0,...,17}{0.25}{gray}{white}
			% bottom left refinement
                		\mymanualgrid{4,4.5,5,5.5,6,6.5,7,7.5,8,8.5,9,9.5,10,10.5,11}{4,4.5,5,5.5,6,6.5,7,7.5,8,8.5,9,9.5,10,10.5,11}{1}{black}{cyan}
			% top right refinement
                		\mymanualgrid{5,5.5,6,6.5,7,7.5,8,8.5,9,9.5,10,10.5,11,11.5,12}{5,5.5,6,6.5,7,7.5,8,8.5,9,9.5,10,10.5,11,11.5,12}{1}{black}{cyan}
			\end{tikzpicture}
			}
			\caption{}
	\end{subfigure}
	\caption{
		The above figures show hierarchical meshes used for building maximally smooth hierarchical B-splines with $\pdeg{\ell}{k} = 6$ for all $\ell$ and $k$.
		The domain $\Omega_{\ell+1}$ is defined to be the union of the supports of two 0-form B-splines at level $\ell$.
		Different subfigures correspond to different choices of the two B-splines and each subfigure progressively leads to a greater overlap of their respective supports.
		Refinement patterns (e), (f), and (g) are not exact, while all others are.
		The assumptions given in this document allow for each of the exact refinement patterns except figure (h), while those of \cite{Evans19} are more restrictive in this respect and cannot capture three of them.
		}\label{fig:2d_refinements_diag}
\end{figure}
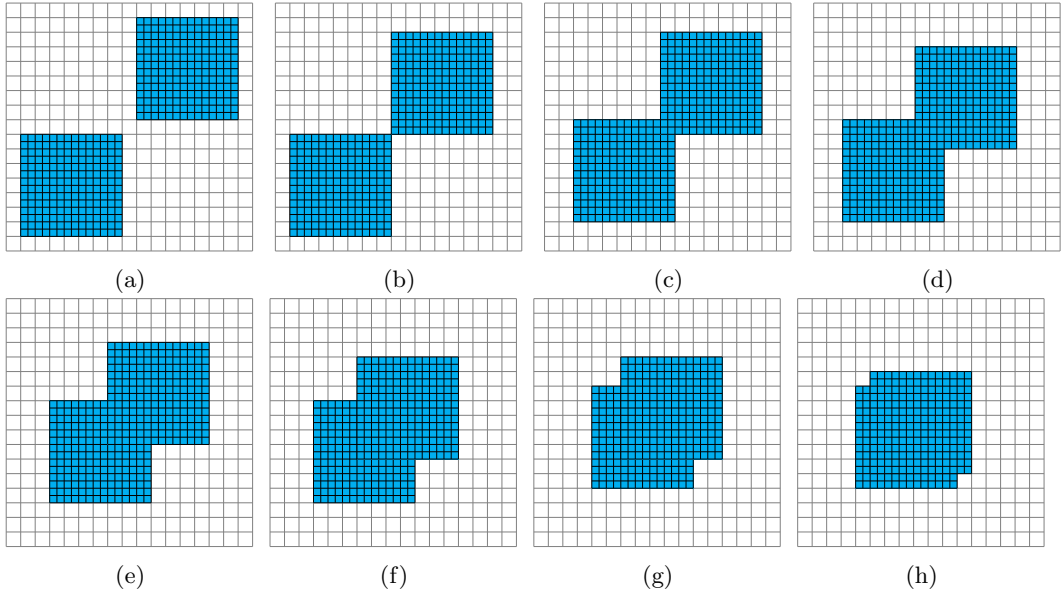

\begin{table}[ht]
	\centering
	\begin{tabular}{|>{\centering\arraybackslash}m{2.8cm}|>{\centering\arraybackslash}m{1.5cm}|>{\centering\arraybackslash}m{1.8cm}|>{\centering\arraybackslash}m{1.8cm}|>{\centering\arraybackslash}m{1.8cm}|>{\centering\arraybackslash}m{1.8cm}|>{\centering\arraybackslash}m{1.8cm}|}
		\hline\hline
		\textsc{Refinement Pattern} & Exact &  Assumption 5.7 \cite{Evans19} & Assumption 3  \\\hline
		\cellcolor[gray]{0.9}Figure \ref{fig:2d_refinements_diag}(a) & Yes  & Yes & Yes \\\hline
		\cellcolor[gray]{0.9}Figure \ref{fig:2d_refinements_diag}(b) & Yes  & No &  Yes \\\hline
		\cellcolor[gray]{0.9}Figure \ref{fig:2d_refinements_diag}(c) & Yes  & No & Yes \\\hline
		\cellcolor[gray]{0.9}Figure \ref{fig:2d_refinements_diag}(d) & Yes  & No &  Yes \\\hline
		\cellcolor[gray]{0.9}Figure \ref{fig:2d_refinements_diag}(e) & No  & No &No \\\hline
		\cellcolor[gray]{0.9}Figure \ref{fig:2d_refinements_diag}(f) & No  & No & No \\\hline
		\cellcolor[gray]{0.9}Figure \ref{fig:2d_refinements_diag}(g) & No  & No & No \\\hline
		\cellcolor[gray]{0.9}Figure \ref{fig:2d_refinements_diag}(h) & Yes  & Yes & No \\\hline
		\hline
	\end{tabular}
\caption{
	For the refinement patterns shown in Figure \ref{fig:2d_refinements_diag}, the table above presents whether the corresponding hierarchical complexes are exact, and whether the local exactness constraints of this work and those of \cite{Evans19} are met.
	Since refinements are performed by combining the supports of two 0-forms, Assumption 2 of this work and Assumption 5.6 of \cite{Evans19} are always satisfied and so we do not include them in the table.
	}\label{tab:2d_refinements_diag}
\end{table}

\vspace{\baselineskip}
\noindent
\emph{\textbf{Three-dimensional examples}}\\
A variety of exact refinement patters in three dimensions are depicted in Figures \ref{fig:exact_configs_supported} and \ref{fig:exact_configs_not_supported}.
Figure \ref{fig:exact_configs_supported} exhibits exact refinement patterns that are permitted by this work, while Figure \ref{fig:exact_configs_not_supported} shows exact refinements that are not permissible given our assumptions.

Since the local exactness condition of \cite{Evans19} does not apply for $\ndim > 2$, we do not compare exact refinement patterns of this work to those of others.
That said, in \cite[Remark 5.9]{Evans19} it is conjectured that the hierarchical B-spline complex would be exact if the intersection between the support of any $j$-form of $\Omega_{\ell}$ and the complement of $\Omega_{\ell+1}$ is homologically trivial (i.e. the zeroth homology group has rank one and all others are of rank zero).
This is a direct extension of \cite[Assumption 5.7]{Evans19} to higher-dimensional spaces.
The results of this paper do not contradict this claim.
Nonetheless, our results show that this conjecture may be more restrictive than is necessary.
Indeed, Assumption \ref{assume:shortest_path} does not require the topology of the coarse and fine Greville grids to match that of the underlying Bezi\'er mesh, as the conjecture of \cite{Evans19} would.
The exact refinement patterns of Figure \ref{fig:exact_configs_supported} permissible under this work illustrate this point: refinements presented in subfigures (a)-(c) would be supported by the extension of \cite{Evans19} to three dimensions, while the refinement patterns in subfigures (d)-(k) would be inadmissible.
We believe that in higher dimensions permissible refinement patterns under the conjuecture of \cite{Evans19} would be similarly much too restrictive.

Finally, the results of this numerical study appear to indicate that our proposed refinement strategy on 0-forms may be more restrictive than a refinement pattern that simply relies on refinement of $\ndim$-forms, as does \cite{Evans19}.
We postulate that Assumptions \ref{assume:zero_union} and \ref{assume:shortest_path} could be modified to only operate on $\ndim$ forms, rather than 0-forms, and still produce a sufficient local exactness result.
Arriving at such a proof, however, would require more advanced techniques than available herein and, in particular, may preclude the use of the Mayer-Vietoris sequence.
We expect that such a result, however, would both generalize and unify the theory of this work and that of \cite{Evans19}.

\begin{figure}
		\centering
	\begin{subfigure}{0.31\textwidth}
		\includegraphics[width=1\textwidth,trim=18cm 2.5cm 17cm 11cm,clip]{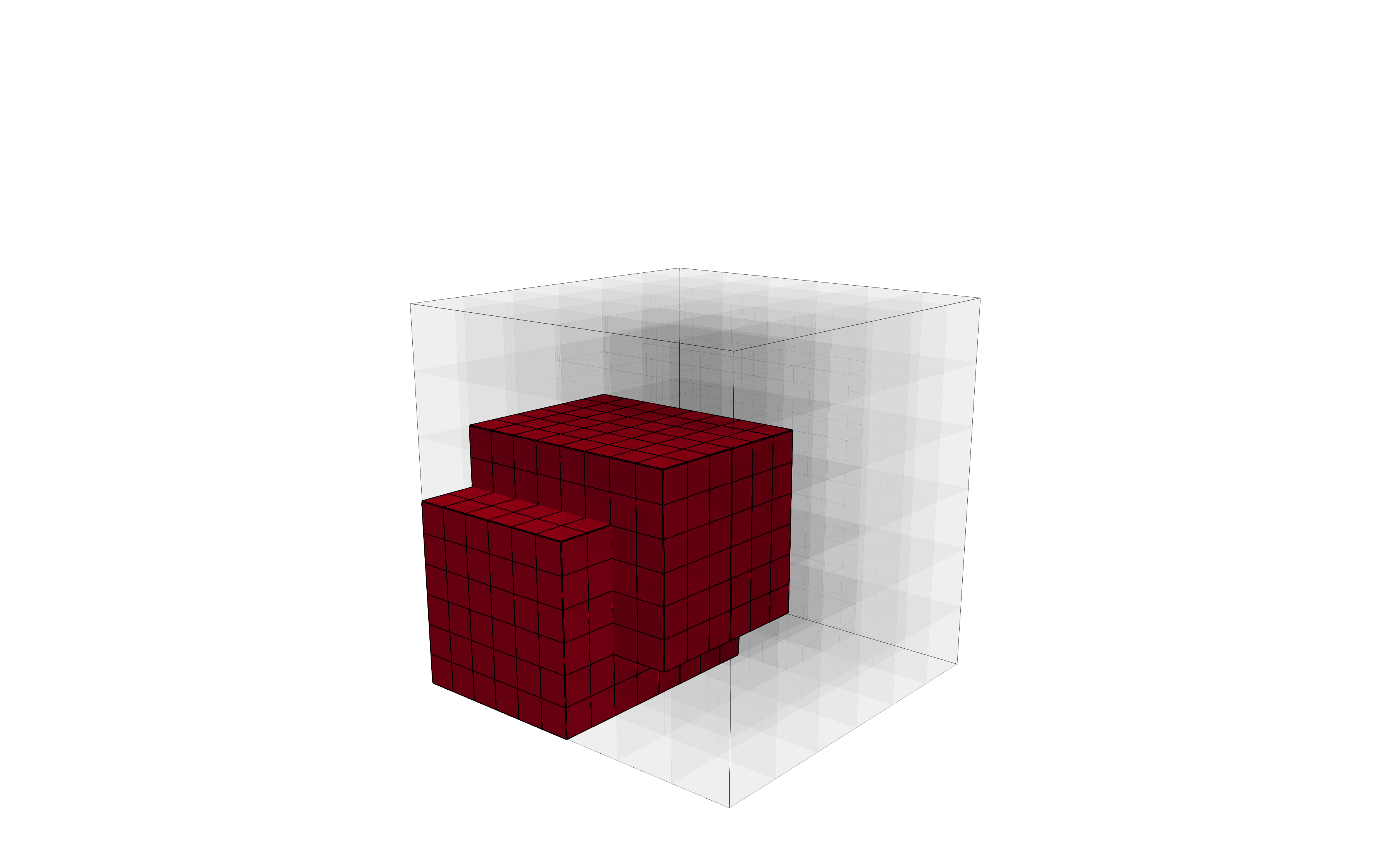}
		\caption{$p=2$}
	\end{subfigure}
	\;
	\begin{subfigure}{0.31\textwidth}
		\includegraphics[width=1\textwidth,trim=18cm 2.5cm 17cm 11cm,clip]{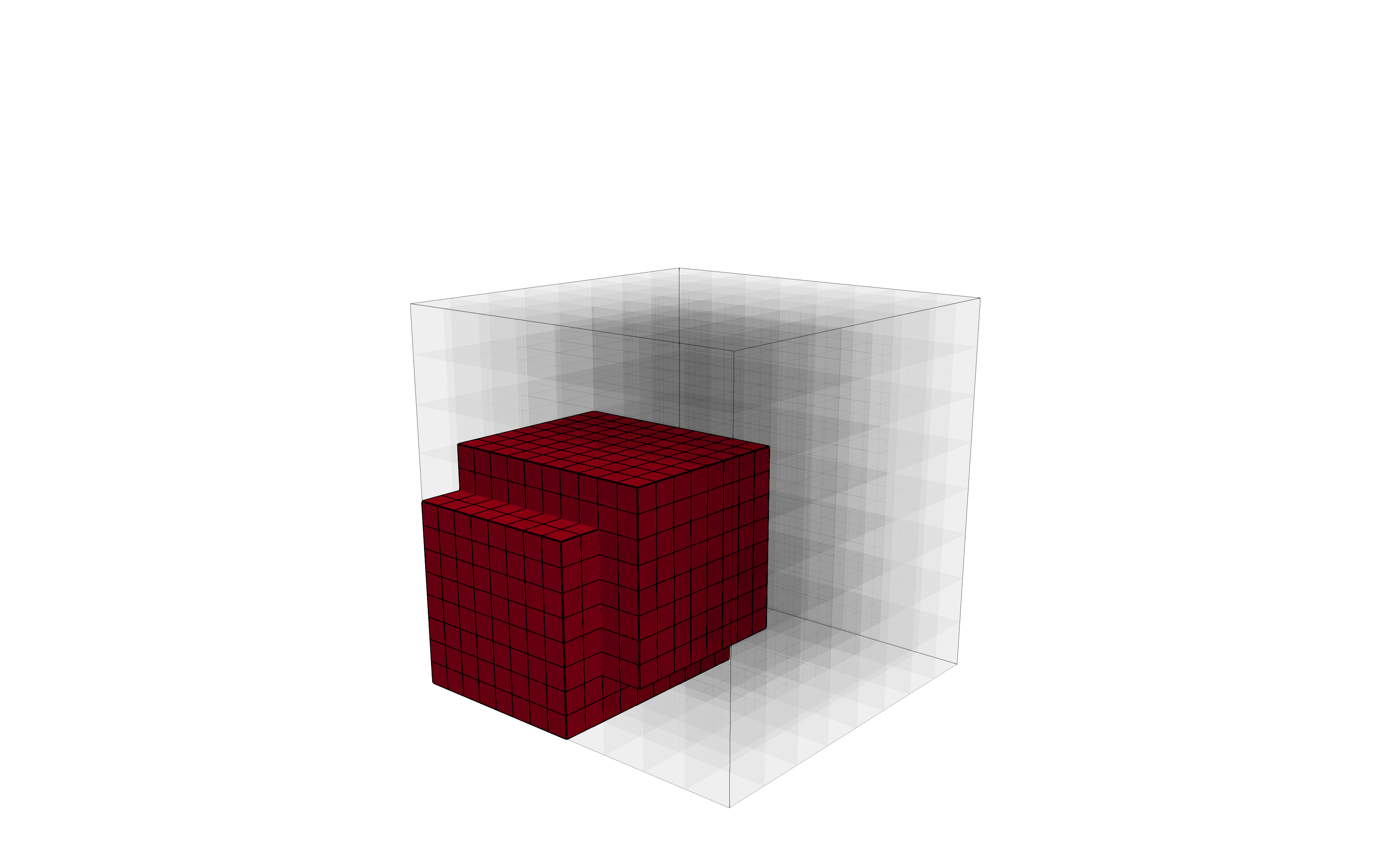}
		\caption{$p=3$}
	\end{subfigure}
	\;
	\begin{subfigure}{0.31\textwidth}
		\includegraphics[width=1\textwidth,trim=18cm 2.5cm 17cm 11cm,clip]{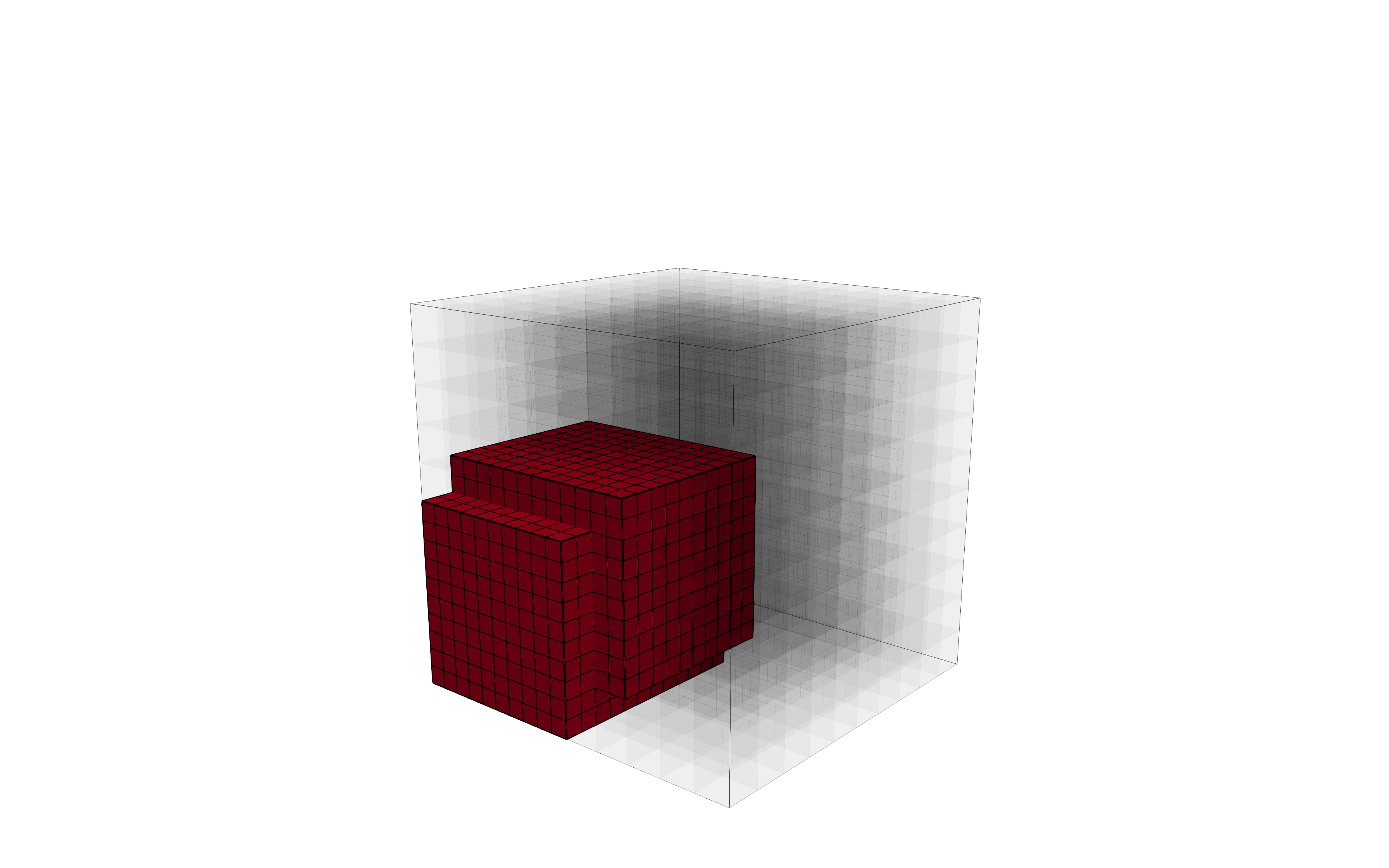}
		\caption{$p=4$}
	\end{subfigure}
	\\
	\begin{subfigure}{0.31\textwidth}
		\includegraphics[width=1\textwidth,trim=18cm 2.5cm 17cm 11cm,clip]{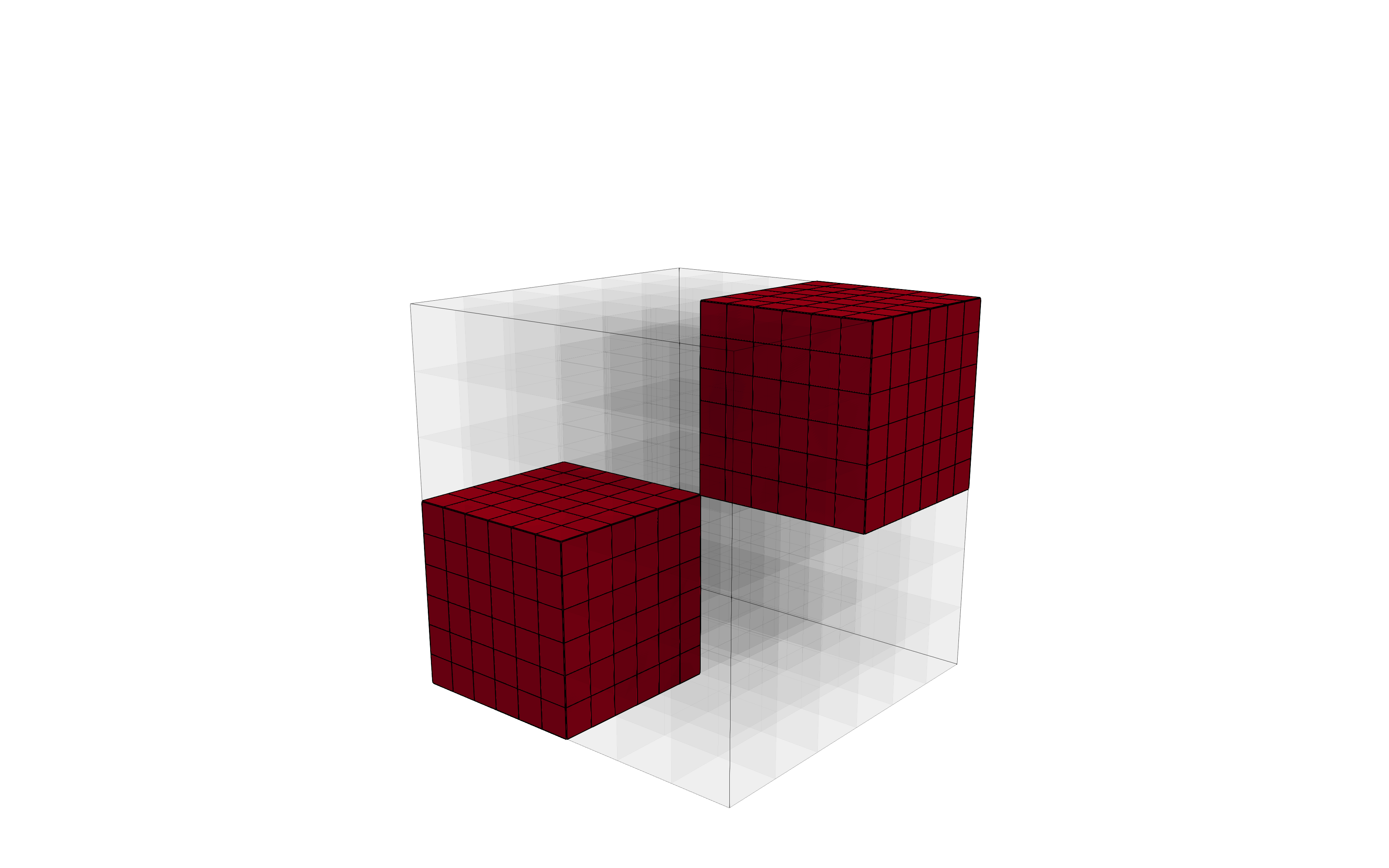}
		\caption{$p=2$}
	\end{subfigure}
	\;
	\begin{subfigure}{0.31\textwidth}
		\includegraphics[width=1\textwidth,trim=18cm 2.5cm 17cm 11cm,clip]{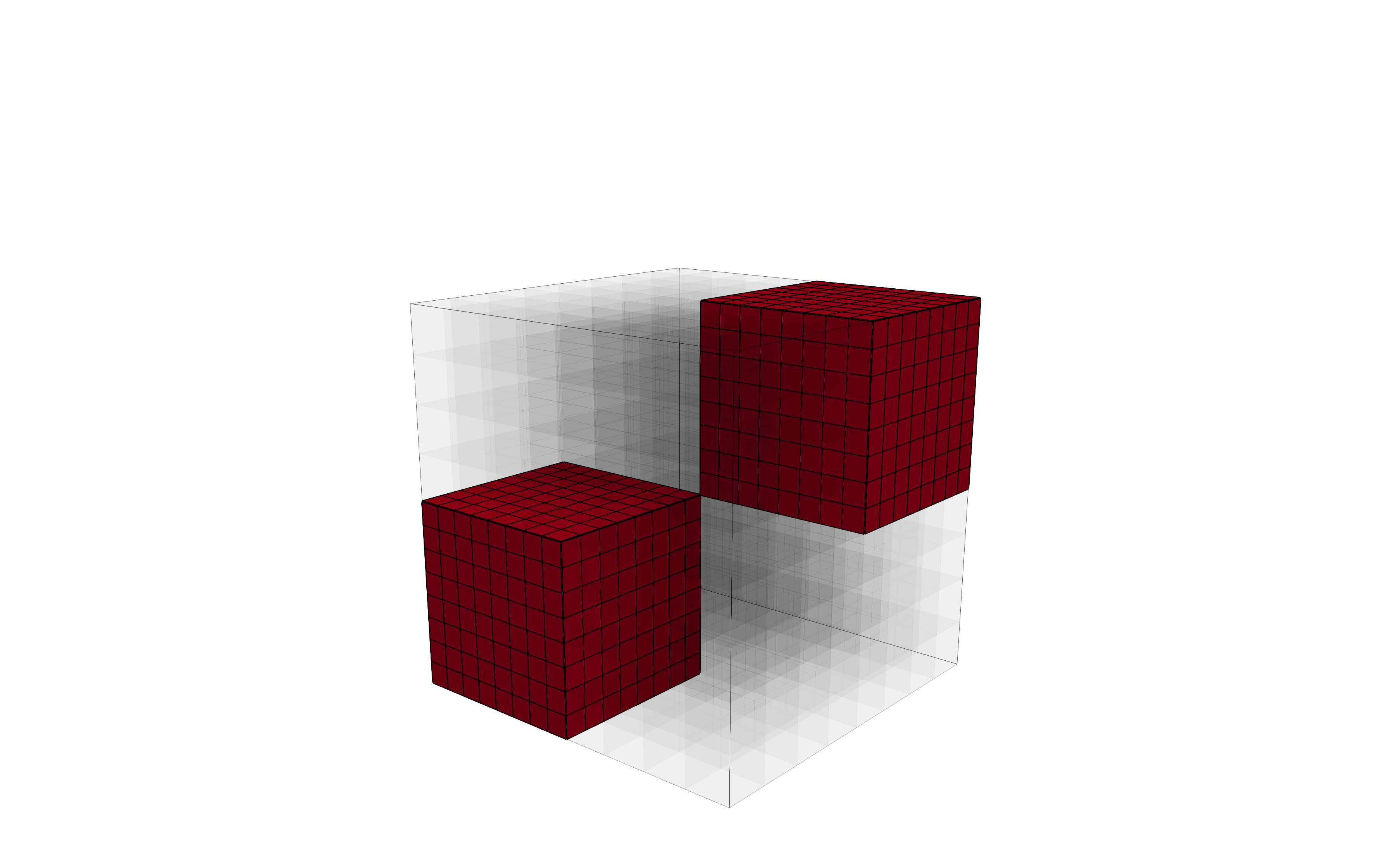}
		\caption{$p=3$}
	\end{subfigure}
	\;
	\begin{subfigure}{0.31\textwidth}
		\includegraphics[width=1\textwidth,trim=18cm 2.5cm 17cm 11cm,clip]{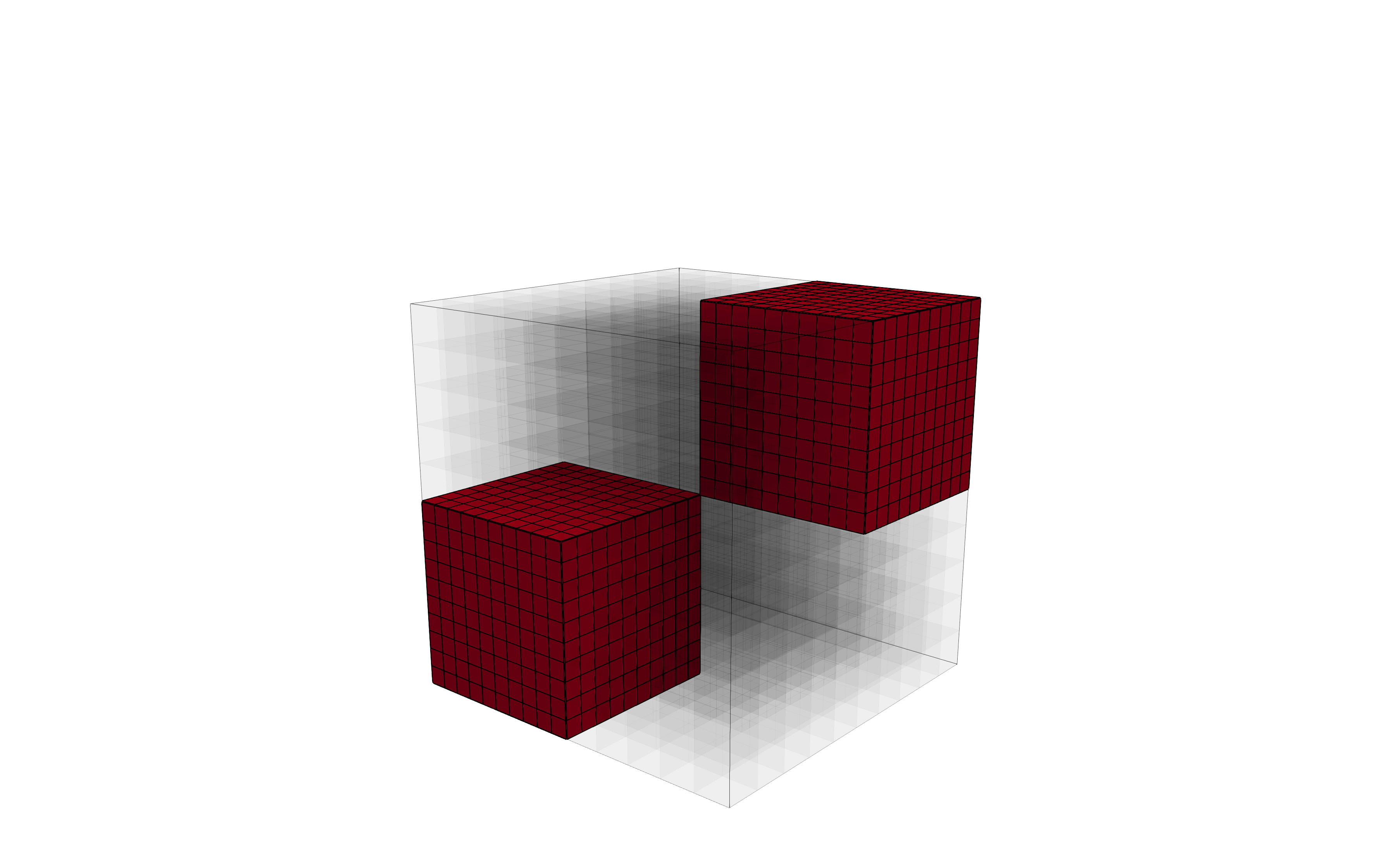}
		\caption{$p=4$}
	\end{subfigure}
	\\
	\begin{subfigure}{0.31\textwidth}
		\includegraphics[width=1\textwidth,trim=18cm 2.5cm 17cm 11cm,clip]{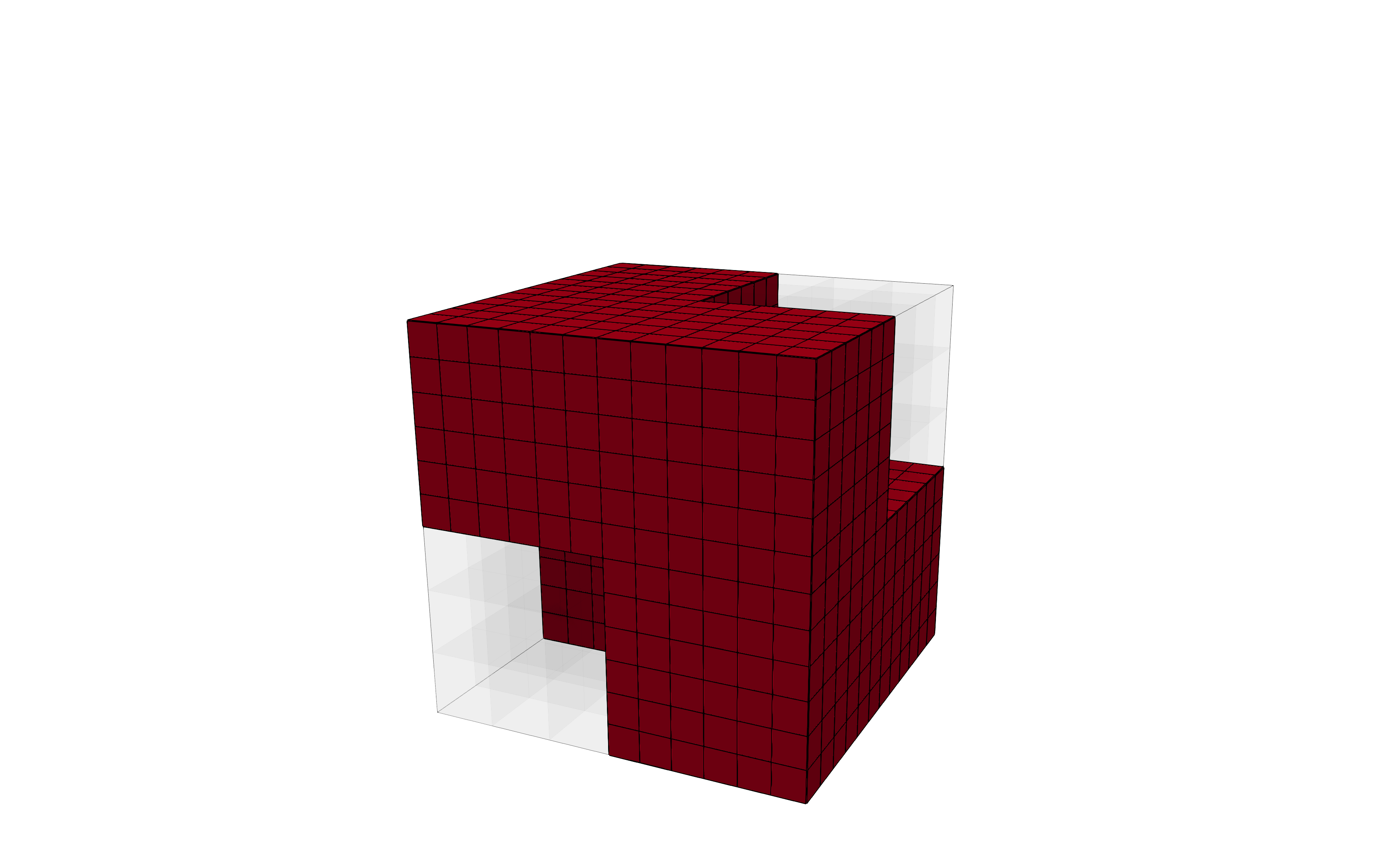}
		\caption{$p=2$}
	\end{subfigure}
	\;
	\begin{subfigure}{0.31\textwidth}
		\includegraphics[width=1\textwidth,trim=18cm 2.5cm 17cm 11cm,clip]{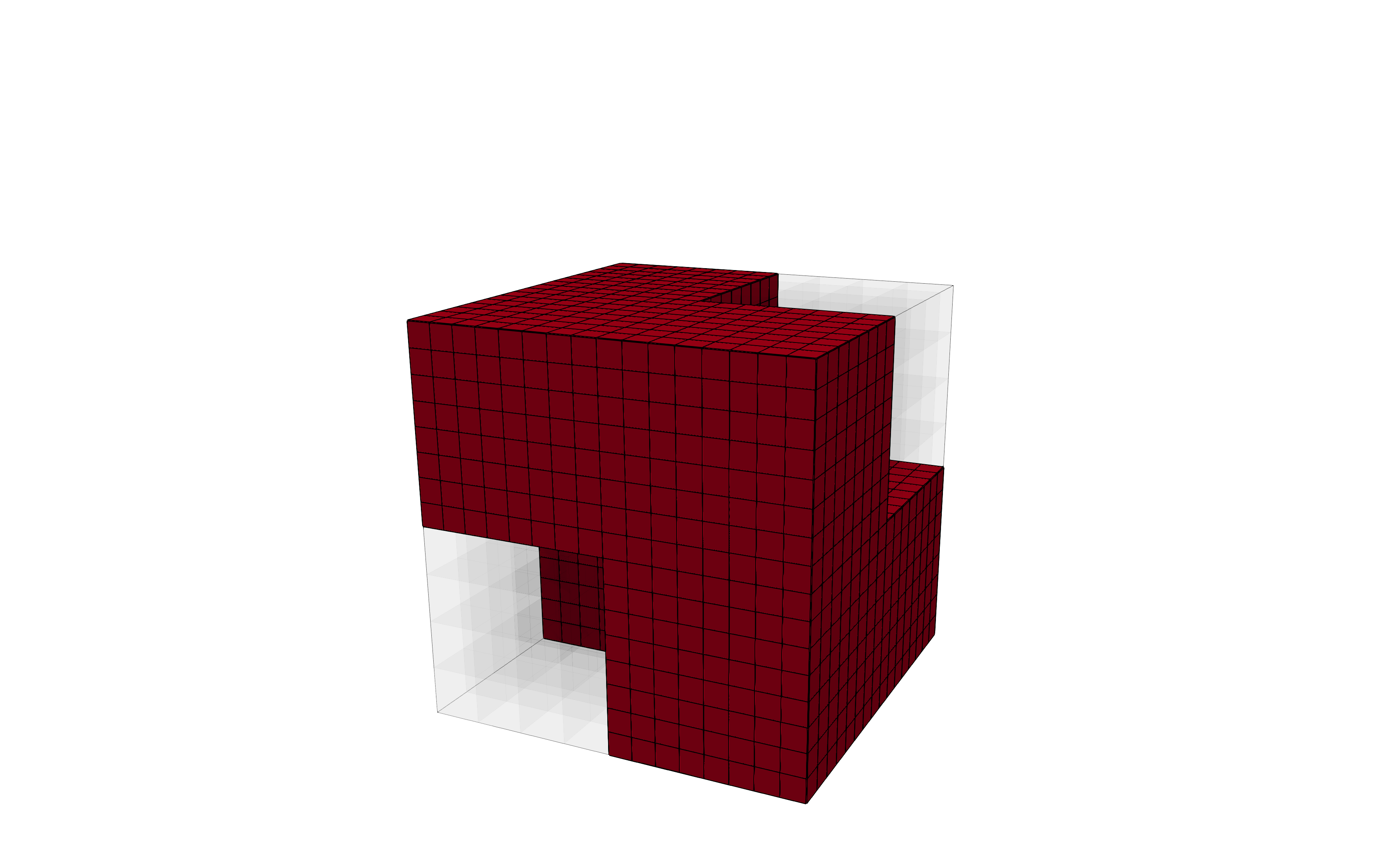}
		\caption{$p=3$}
	\end{subfigure}
	\;
	\begin{subfigure}{0.31\textwidth}
		\includegraphics[width=1\textwidth,trim=18cm 2.5cm 17cm 11cm,clip]{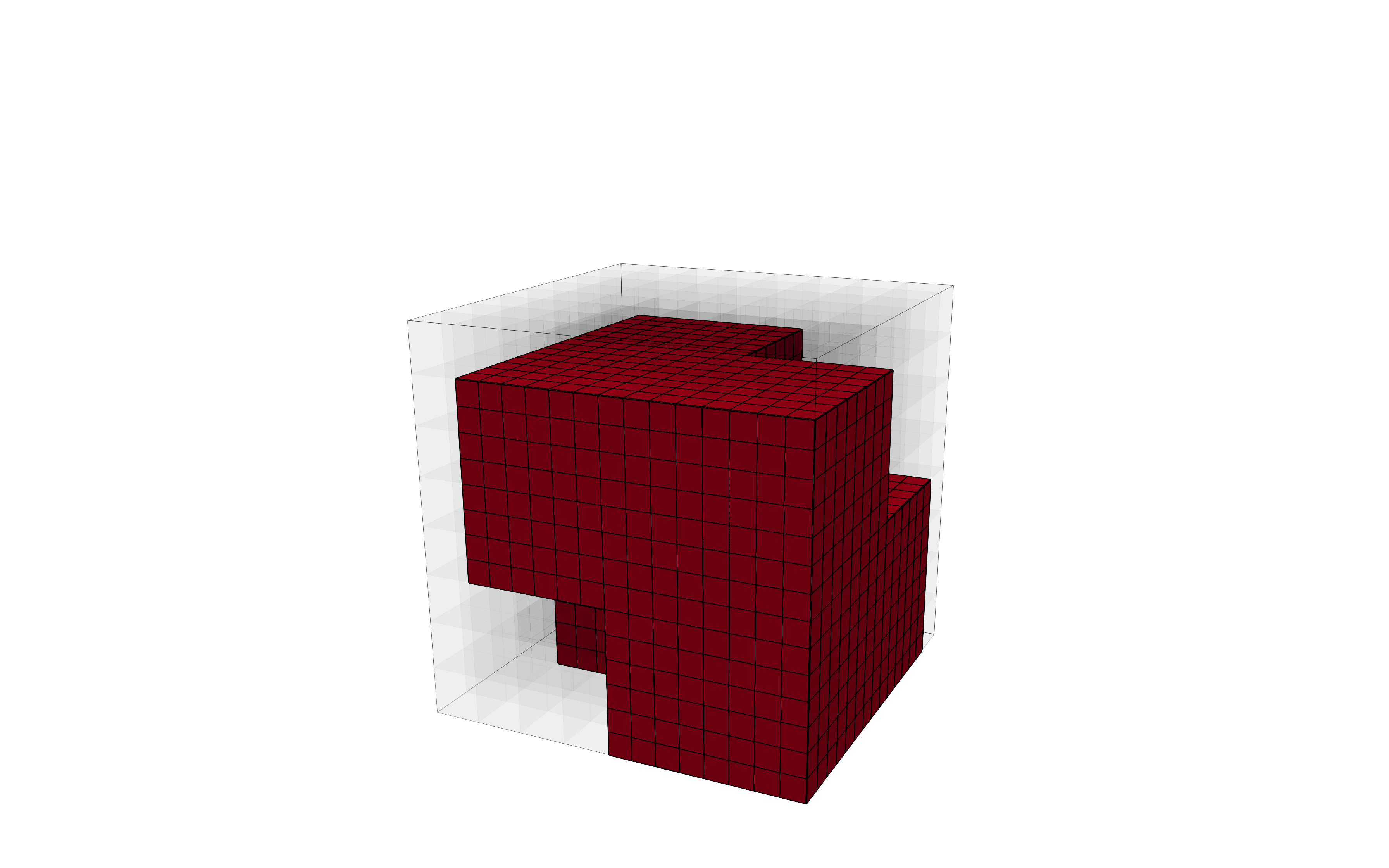}
		\caption{$p=3$}
	\end{subfigure}
	\\
	\begin{subfigure}{0.31\textwidth}
		\includegraphics[width=1\textwidth,trim=20.5cm 5.5cm 21cm 12.8cm,clip]{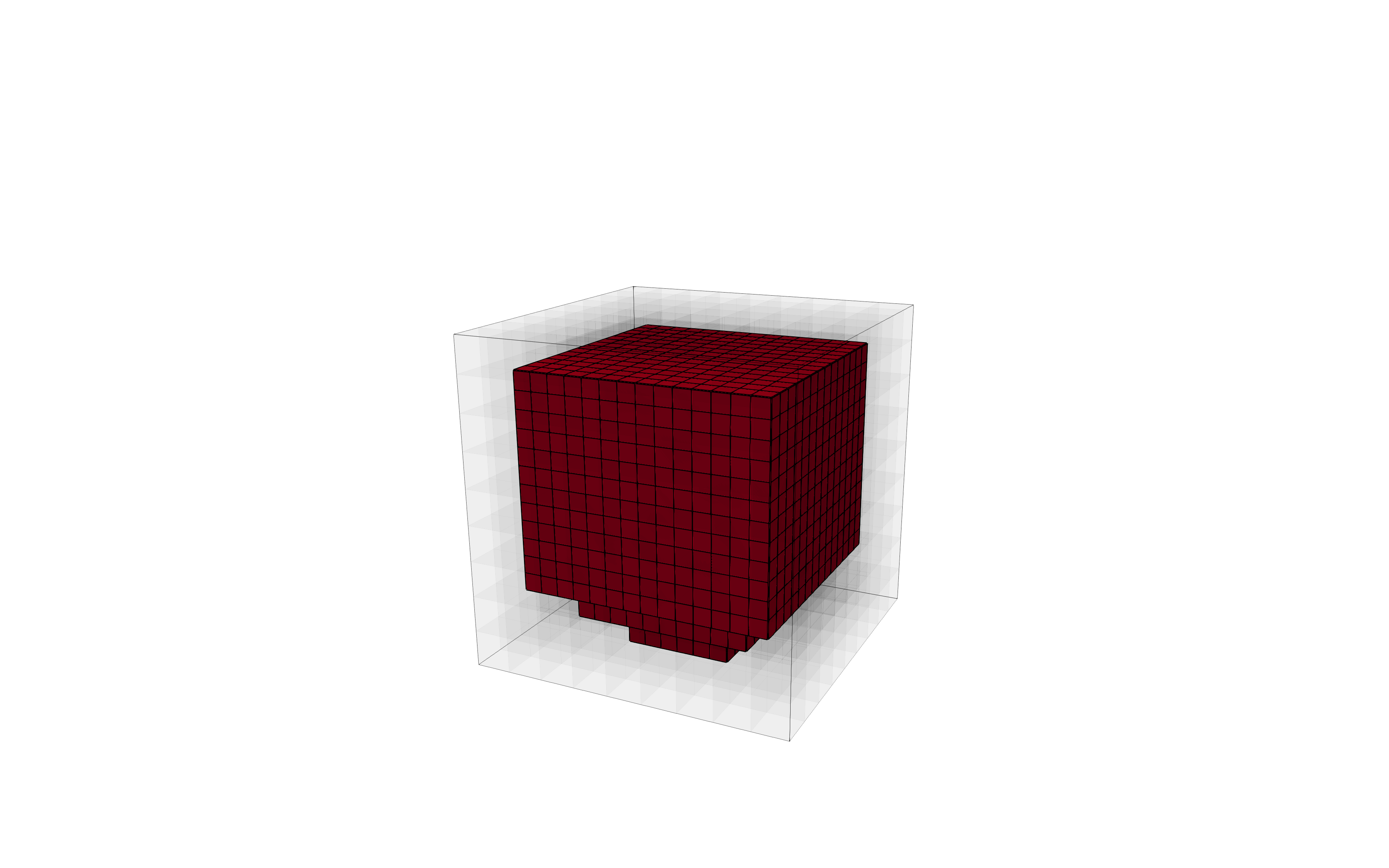}
		\caption{$p=2$ with cavity}
	\end{subfigure}
	\;
	\begin{subfigure}{0.31\textwidth}
		\includegraphics[width=1\textwidth,trim=20.5cm 5.5cm 21cm 12.8cm,clip]{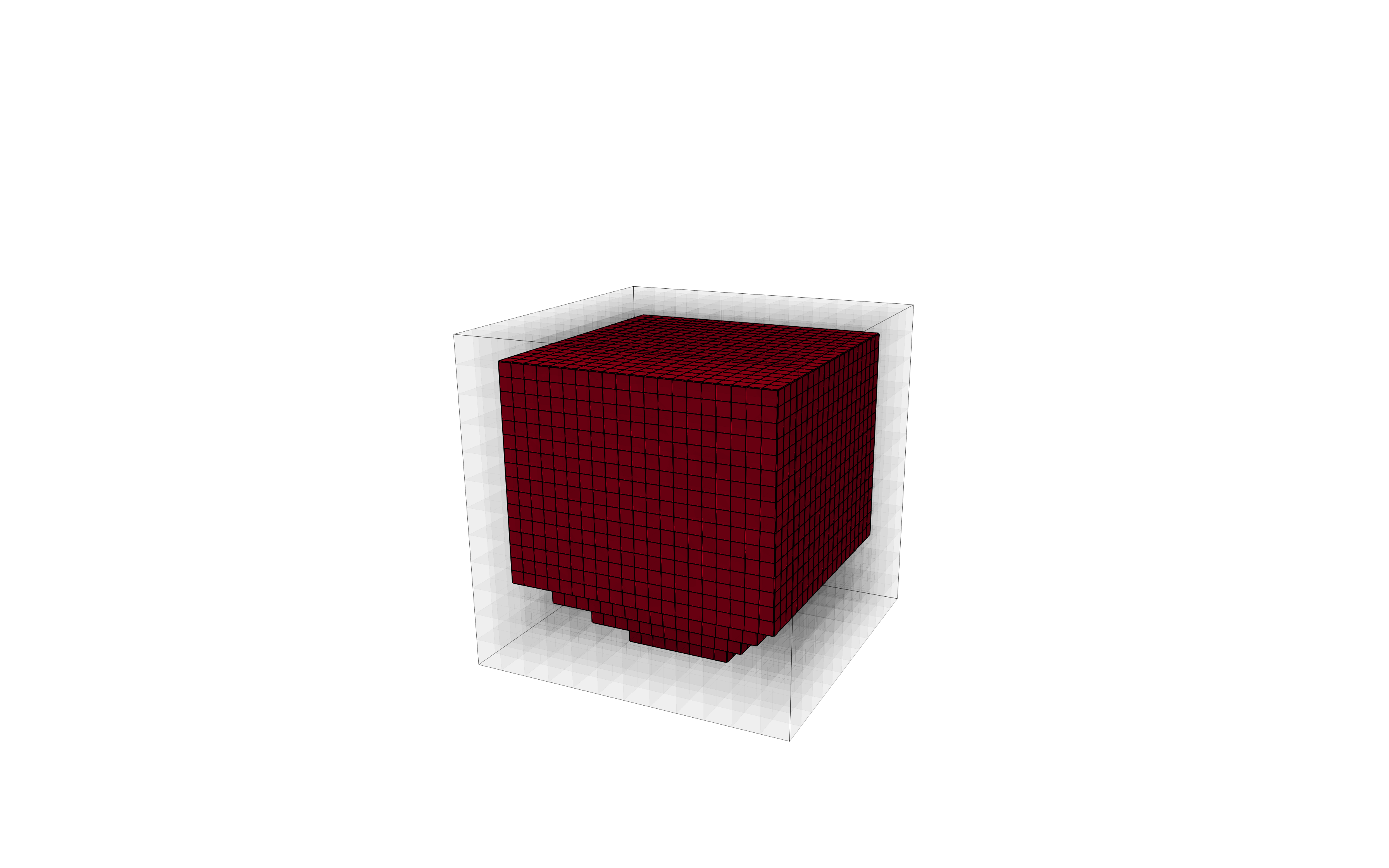}
		\caption{$p=3$ with cavity}
	\end{subfigure}
	\caption{Examples of refinement patterns for splines of maximal smoothness allowed by the proposed local exactness condition are shown above.
	The Greville grids $\gmesh{0,1}$ to $\gmesh{1,1}$ for the refinement patterns in each row are, respectively, a single contractible domain, two disconnected contractible domains, a single component with a non-trivial loop, and a single component with a void.}
	\label{fig:exact_configs_supported}
\end{figure}
\begin{figure}
	\centering
	\begin{subfigure}{0.31\textwidth}
		\includegraphics[width=1\textwidth,trim=18cm 2.5cm 17cm 11cm,clip]{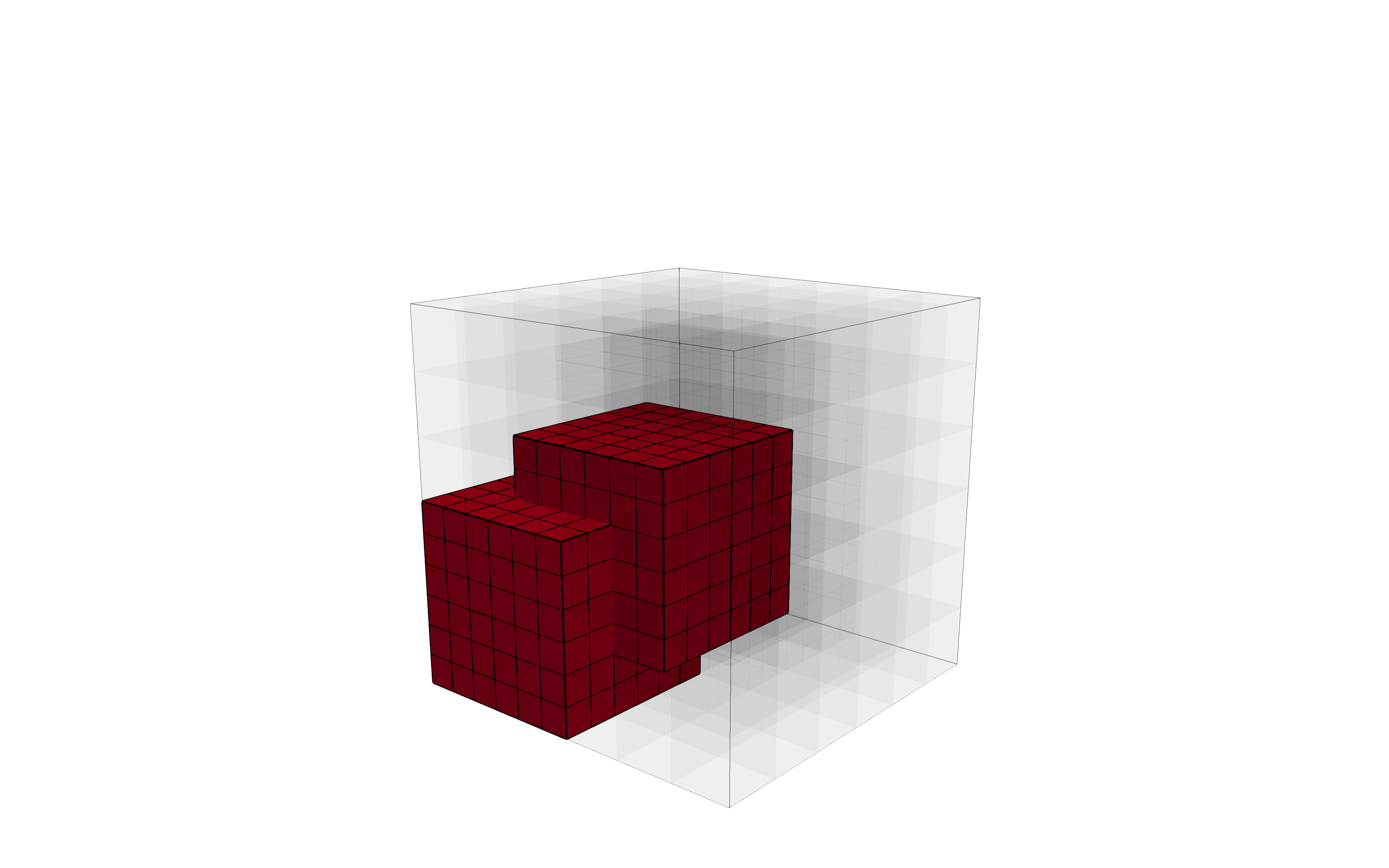}
		\caption{$p=2$}
	\end{subfigure}
	\;
	\begin{subfigure}{0.31\textwidth}
		\includegraphics[width=1\textwidth,trim=18cm 2.5cm 17cm 11cm,clip]{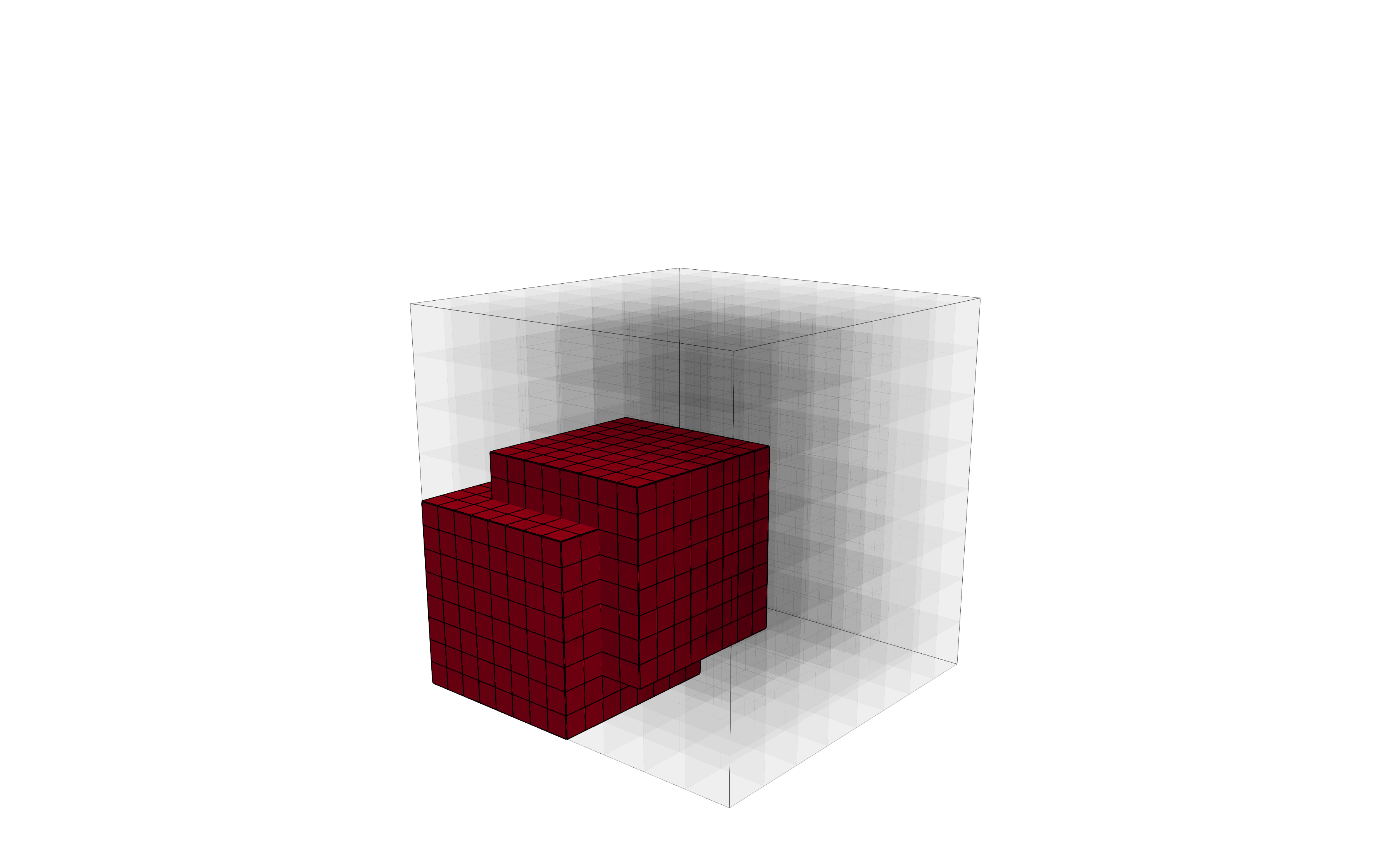}
		\caption{$p=3$}
	\end{subfigure}
	\;
	\begin{subfigure}{0.31\textwidth}
		\includegraphics[width=1\textwidth,trim=18cm 2.5cm 17cm 11cm,clip]{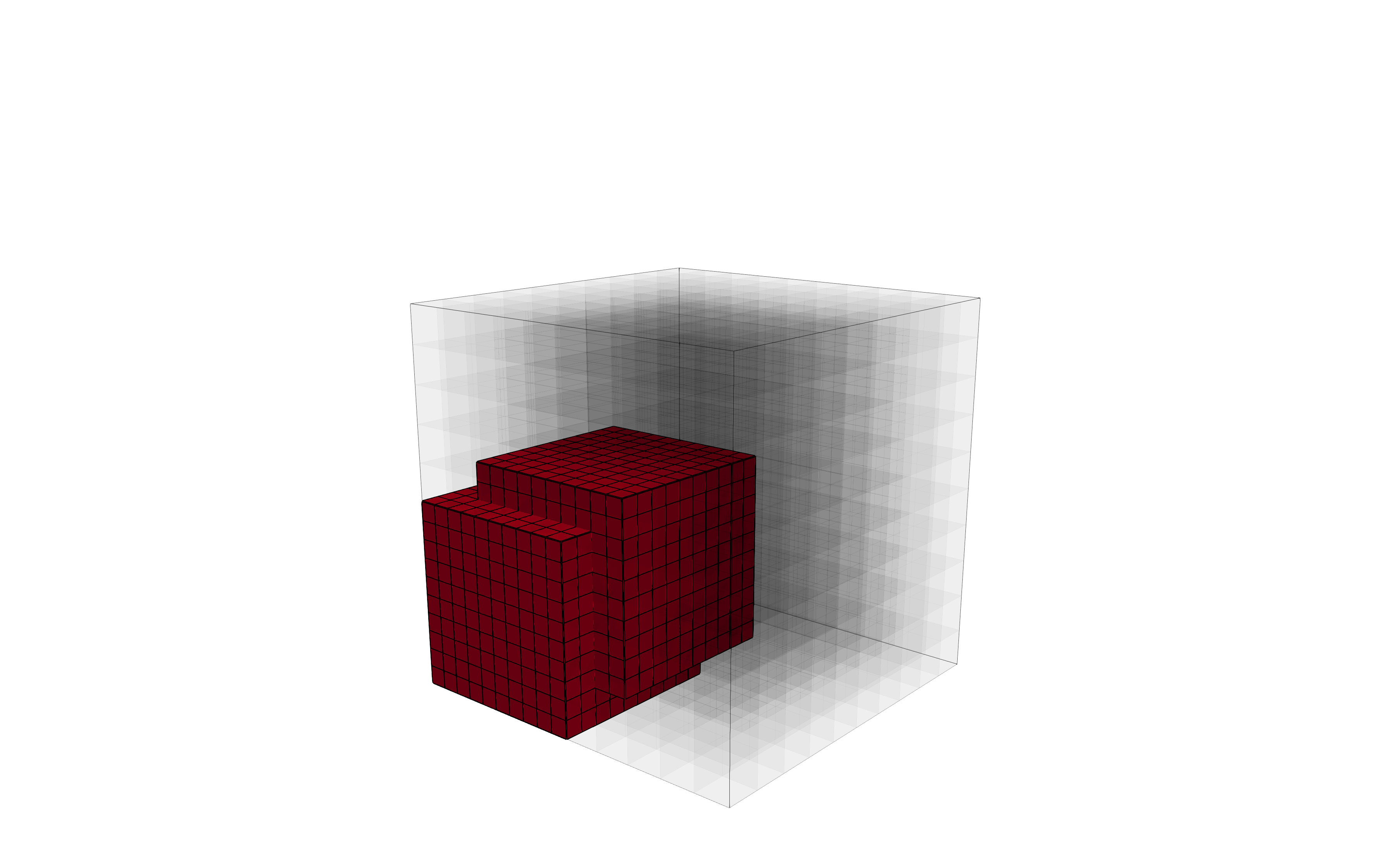}
		\caption{$p=4$}
	\end{subfigure}
	\\
	\begin{subfigure}{0.31\textwidth}
		\includegraphics[width=1\textwidth,trim=18cm 2.5cm 17cm 11cm,clip]{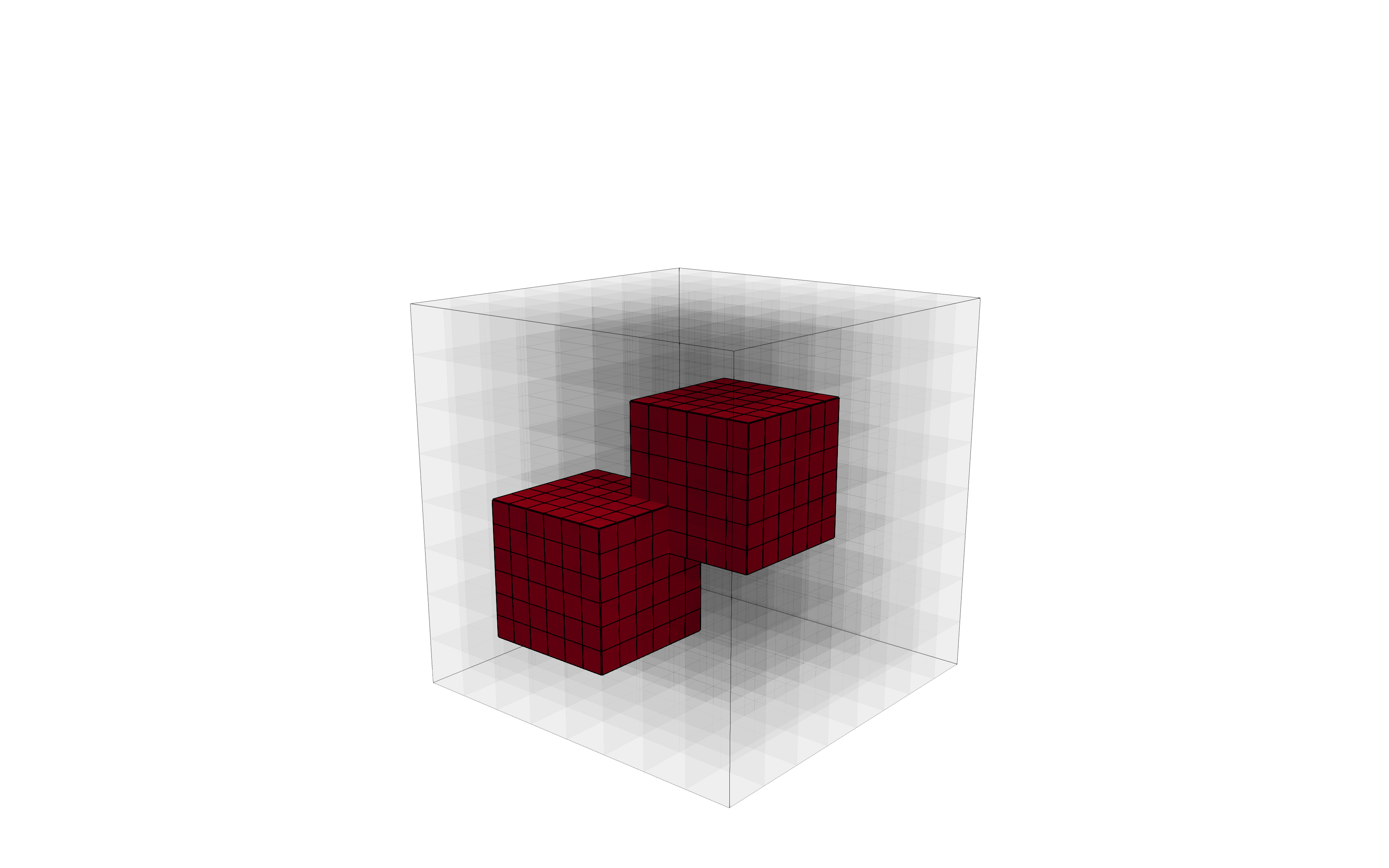}
		\caption{$p=3$}
	\end{subfigure}
	\;
	\begin{subfigure}{0.31\textwidth}
		\includegraphics[width=1\textwidth,trim=18cm 2.5cm 17cm 11cm,clip]{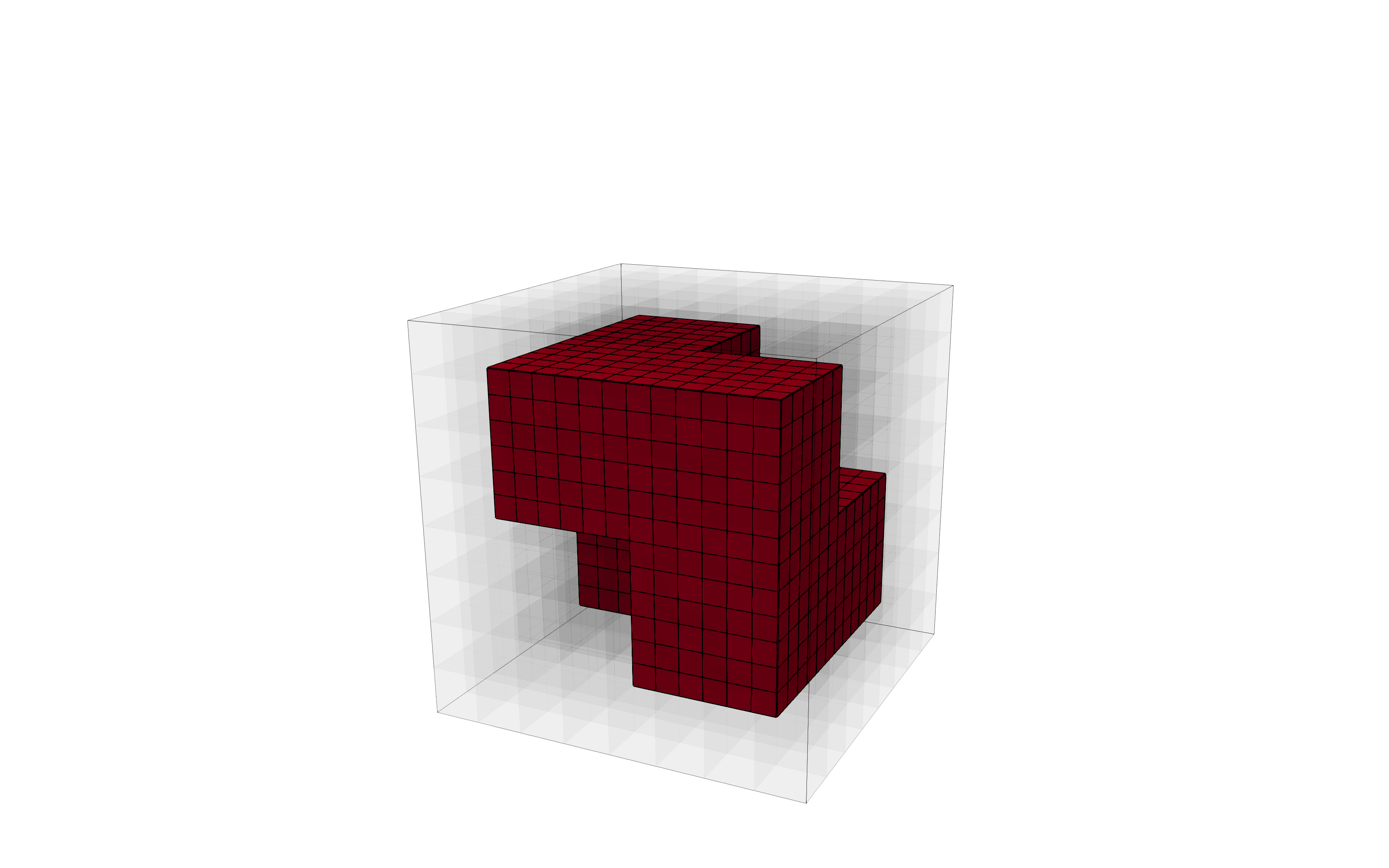}
		\caption{$p=3$}
	\end{subfigure}
	\;
	\begin{subfigure}{0.31\textwidth}
		\includegraphics[width=1\textwidth,trim=18cm 2.5cm 17cm 11cm,clip]{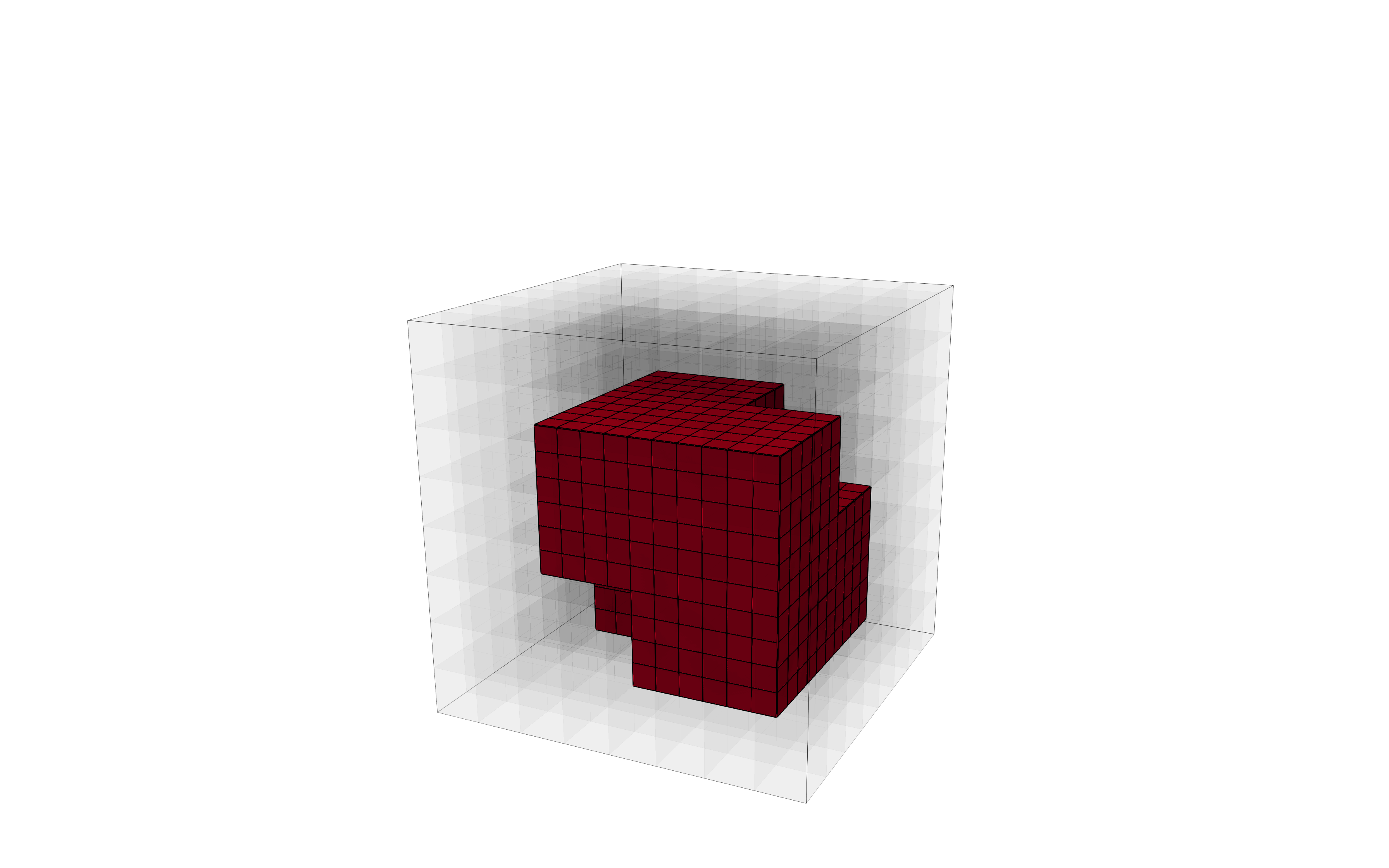}
		\caption{$p=3$}
	\end{subfigure}
	\caption{Refinement patterns that are not supported based on the proposed local exactness condition, but that are exact, are shown above. 
	In the first row, there is no shortest path of translation operators going between the two refined $0$-form basis functions: to meet the proposed criteria, additional refinement would need to be made, e.g., as in the first row of Figure \ref{fig:exact_configs_supported}.
	In the second row, refinements are made using the support of three-form basis functions, rather than $0$-form basis functions.}
	\label{fig:exact_configs_not_supported}
\end{figure}
\section{Conclusions}\label{sec:conclusions}

The incorporation of smooth, locally-refinable splines within the framework of \emph{finite element exterior calculus} can help build stable, accurate and efficient numerical methods.
Motivated by this, we have presented a theoretical analysis of a discrete de Rham complex built using hierarchical B-splines on a hypercube $\Omega \subset \RR^\ndim$.
In particular, we have presented locally-verifiable conditions that are sufficient for ensuring exactness of this discrete de Rham complex.
These theoretical results are accompanied by numerical tests to showcase their applicability.
These numerical tests help us investigate the different refinement patterns allowed by our assumptions and, in the special case of $n = 2$, allow us to contrast our approach with the one from \cite{Evans19}.
We find that our approach is applicable to certain refinement patterns disallowed by \cite{Evans19} and vice versa.
Future work should aim to unify the strengths of each approach.
There are other promising lines of theoretical and applied research that can follow this manuscript, and in the following we briefly introduce some extensions of interest.

For applications, perhaps the most obvious case of interest is the use of our results for building stable adaptive numerical methods for problems in electromagnetism and fluid mechanics on domains in $\RR^3$.
A first interesting open problem here is, given meshes that do not satisfy Assumption \ref{assume:shortest_path}, how to transform them into meshes that do by adding minimal number of degrees of freedom.
Another practical extension is the incorporation of boundary conditions other than Dirichlet boundary conditions.
Alternatively, for solving problems on an arbitrary domain $\Omega$, it will be interesting to build a discrete complex by performing a cuboidal decomposition of $\Omega$ and piecing together the discrete spline complexes defined on each cuboidal subdomain.

There are also several theoretical follow-ups that are of interest in numerical analysis.
Since exactness of the discrete complex is only one of the ingredients in the construction of stable numerical methods, an interesting question is whether there exist commuting projection operators from the continuous complex to the exact discrete subcomplexes that we consider in this document.
Similarly, constructive methods for building local commuting projection operators is of interest for both theoretical and applied studies.

\section*{Acknowledgements}
K. Shepherd was partially supported by the National Science Foundation Graduate Research Fellowship under Grant No. DGE-1610403.
Any opinion, findings, and conclusions or recommendations expressed in this material are those of the authors and do not necessarily reflect the views of the National Science Foundation.
The research of Deepesh Toshniwal is supported by project number 212.150 awarded through the Veni research programme by the Dutch Research Council (NWO).

\bibliographystyle{plain}
\bibliography{Combined_Bibliography.bib}

\def\cprime{$'$}
\begin{thebibliography}{10}

\bibitem{Arnold:2018}
D.~N. Arnold.
\newblock {\em Finite Element Exterior Calculus}.
\newblock CBMS-NSF regional conference series in applied mathematics. Society
  for Industrial and Applied Mathematics, Philadelphia, PA, 2018.

\bibitem{AFW06}
D.N. Arnold, R.S. Falk, and R.~Winther.
\newblock Finite element exterior calculus, homological techniques, and
  applications.
\newblock {\em Acta Numer.}, 15:1--155, 2006.

\bibitem{AFW-2}
D.N. Arnold, R.S. Falk, and R.~Winther.
\newblock Finite element exterior calculus: from {H}odge theory to numerical
  stability.
\newblock {\em Bull. Amer. Math. Soc. (N.S.)}, 47(2):281--354, 2010.

\bibitem{bazilevs2012isogeometric}
Y~Bazilevs, M-C Hsu, and MA~Scott.
\newblock Isogeometric fluid--structure interaction analysis with emphasis on
  non-matching discretizations, and with application to wind turbines.
\newblock {\em Comput. Methods Appl. Mech. Engrg.}, 249:28--41, 2012.

\bibitem{Bressan:2019}
Andrea Bressan and Espen Sande.
\newblock Approximation in {FEM}, {DG} and {IGA}: a theoretical comparison.
\newblock {\em Numerische Mathematik}, 143(4):923--942, 2019.

\bibitem{BRSV11}
A.~Buffa, J.~Rivas, G.~Sangalli, and R.~V\'{a}zquez.
\newblock Isogeometric discrete differential forms in three dimensions.
\newblock {\em SIAM J. Numer. Anal.}, 49(2):818--844, 2011.

\bibitem{Buffa_Sangalli_Vazquez}
A.~Buffa, G.~Sangalli, and R.~V{\'a}zquez.
\newblock Isogeometric analysis in electromagnetics: B-splines approximation.
\newblock {\em Comput. Methods Appl. Mech. Engrg.}, 199(17-20):1143 -- 1152,
  2010.

\bibitem{BSV14}
A.~Buffa, G.~Sangalli, and R.~V\'azquez.
\newblock Isogeometric methods for computational electromagnetics: {B}-spline
  and {T}-spline discretizations.
\newblock {\em J. Comput. Phys.}, 257, Part B:1291 -- 1320, 2014.

\bibitem{DeBoor}
C.~de~Boor.
\newblock {\em {A Practical Guide to Splines}}, volume~27 of {\em Applied
  Mathematical Sciences}.
\newblock Springer-Verlag, New York, revised edition, 2001.

\bibitem{Evans_Bazilevs_Babuska_Hughes}
J.~A. Evans, Y.~Bazilevs, I.~Babu\v{s}ka, and T.~J.~R. Hughes.
\newblock $n$-widths, sup-infs, and optimality ratios for the $k$-version of
  the isogeometic finite element method.
\newblock {\em Comput. Methods Appl. Mech. Engrg.}, 198:1726--1741, 2009.

\bibitem{EvHu12}
J.A. Evans and T.J.R. Hughes.
\newblock Isogeometric divergence-conforming {B}-splines for the
  {D}arcy-{S}tokes-{B}rinkman equations.
\newblock {\em Math. Models Methods Appl. Sci.}, 23(04):671--741, 2013.

\bibitem{Evans19}
John~A. Evans, Michael~A. Scott, Kendrick~M. Shepherd, Derek~C. Thomas, and
  Rafael V\'azquez.
\newblock Hierarchical {B}-spline complexes of discrete differential forms.
\newblock {\em IMA J. Numer. Anal.}, 39:preprint, 2019.

\bibitem{Farin:2002}
Gerald~E Farin and Gerald Farin.
\newblock {\em Curves and surfaces for CAGD: a practical guide}.
\newblock Morgan Kaufmann, 2002.

\bibitem{Ferus:2008}
D.~Ferus.
\newblock Analysis {III}: Wintersemester 2007/8, 2008.
\newblock Course lecture notes.

\bibitem{Hatcher}
A.~Hatcher.
\newblock {\em Algebraic topology}.
\newblock Cambridge University Press, Cambridge, 2002.

\bibitem{Hughes:2005}
T.~J.~R. Hughes, J.~A. Cottrell, and Y.~Bazilevs.
\newblock Isogeometric analysis: {CAD}, finite elements, {NURBS}, exact
  geometry and mesh refinement.
\newblock {\em Comput. Methods Appl. Mech. Engrg.}, 194(39-41):4135--4195,
  2005.

\bibitem{Hughes_book}
Thomas J.~R. Hughes.
\newblock {\em The finite element method}.
\newblock Prentice Hall Inc., Englewood Cliffs, NJ, 1987.

\bibitem{Johannessen15}
K.A. Johannessen, M.~Kumar, and T.~Kvamsdal.
\newblock {Divergence-conforming discretization for Stokes problem on locally
  refined meshes using LR B-splines}.
\newblock {\em Comput. Methods in Appl. Mech. Engrg.}, 293:38--70, 2015.

\bibitem{Kamensky17}
D.~Kamensky, M.-C. Hsu, Y.~Yu, J.A. Evans, M.S. Sacks, and T.J.R. Hughes.
\newblock {Immersogeometric cardiovascular fluid-structure interaction analysis
  with divergence-conforming B-splines}.
\newblock {\em Comput. Methods in Appl. Mech. Engrg.}, 314:408--472, 2017.

\bibitem{Kraft}
R.~Kraft.
\newblock Adaptive and linearly independent multilevel {$B$}-splines.
\newblock In {\em Surface Fitting and Multiresolution Methods
  ({C}hamonix--{M}ont-{B}lanc, 1996)}, pages 209--218. Vanderbilt Univ. Press,
  Nashville, TN, 1997.

\bibitem{perse2021geometric}
Benedikt Perse, Katharina Kormann, and Eric Sonnendr\"ucker.
\newblock Geometric particle-in-cell simulations of the vlasov--maxwell system
  in curvilinear coordinates.
\newblock {\em SIAM Journal on Scientific Computing}, 43(1):B194--B218, 2021.

\bibitem{Sande:2019}
Espen Sande, Carla Manni, and Hendrik Speleers.
\newblock Sharp error estimates for spline approximation: Explicit constants,
  n-widths, and eigenfunction convergence.
\newblock {\em Mathematical Models and Methods in Applied Sciences},
  29(06):1175--1205, 2019.

\bibitem{Schumaker:2007}
Larry Schumaker.
\newblock {\em Spline functions: basic theory}.
\newblock Cambridge University Press, 2007.

\bibitem{Spivak:1995}
M.~Spivak.
\newblock {\em Calculus on Manifolds}.
\newblock Addison-Wesley, New York, NY, 1995.

\bibitem{toshniwal2021isogeometric}
Deepesh Toshniwal and Thomas~JR Hughes.
\newblock Isogeometric discrete differential forms: Non-uniform degrees,
  b{\'e}zier extraction, polar splines and flows on surfaces.
\newblock {\em Computer Methods in Applied Mechanics and Engineering},
  376:113576, 2021.

\bibitem{Vazquez:2016}
R.~V\'azquez.
\newblock A new design for the implementation of isogeometric analysis in
  {O}ctave and {M}atlab: {G}eo{PDE}s 3.0.
\newblock {\em Comput. Math. Appl.}, 72(3):523--554, 2016.

\bibitem{Vuong_giannelli_juttler_simeon}
A.-V. Vuong, C.~Giannelli, B.~J\"uttler, and B.~Simeon.
\newblock A hierarchical approach to adaptive local refinement in isogeometric
  analysis.
\newblock {\em Comput. Methods Appl. Mech. Engrg.}, 200(49-52):3554--3567,
  2011.

\end{thebibliography}

\end{document}